\documentclass[11pt,a4paper]{article}

\usepackage{caption2}
\usepackage{CJK}
\usepackage{tabls}
\usepackage{bm}
\usepackage{multirow}
\usepackage{amssymb}
\usepackage{amsfonts}
\usepackage{mathrsfs}
\usepackage{empheq}
\usepackage{amsmath}
\allowdisplaybreaks[4]
\usepackage{amsthm}
\usepackage{graphicx}
\usepackage{color}
\usepackage{indentfirst}
\usepackage{hyperref}
\usepackage{float}
\usepackage{cases}
\usepackage[titletoc]{appendix}
\usepackage{bm}
\usepackage[colorinlistoftodos,english]{todonotes}

\allowdisplaybreaks%

\def\i1n{i=1,\cdots,n}
\def\j1n{j=1,\cdots,n}
\def\ij1n{i,j=1,\cdots,n}

\def\R{\mathbb R}
\def\N{\mathbb N}

\def \i{\mathrm i}

\numberwithin{equation}{section}
\theoremstyle{definition}

\newtheorem{thm}{\indent Theorem}[section]

\newtheorem{lem}{\indent Lemma}[section]
\newtheorem{prop}{\indent Proposition}[section]
\newtheorem{defn}{\indent Definition}[section]
\newtheorem{rem}{\indent Remark}[section]
\newtheorem{theorem}{Theorem}[section]
\newtheorem{lemma}[theorem]{Lemma}
\newtheorem{corollary}[theorem]{Corollary}
\newtheorem{proposition}[theorem]{Proposition}

\newtheorem{remark}[theorem]{Remark}
\newtheorem{assumption}[theorem]{Assumption}
\newtheorem{claim}[theorem]{Claim}

\theoremstyle{definition}

\newcommand{\dd}{{\mathrm d}}
\newcommand{\be}{\begin{equation}}
	\newcommand{\ee}{\end{equation}}
\newcommand{\beq}{\begin{equation*}}
	\newcommand{\eeq}{\end{equation*}}

\begin{document}
	
	\title{Null controllability of damped nonlinear wave equation}
	

	\author{Yan Cui\thanks{Department of Mathematics, Jinan University, Guangzhou, P. R. China (cuiy32@jnu.edu.cn).}
		\and Peng Lu\thanks{Department of Mathematics, Zhejiang Sci-Tech University, Hangzhou, P. R. China (plu25@zstu.edu.cn).}
		\and Yi Zhou\thanks{School of Mathematics Science, Fudan University, Shanghai, P. R. China (yizhou@fudan.edu.cn).}}
	
	\date{}
	
	\pagestyle{myheadings} \markboth{Null controllability of damped nonlinear wave equation}{}
	\maketitle
	
	\begin{abstract}
		In this paper, we investigate the null controllability of nonlinear wave systems. Initially, we employ a combination of the Galerkin method and a fixed point theorem to establish the null controllability for semi-linear wave equations with nonlinear functions that are dependent on velocities, under the geometric control condition. Subsequently, utilizing a novel iterative method, we demonstrate the null controllability for a class of quasi-linear wave systems in a constructive manner. Lastly, we present a control result for a class of fully nonlinear wave systems, serving as an  application.
		
	\end{abstract}
	
	

	\section{Introduction and main results}\label{intro}

	Assuming $T > 0$, we consider $\Omega\subset\mathbb{R}^n$, an open and bounded domain with a smooth boundary $\partial\Omega$. Here, $\omega$ is an open non-empty subset of $\Omega$. The characteristic function of $\omega$ is denoted by $\chi_\omega$.
	
	In this paper, our focus lies on the internal controllability issue pertaining to the subsequent nonlinear wave system:
	\begin{equation}\label{Wave Eqn with internal control}
		\left\{
		\begin{aligned}
			& y_{tt}-\Delta y+f(t,x,y,y_t,\nabla y,\nabla^2 y)=\chi_\omega(x)u(t,x), & (t,x)\in & ~ (0,T)\times\Omega, \\
			& y(t,x)=0, & (t,x)\in & ~ (0,T)\times\partial\Omega, \\
			& y(0,x)=y^0, ~ y_t(0,x)=y^1, & x\in & ~\Omega.
		\end{aligned}
		\right.
	\end{equation}
	
	{Here, $u$ represents the control (or input), $(y^0,y^1)$ is the initial data, and the nonlinear function  $f$ will be considered in several cases later.
		
		Our goal in this paper is to investigate the internal controllability problem when the nonlinear term $f$ meets specific criteria: for a given $T>0$, and given $(y^0,y^1), (y^0,y^1)$ within certain functional spaces, we aim to determine whether there exists a control such that the solution $y$ of \eqref{Wave Eqn with internal control} with initial data $(y^0,y^1)$ fulfills the condition $(y(T),y_t(T)) = (y^0,y^1)$?
		
		The issue of controllability for wave equations is steeped in a rich history. D. Russell \cite{Russell78} and J. L. Lions \cite{Lions88} laid the groundwork by establishing the duality principle, which reveals that the exact controllability of the control system is intrinsically linked to the observability inequality of the adjoint system.  C. Bardos, G. Lebeau, and J. Rauch \cite{BardosLebeauRauch92} highlighted that the geometric control condition (GCC) is crucial for the controllability of scalar wave equations. We roughly state that a subdomain $ \omega\subset \Omega$ and a time $T>0$ satisfy  GCC if and only if every general bicharacteristic intersects the set $(0,T)\times \omega $. For further details, we refer the reader to \cite{BurqGerardCNS,RauchTaylordecay,LaurentLeautaud16}.
	}
	\subsection{Semi-linear case }
	{Numerous studies have investigated scenarios where nonlinearity is expressed as $ f = f(u) $. E. Zuazua \cite{zuazua} has demonstrated the exact controllability of semi-linear wave equations by employing a blend of the Hilbert Uniqueness Method (HUM) and Schauder's fixed point theorem, provided that the nonlinearity $f$ exhibits Lipschitz continuity. I. Lasiecka and R. Triggiani \cite{LasieckaTriggiani91} expanded upon this foundational work by applying a global inversion theorem, which allowed for the inclusion of nonlinearities $f$ that are absolutely continuous with a first derivative $f'$ that is almost everywhere uniformly bounded. E. Zuazua \cite{Zuazua1} delved into additional cases where the nonlinearity exhibits logarithmic growth, characterized by $ f(u) \sim u \ln^p(u) $, specifically within the context of one spatial dimension. X. Fu, J. Yong, and X. Zhang \cite{Fu3} subsequently overcame the dimensional constraints of these findings by extending the results to higher-dimensional spaces. Their approach is grounded in the application of fixed point theorems, which enables the reduction of exact controllability to obtain global Carleman estimates for the linearized wave equation with a potential, as detailed in \cite{DZZ08}.
		
		In a more recent contribution, A. Munch and E. Trelat \cite{Trlat22} have provided constructive proof for the results initially presented in \cite{Zuazua1}. Their approach involves the design of a least-squares algorithm, which is adept at yielding both the control inputs and the corresponding solutions for one-dimensional semi-linear wave equations. When the nonlinearity is of power-type, $ f(u) = |u|^{p-1}u $ with $ 1 \leq p < 5 $, B. Dehman, G. Lebeau, and E. Zuazua \cite{DLZstabNLW} have demonstrated the exact controllability, assuming that the control is exerted on a subdomain situated exterior to a spherical boundary, thereby truncating the nonlinear effects. This framework has been generalized in \cite{DL09} to encompass the Geometric Control Condition (GCC) and to accommodate nonlinearities without the need for truncation, albeit with the stipulation that the lower frequency components of the initial data must be sufficiently diminutive. The critical case, where $ p = 5 $, has been addressed by C. Laurent \cite{Laurent} through the application of profile decomposition techniques on compact Riemannian manifolds. For cases with more general structures of nonlinear terms, the reader is referred to \cite{ZL10}; further details are provided in \cite{Hongheng} and \cite{Zhang}.
		
		When a system includes a term of the form $y_t$, it is typically understood to exhibit damping or anti-damping characteristics. For example, boundary feedback damping of this nature can be utilized to show, through Huygens's principle, that linear wave equations in odd dimensions achieve null controllability under damping within a finite time (refer to \cite{cirina,ZL}). This property is also known as rapid stabilization (see \cite{Coronlivre}).
		
		Moreover, closely linked to the controllability issue, there is a wealth of research on the stabilizability of such systems. For more information, we suggest consulting \cite{BardosLebeauRauch92, DLZstabNLW}, and for advancements on more general systems, we refer to the works of M. M. Cavalcanti, V. N. D. Cavalcanti, R. Fukuoka, and J. A. Soriano, as detailed in \cite{Cavalcanti} and \cite{Cavalcanti2}.
		
		In the context where the nonlinearity is defined as $ f = f(y_t) $, X. Zhang \cite{Zhang} has proposed an open problem: whether the following type of semilinear system
		\begin{equation}\label{semilinear-intro}
			y_{tt}+\mathcal{A}y+f(y_t)=\chi_\omega(x)u(t,x)
		\end{equation}
		is exactly controllable in the energy space, even though the nonlinearity $ f $ is globally Lipschitz continuous.
		
		To our knowledge, there are fewer results regarding this problem. In the first part of this paper, we address this problem, attempting to solve it with some additional assumptions.
	}
	
	{\color{black}For simplicity of notation, we denote $H^s=H^s(\Omega), H^0=L^2(\Omega)$ and we define (see \cite{DL09})
		\begin{align}
			\mathcal{H}^s=\Big\{v\in H^s\Big| \Delta^{i} v|_{\partial\Omega}=0,  i=0,1,..., \Big\lfloor\frac{s}{2}-\frac{1}{4}\Big\rfloor\Big\},
		\end{align}
		where $ \lfloor\cdot\rfloor$ stands for floor function: For any $x\in\mathbb{R}$,
		\begin{align}
			\lfloor x\rfloor:=\max\{y\in\mathbb{Z}: ~ y\leq x\}.
	\end{align}}
	
	Our first goal is to study the null controllability  of the following system:
	\begin{equation}\label{semilinear damped wave control}
		\left\{
		\begin{aligned}
			& y_{tt}-\Delta y+f(y_t)=\chi_\omega(x)u, & (t,x)\in & ~ (0,T)\times\Omega, \\
			& y(t,x)=0, & (t,x)\in & ~ (0,T)\times\partial\Omega, \\
			& y(0,x)=y^0, ~ y_t(0,x)=y^1, & x\in & ~ \Omega,
		\end{aligned}
		\right.
	\end{equation}
	where $\omega\subset\Omega$, $\chi_\omega\in C^2(\overline{\Omega})$ satisfies $0\leq\chi_\omega(x)\leq 1$, $\chi_\omega|_{\omega}\equiv 1$, and $\chi_\omega$ supports in a neighbourhood of $\omega$. Let $ f : \mathbb{R} \to \mathbb{R} $ be a nonlinearity satisfying $ f(0) = 0 $ and assume that $ f $ is Lipschitz continuous. That is, there exist constants $ L > \tilde{L} > 0 $ such that the following conditions hold:
	\begin{enumerate}
		\item Lipschitz Continuity: For all $ a, b \in \mathbb{R} $, the function $ f $ satisfies the inequality
		\begin{equation}\label{Lipschitz}
			\big| f(a) - f(b) \big| \leq L |a - b|.
		\end{equation}
		
		\item Monotonicity Condition: Additionally, for all $ a, b \in \mathbb{R} $ with $ a \neq b $, it is required that
		\begin{equation}\label{tilde L condition}
			(a - b) \big( f(a) - f(b) \big) \geq \tilde{L} (a - b)^2.
		\end{equation}
		
	\end{enumerate}

	Our primary result is as follows.
	
	\begin{thm}\label{thm:semilinear}
		Suppose that $(T,\omega)$ fulfills the Geometric Control Condition (GCC). Then there exists a constant $D>0$ such that if $f$ satisfies \eqref{Lipschitz}--\eqref{tilde L condition} and
		\begin{equation}\label{condition of D L and tilde L}
			\left(\frac{L}{\tilde{L}}-1\right)^2<\frac{L}{2D},
		\end{equation}
		then for any $(y^0,y^1)\in \mathcal{H}^2\times \mathcal{H}^1$, there exists a control function $u\in L^2(0,T; H^1(\omega))$ that ensures
		\begin{equation}\label{estimate of similinear control function}
			\begin{aligned}
				& \int_0^T\int_\omega|\nabla u|^2\dd x\dd t+\int_0^T\int_\omega|u|^2\dd x\dd t \\
				\leq & ~ D^*\left(\int_\Omega(|y^1|^2+|\nabla y^0|^2)\dd x+\int_\Omega(|\nabla y^1|^2+|\Delta y^0|^2)\dd x\right)
			\end{aligned}
		\end{equation}
		for some $D^*>0$. Additionally, the corresponding solution $(y,y_t)$ to \eqref{semilinear damped wave control} with initial data $(y^0,y^1)$ satisfies
		\begin{equation}
			y(T)=0, \quad y_t(T)=0.
		\end{equation}
	\end{thm}

		
		\begin{remark}
			
			\begin{itemize}
				\item $D$ comes from observability inequality in Lemma \ref{lem:the example prop}.
				\item {\color{black}When $(\omega,T)$ satisfies GCC, for any fixed $L>0$,  \eqref{condition of D L and tilde L} can be rewritten as:
					\begin{align}\label{restrict:L}
						\frac{L}{1+\sqrt{\frac{L}{2D}}}<\tilde{L}<L.
					\end{align}
					
					Since  $D$ would be of form $e^{CL}$ for some constant $C$ (combing a time transformation and \cite[Theorem 1.5]{LaurentLeautaud16}),  \eqref{restrict:L} is an explicit lower bound for $\tilde{L}$.  However, when $L$ is large enough, $\tilde{L}$ is a small perturbation of $L$. So we expect that \eqref{restrict:L} can be improved by other types of geometric conditions.}
				\item $D^*$ in \eqref{estimate of similinear control function} can be given explicitly in terms of $D,L,\tilde{L}$ and $\chi$. Actually, $D^*=\frac{C^*}{\delta}, $ $C^*$ is given by \eqref{C*} and $\delta$ is given by \eqref{delta-similinear}.
			\end{itemize}
		\end{remark}
		
		\begin{remark}
			The proof relies heavily on the specific damping structure, allowing us to employ the Galerkin method and a fixed point argument as discussed in L. C. Evans \cite{Evans}. It might be expected that this approach could also be applicable to other types of damping within the wave system, even with varying boundary conditions.
		\end{remark}
		
		\begin{remark}
			Note that in Theorem \ref{thm:semilinear}, the time and domain of control are assumed to satisfy the GCC, which is necessary  when $f$ is linear (\cite{Bardos,BurqGerardCNS,RauchTaylordecay}).
		\end{remark}
		
		\begin{remark}
			This result partially solves the problem posed by Xu Zhang in \cite[Remark 7.2]{Zhang}.  The initial data here are assumed in  $\mathcal{H}^2\times \mathcal{H}^1$. It is still not known whether Theorem \ref{thm:semilinear} holds in general for any initial data in energy space  $H_0^1\times L^2$.
		\end{remark}
		We outline the proof as follows.
		We expand any given initial value $(y^0 ,y^1)\in \mathcal{H}^2\times \mathcal{H}^1 $
		of system \eqref{semilinear damped wave control} as follows:
		\begin{equation}\label{intro:yinitial}
			y^0=\sum\limits_{j=1}^\infty(y^0,\varphi_j)_{L^2}\varphi_j, \quad y^1=\sum\limits_{j=1}^\infty(y^1,\varphi_j)_{L^2}\varphi_j.
		\end{equation}
		where $\{\varphi_j\}_{j=1}^\infty$  is a sequence of orthonormal bases in
	$L^2$
		space, satisfying the elliptic eigenvalue problem.
		Next, we define the finite energy elements $(y_N^0,y_N^1)$ as follows:
		\begin{equation}\label{intro:yninitial}
			y_{N}^{0}=\sum\limits_{j=1}^N(y^0,\varphi_j)_{L^2}\varphi_j, \quad y_{N}^{1}=\sum\limits_{j=1}^N(y^1,\varphi_j)_{L^2}\varphi_j.
		\end{equation}
	Then, let
			\begin{equation}\label{yvN-}
			y_N=\sum\limits_{j=1}^Ng_{jN}(t)\varphi_j, \quad v_N=\sum\limits_{j=1}^Nh_{jN}(t)\varphi_j,
		\end{equation}
		we consider the following coupled finite-dimensional system of ordinary differential equations:
		which solves the finite-dimensional system
		\begin{equation}\label{intro:finite dimision yN-example}
			\left\{
			\begin{aligned}
				& \Big(\partial_t^2y_N-\Delta y_N+2\partial_ty_N-\chi_\omega\partial_tv_N,\varphi_i\Big)_{L^2}=0, \quad i=1,2,\cdots,N \\
				& t=0: g_{jN}=(y^0,\varphi_j)_{L^2}, ~ g'_{jN}=(y^1,\varphi_j)_{L^2} \\
			\end{aligned}
			\right.
		\end{equation}
		and the backward system
		\begin{equation}\label{intro:finite dimision vN-example}
			\left\{
			\begin{aligned}
				& \Big(\partial_t^2v_N-\Delta v_N-2\partial_tv_N,\varphi_i\Big)_{L^2}=0, \quad i=1,2,\cdots,N \\
				& t=T: h_{jN}=a_j, ~ h'_{jN}=b_j, \\
			\end{aligned}
			\right.
		\end{equation}
			where $(\vec{a}_N,\vec{b}_N)=(a_1,\cdots,a_N,b_1,\cdots,b_N)\in \mathbb{R}^{2N}$ with
		\begin{equation}\label{abb-}
			\sum_{i=1}^N \left(|\lambda_i|^4|a_i|^2+|\lambda_i|^2|b_i|^2\right)=\|(v_N(T),\partial_tv_{N}(T))\|^2_{\mathcal{H}^2\times \mathcal{H}^1}<\infty.
		\end{equation}
	    We then prove the conclusions of our theorem in two steps:
	    \begin{enumerate}
	    	\item [(1)] 	
	    There exists a time $T>0$, for any $N$, we prove that there exist
	   $(\vec{a}_N,\vec{b}_N)$ satisfying \eqref{abb} such that the system \eqref{intro:finite dimision yN-example}--\eqref{intro:finite dimision vN-example} has a unique solution $y_N, v_N \in C^0(0,T;\mathcal{H}^2)\cap C^1(0,T;\mathcal{H}^1)\cap C^2(0,T;L^2)$ satisfying $$\|y_N\|_{C^i(0,T;\mathcal{H}^{2-i})}\leq C\|(y^0,y^1)\|_{\mathcal{H}^2\times \mathcal{H}^1}, \quad i=0,1,2,$$ and
		$$\|v_N\|_{C^i(0,T;\mathcal{H}^{2-i})}\leq C\|(y^0,y^1)\|_{\mathcal{H}^2\times \mathcal{H}^1}, \quad i=0,1,2.$$
		Here $C$ is a positive constant independent of $N$.
		Furthermore, $y_N$ satisfies
		$$(y_N(T),\partial_ty_{N}(T))=(0,0).$$
		
	\item 	[(2)]Based on the above norm control, we can employ a compactness argument to obtain the following convergence results: There exist functions $y,v \in C^0(0,T;\mathcal{H}^2)\cap C^1(0,T;\mathcal{H}^1)\cap C^2(0,T;L^2)$ such that the sequences  $(y_N,\partial_ty_{N})$ (resp. $(v_N,\partial_tv_{N})$ )  in $C(0,T;\mathcal{H}^2)\times C(0,T;\mathcal{H}^1)$ converge weakly to $(y,y_t)$ (resp. $(v,v_{t})$ ) in $C(0,T;\mathcal{H}^2)\times C(0,T;\mathcal{H}^1)$, and strongly in $C(0,T;\mathcal{H}^1)\times C(0,T;L^2)$. Since  $(y_N,v_N)$ solves system \eqref{intro:finite dimision yN-example}--\eqref{intro:finite dimision vN-example}, the limit $(y,v)$  satisfies the equation in the sense of  $L^2$ and due to the convergence of the initial and terminal values of
		 $(y_N,y_{Nt})$:
		$$(y_N(0),\partial_ty_{N}(0))\rightarrow (y^0,y^1), \text{~as~} N\rightarrow \infty,$$
		and
		$$(y_N(T),\partial_ty_{N}(T))= (0,0),$$
	we have $$(y(0),y_t(0))=(y^0,y^1), ~ (y(T),y_{t}(T))=(0,0).$$
	 \end{enumerate}
	
		
		The first step relies on a novel application of a zero-point lemma, which is essentially a variant of Brouwer's fixed-point theorem. We  construct a sequence of vector maps
		 $\mathcal{F}_N: \mathbb{R}^{2N}\to \mathbb{R}^{2N}$
		 defined by
		\begin{equation}
			\mathcal{F}_N:\big(a_1,\cdots,a_N,b_1,\cdots,b_N\big)^\top\mapsto \Lambda_N\big(g_{1N}(T),\cdots,g_{NN}(T),g'_{1N}(T),\cdots,g'_{NN}(T)\big)^\top,
		\end{equation}
		which maps the initial values of the finite-dimensional dual system to the terminal values of the target system.
		Here $\Lambda_N=\text{diag}(\lambda_1^2,\cdots,\lambda_N^2,\lambda_1,\cdots,\lambda_N)  \in \R^{2N\times 2N}$. We have an equivalent $\ell_2$ norm given by:
		\begin{equation*}\begin{split}
			(\mathcal{F}_N(x_N), x_N)_{\tilde{\ell_2}_\delta}=~&\frac{1}{\delta}\int_\Omega\Big(\partial_ty_N(T)\partial_tv_N(T)+\nabla y_N(T)\cdot\nabla v_N(T)\Big)\dd x		 \\& +\int_\Omega\Big(\nabla\partial_ty_N(T)\cdot\nabla\partial_tv_N(T)+\Delta y_N(T)\Delta v_N(T)\Big)\dd x.
			\end{split}		
		\end{equation*}
		  This transformation reduces the problem to determining the existence of zeros for the functions $\mathcal{F}_N$.  We demonstrate that the function $\mathcal{F}_N$ satisfies
		  \begin{align}
		  	& (\mathcal{F}_N(x_N), x_N)_{\tilde{\ell_2}_\delta} \nonumber \\
		  	\geq & ~ \left(\frac{1}{2D}-\frac{L-\tilde{L}}{\tilde{L}\sqrt{2DL}}\right)E_2(v_N(T))+\frac{1}{\delta}\left(\frac{1}{2D}-\frac{L-\tilde{L}}{2\tilde{L}\sqrt{DL}}\right)E_1(v_N (T)) \nonumber \\
		  	& -\Big(\frac{\|\Delta\chi\|_{L^\infty}}{4L}+\frac{(L-\tilde{L})}{2L\tilde{L}}\sqrt{\frac{D}{2L}}\|\nabla\chi\|_{L^\infty}^2\Big)E_1(v_N(T)) \nonumber \\
		  	& -\frac{\tilde{L}}{\delta}\sqrt{\frac{D}{L}}\Big(\frac{L-\tilde{L}}{2\tilde{L}}+\frac{L}{L-\tilde{L}}\Big)E_1(y_N(0))-\tilde{L}\sqrt{\frac{D}{2L}}\Big(\frac{L-\tilde{L}}{2 \tilde{L}}+\frac{L}{L-\tilde{L}}\Big)E_2(y_N(0)). \nonumber
		  \end{align}	
		  where
		  	\begin{equation*}			
		  	E_1(u(t)):=\int_\Omega\big(|u_t(t)|^2+|\nabla u(t)|^2\big)\dd x,\qquad E_2(u(t)):=\int_\Omega\big(|\nabla u_t(t)|^2+|\Delta u(t)|^2\big)\dd x.
		  \end{equation*}
		 Then, applying the observation inequality of the linear system, we demonstrate that $$(\mathcal{F}_N(x_N), x_N)_{\tilde{\ell_2}_\delta}\geq 0,$$
		for $|x_N|=r$ on a specific sphere, where $r$ is a sufficiently large radius. Using Lemma \ref{lemma from Evans-example}, we establish the existence of zeros of $\mathcal{F}_N$. Furthermore, by utilizing energy estimates of the equations, we obtain  the uniform bound of the solutions for the finite-dimensional system. Consequently, the proof is completed. For detailed proof, one may refer to Section 4; alternatively, the methodological exposition provided in Section 2 for the linear system serves as an illustrative example.

		\subsection{Quasi-linear case}
		Many studies have also been done on the subject related to the exact controllability. In \cite{Li06}, by using a constructive method, T. Li and L. Yu obtained the exact boundary controllability for 1D quasi-linear wave system. We refer the reader to \cite{LiRao03,Li08} for a system theory of controllability for 1D quasi-linear hyperbolic system. It was generalized by the third author and Z. Lei to the two or three space dimensional case \cite{ZL}. Their proofs strongly relied on boundary damping and Huygens's principle. By using a different method based on Riemannian geometry, P. Yao \cite{Yao1} also obtained the exact boundary controllability for a class of quasi-linear wave in high space dimensional case. Let us mention that the above results concern boundary control problem. As far as we know, there are much fewer known results about internal controllability for quasi-linear case. K. Zhuang \cite{zhuang} studied the exact internal controllability for a class of 1D quasi-linear wave equation.

		When considering the internal energy controllability of nonlinear wave equations in higher dimensions, the boundary conditions are typically prescribed, which precludes the direct application of Huygens's principle to the linearized system. Our second contribution extends the work in \cite{ZL} by examining the internal null controllability of damped quasilinear wave equations. Let us consider a nonlinear term $ f $ that is defined as follows:
		\begin{equation}\label{cond:f1} f(t,x,y,y_t,\nabla y,\nabla^2 y) = -y_t+g_1(t,x,y,y_t,\nabla y) + \sum_{i,j=1}^{n} \frac{\partial}{\partial x_j} \left( g_2^{ij}(t,x,y,y_t,\nabla y) \frac{\partial y}{\partial x_i} \right), \end{equation}
		where $ g_1 $ and $ g_2^{ij} $, for $ i, j = 1, \ldots, n $, are smooth functions satisfying the following conditions:
		 \begin{equation}\label{g1} \begin{cases}  g_1(t,x,0,0,0) = 0, \\ g_1(t,x,y,y_t,\nabla y) = O(|y|^2 + |y_t|^2 + |\nabla y|^2) \quad \text{as} \quad (|y| + |y_t| + |\nabla y|) \to 0, \end{cases} \end{equation} and \begin{equation}\label{g2} \begin{cases}  g_2^{ij} = g_2^{ji},  g_2^{ij}(t,x,0,0,0) = 0, \\  g_2^{ij}(t,x,y,y_t,\nabla y) = O(|y| + |y_t| + |\nabla y|) \quad \text{as} \quad (|y| + |y_t| + |\nabla y|) \to 0. \end{cases} \end{equation}
		
		The notation $ \nabla $ denotes the gradient operator, $ \nabla^2 $ represents the Hessian matrix, and $ O $ denotes the Landau symbol, indicating the asymptotic behavior of the functions $ g_1 $ and $ g_2^{ij} $ as their arguments approach zero.
		
		{\color{black}In order to study the controllability of the system \eqref{Wave Eqn with internal control}, it is imperative to specify the suitable functional space in which the solution exists for some time interval $(0,T)$. As the well-posedness of the wave equation necessitates that the principal coefficient term adhere to certain regularity criteria, the functions $ g_1 $ and $ g_2^{ij} $, among others, must satisfy these conditions. Consequently, the initial conditions for the system must belong to the Sobolev space $ \mathcal{H}^s \times \mathcal{H}^{s-1} $, with $ s $ being sufficiently large. This, in turn, necessitates that the boundary conditions of the equation satisfy certain compatibility constraints.
			To facilitate the description of conditions, we rewrite the quasi-linear system \eqref{Wave Eqn with internal control} with the nonlinear term f satisfying \eqref{cond:f1} as follows:
			\begin{equation}\label{system:quasilinear0}
				\left\{
				\begin{aligned}
					& y_{tt}+b_0y_t-{\color{black}\sum\limits_{i,j=1}^n\big(a_{ij}y_{x_i}\big)_{x_j}}+\sum\limits_{k=1}^nb_ky_{x_k}+\tilde{b}y=\chi_\omega u, & (t,x)\in & ~ (0,T)\times\Omega, \\
					& y(t,x)=0, & (t,x)\in & ~ (0,T)\times\partial\Omega, \\
					& y(0,x)=y^0, ~ y_t(0,x)=y^1, & x\in & ~ \Omega,
				\end{aligned}
				\right.
			\end{equation}
			with
			\begin{equation} \label{coe:1}
				\begin{aligned}
					& {\color{black}a_{ij}=a_{ji}=\delta_{ij}-g^{ij}_2, \quad b_0=1+\int_0^1\frac{\partial g_1}{\partial y_t}(t,x, y,\tau y_t,\nabla y)\dd\tau,} \\
					& {\color{black}b_k=\int_0^1\frac{\partial g_1}{\partial y_{x_k}}(t,x, y, y_t, y_1,\cdots,\tau y_{x_k},\cdots,y_{x_n})\dd\tau, \quad \tilde{b}=\int_0^1\frac{\partial g_1}{\partial y}(t,x,\tau y, y_t,\nabla y)\dd\tau,}
				\end{aligned}
			\end{equation}
		where $\delta_{ij}$ is  Kronecker delta function.
		
			We impose the following boundary compatibility conditions:
			\begin{assumption}[$\mathcal{H}^s $-Boundary compatibility conditions]\label{assumption:boundary} Let $s\geq 2$. The smooth coefficients $a_{ij}, b_k $ for $ i,j=1,\cdots,n$ and $ k=0,1,\cdots,n$, as well as $\tilde{b}$ in System \eqref{Wave Eqn with internal control}, satisfy the following conditions for any $t\in [0,T]$ and
			 $ u,v\in \cap_{i=0}^2C^i(0,T;\mathcal{H}^{s-i}),  $
				\begin{equation}\label{cond:f2}
					\begin{cases}
					 \sum\limits_{i,j=1}^{n} \frac{\partial a_{ij}}{\partial x_j} (t,x,v,0,\nabla v) \frac{\partial u}{\partial x_i} \in C(0,T;\mathcal{H}^{s-2}), \\
					  \sum\limits_{i,j=1}^{n} a_{ij}(t,x,v,0,\nabla v) \frac{\partial^2 u}{\partial x_i \partial x_j } \in C(0,T;\mathcal{H}^{s-2}), \\
					   \sum\limits_{i,j=1}^{n}\partial_t a_{ij}(t,x,v,0,\nabla v) \frac{\partial^2 u}{\partial x_i \partial x_j }\in C(0,T;\mathcal{H}^{s-2}), \\
						 \sum\limits_{k=1}^{n}b_k(t,x, v,0,\nabla v)\frac{\partial u}{\partial x_k}
					  \in C(0,T;\mathcal{H}^{s-2}),  \\
					 b_0(t,x, v,0,\nabla v) \partial_tu, \quad \tilde{b}(t,x, v,0,\nabla v) u \in C(0,T;\mathcal{H}^{s-2}).
					\end{cases}				
				\end{equation}
	
			\end{assumption}
			
		}
		
		Before stating our main result, we introduce a  geometric condition on the pair $(\omega,T)$.
		{\color{black}\begin{assumption}\label{defi:weak-Gamma}

				Assume that there exists a point $ x_0 \notin \overline{\Omega} $ such that the following inequality is satisfied:
				\begin{equation}
					T > 1 + 100 (n + 2) \sqrt{n} \max_{x \in \overline{\Omega}} |x - x_0|,
				\end{equation}		
				and  \begin{equation}			
					\omega:=\Omega\cap O_{\varepsilon_0}(\Gamma_{\varepsilon_0})
				\end{equation}   for some $ \varepsilon_0 > 0 $,
				where
				$ \Gamma_{\varepsilon_0} $ is a subset of the boundary $ \partial \Omega $ defined by
				\begin{equation}
					\Gamma_{\varepsilon_0}:= \left\{ x \in \partial \Omega \middle| (x - x_0) \cdot \bm{\nu} > -\varepsilon_0 \right\},
				\end{equation} and
                \begin{equation}
                O_{\varepsilon_0}(\Gamma_{\varepsilon_0}):=\left\{ x \in \R^n: ~ d(x, \Gamma_{\varepsilon_0})<\varepsilon_0\right\},
                \end{equation}
                denotes neighborhood of $\Gamma_{\varepsilon_0}$ with a width of $\varepsilon_0$. Here, $ \bm{\nu}=(\nu^1,\cdots,\nu^n) $ denotes the unit outward normal vector to the boundary $ \partial \Omega $ of the domain $ \Omega $ and 		$$d(x, \Gamma_{\varepsilon_0}):=\inf\{|x-y|\big |y\in \Gamma_{\varepsilon_0}\}$$ means the distance between $x\in \R^n$ and  $\Gamma_{\varepsilon_0}$.
				
		\end{assumption}}
		
		\begin{thm}\label{thm:main1} Let $s\geq  \max\{n+2,4\}$ be an integer. Assume that Assumption \ref{assumption:boundary} on coefficients holds. Additionally,
			 assume that $(\omega,T)$ satisfy Assumption \ref{defi:weak-Gamma}.
			Then
			there exists a small positive constant $\varepsilon_{mthm}>0$, such that for any given initial data $(y^0,y^1)\in \mathcal{H}^s\times \mathcal{H}^{s-1}$, if the following norm condition is satisfied:
			\begin{equation}
				\big\|(y^0,y^1)\big\|_{\mathcal{H}^s\times \mathcal{H}^{s-1}}\leq \varepsilon_{mthm},
			\end{equation}
			then there exists a control $u\in L^\infty(0,T;\mathcal{H}^{s-1})$ and a constant $C_{uni}$ such that there exists a unique solution \begin{equation}y\in C(0,T;\mathcal{H}^{s})\cap C^1(0,T;\mathcal{H}^{s-1})\end{equation} of \eqref{system:quasilinear0} with internal control $u$, corresponding to the initial data $(y^0,y^1)$ and satisfying:
\begin{equation}\label{cond:y}
\big\|(y,y_t)\big\|_{\mathcal{H}^s\times \mathcal{H}^{s-1}}\leq C_{uni}\varepsilon_{mthm},
\end{equation}
and
\begin{equation}
				y(T,x)=0, ~ y_t(T,x)=0.
			\end{equation}
		\end{thm}
		
		Several remarks are given in order.

		\begin{remark}
			{\color{black}
			We first note that a smooth function  $g(t,x,y)$  satisfies
            $$ g= O(|y|) \quad \text{as} \quad |y|  \to 0,$$
            implies that there exist constants $C>0$, $\nu>0$, and a smooth function $\tilde{g}(t,x,y)$ such that  for $|y|\leq \nu$, we have $g=\tilde{g}y$ with the property that  $|\partial_t^i\partial_x^j\partial_{y}^k\tilde{g}|\leq C$, for all $ i,j,k\in \N$.
				
				Therefore, the functions with integral forms in compatibility condition \eqref{cond:f2} can be expressed in an alternative form, as detailed in Remark \ref{rem:alphabound}.
			
				Here we provide a few examples of nonlinearities $ f $ that adhere to the boundary compatibility condition \eqref{cond:f2}:
				\begin{itemize}
					\item 	Let $g_1(t,x,y,y_t,\nabla y)=a\chi_O y^2, g_2^{ij}(t,x,y,y_t,\nabla y)=\delta_{ij}-\chi_O by $ where $\chi_O$ is a cut-off function with compact support $O\subset \Omega\setminus \partial\Omega$ and $a,b\in \R$. Then
					\begin{equation}
						a_{ij}=\delta_{ij}-\chi_O by, ~ b_0=1, ~ b_k=0, ~ \tilde{b}=ay\chi_O .
					\end{equation}
					
					\item   Let $g_1(t,x,y,y_t,\nabla y)=c\chi_O y_t^2, g_2^{ij}(t,x,y,y_t,\nabla y)=\delta_{ij}-d\chi_O y_t $ for any $c,d\in \R$;
Then
					\begin{equation}
						a_{ij}=\delta_{ij}-\chi_O dy_t, ~ b_0=1, ~ b_k=0, ~ \tilde{b}=ay_t\chi_O.
					\end{equation}
					
				\end{itemize}
			}
		\end{remark}

		\begin{remark}Note that our argument is based on a  transformation that transmutes the original system into an analogous system incorporating damping terms. Consequently, this enables the construction of an algorithmic procedure that engenders sequences for both control inputs and solutions. By substantiating an observability inequality for a linearized system with coefficients that are time-space dependent, particularly within the system's principal component (as elaborated in Theorem \ref{fully nonlinear observability}), and subsequently applying the contraction mapping theorem, we deduce the convergence of the aforementioned sequences for control inputs and solutions.
		\end{remark}

		\begin{remark} Since the condition \eqref{cond:f1} is satisfied by the nonlinearity $ f(T - t, \cdot, \cdot, \cdot, \cdot) $ with the same validity as it is for $ f(t, \cdot, \cdot, \cdot, \cdot) $, the combination of Theorem \ref{thm:main1} and the well-posedness of the system governed by equation \eqref{Wave Eqn with internal control} enables us to demonstrate the exact controllability of the system delineated by equation \eqref{Wave Eqn with internal control}.
		\end{remark}

		\begin{remark}
			To establish the convergence of the solutions to the constructed linear system with respect to initial values, we assume that
			$ s \geq \max\{n+2,4\} $, as detailed in Proposition \ref{prop:vz} within Section \ref{sec:4}. Additionally, since our proof relies on higher-order space-time norm estimates, we also need to assume that
		$s$ is an integer, as specified in Lemma \ref{lem:VZ} within Section \ref{sec:4}. Consequently, the analogous result cannot be deduced under the condition $ s > \frac{n}{2} + 2 $, which is corroborated by the findings in \cite[Theorem 5.1]{Zhang} and \cite[Theorem 4.3]{Hongheng}. Nonetheless, our methodology of proof is, to a certain extent, constructive in nature. The control inputs and solutions are amenable to numerical computation via an iterative algorithm, as articulated by equations \eqref{eqnofvalpha} and \eqref{eqnofzalpha} presented in Section \ref{sec:4}. We expect that the regularity condition imposed on $ s $ may be relaxed.
		\end{remark}
		
		\subsection{Fully nonlinear case}
		Finally, we are going to consider the full nonlinear system. Assume that nonlinearity  $f(t,x,y,y',\nabla^2 y)$ is a smooth function and satisfies the following condition:
		\begin{equation}\label{assume:f3}
			f=O(|y|^2+|y'|^2+|\nabla^2y|^2), \quad {\color{black}\text{as}~(|y|+|y'|+|\nabla^2y|\to 0),}
		\end{equation}
		where $y'=(y_t,\nabla y)$. We then present our result concerning another type of null controllability:
		{\color{black}	\begin{thm}\label{fully nonlinear controllability}
				Let nonlinearity $f$ satisfies conditions  \eqref{cond:f2} and \eqref{assume:f3}, and let 	$(T,\omega)$ satisfy Assumption \ref{defi:weak-Gamma}.
				Assume further that there exists a positive constant $\varepsilon_{fnon}>0$, such that for any initial data $(y^0,y^1) $, the following norm condition
				\begin{equation}
					\|y^0\|_{\mathcal{H}^s}+\|y^1\|_{\mathcal{H}^{s-1}}\leq \varepsilon_{fnon},
				\end{equation}			
				holds
				for some integer $s\geq {\color{black}\max\{n+3,5\}}$. Then there exists a control  $u(t,x)\in L^2(0, T; H^{s-1})$ and a unique solution $y\in C(0,T;\mathcal{H}^{s})\cap C^1(0,T;\mathcal{H}^{s-1})\cap C^2(0,T;\mathcal{H}^{s-2})$ to \eqref{Wave Eqn with internal control} with internal control that  satisfies
				\begin{equation}
					y_t(T)=0, ~ y_{tt}(T)=0.
				\end{equation}
		\end{thm}}

		\subsection{Organization of this paper}
	
		The rest of this paper is organized as follows. In Section \ref{sec:2}, we introduce three distinct methods and establish the exact controllability of the damped Klein-Gordon equation, thereby laying the groundwork for our subsequent analysis. Section \ref{sec:3} is dedicated to demonstrating the null controllability of the damped semilinear wave equation, with the proof of Theorem \ref{thm:semilinear} as its culmination. In Section \ref{sec:4}, we present the proof of Theorem \ref{thm:main1}, which addresses the controllability of the quasilinear damping wave system with small initial data. Section \ref{sec:5} focuses on proving Theorem \ref{fully nonlinear controllability}, which concerns the local null controllability of the damped fully nonlinear wave equation. Finally, the Appendix \ref{appdix} contains the proof of an observability inequality for the linear time-dependent wave system, a result that is crucial for establishing Theorem \ref{thm:main1}.

		\section{Controllability for linear damped hyperbolic system}\label{sec:2}

		In this section, we will consider the null controllability problem for the
		linear system
		\begin{equation}\label{Damped Wave Eqn with internal control}
			\left\{
			\begin{aligned}
				&	y_{tt}+b_0y_t-\sum\limits_{i,j=1}^n\big(a^{ij}y_{x_i}\big)_{x_j}+\sum\limits_{k=1}^nb_ky_{x_k}+\tilde{b}y=\chi_\omega(x)u(t,x), & (t,x)\in & ~ (0,T)\times\Omega, \\
				& y(t,x)=0, & (t,x)\in & ~ (0,T)\times\partial\Omega, \\
				& y(0,x)=y^0, ~ y_t(0,x)=y^1, & x\in & ~ \Omega. \\
			\end{aligned}
			\right.
		\end{equation}
		Here we assume that coefficients $a^{ij}\in C^1( [0,T]\times \overline{\Omega})$ satisfy
		\begin{equation}\label{AA1}
			a^{ij}(t,x)=a^{ji}(t,x), ~ \text{for} ~ (t,x)\in [0,T]\times \overline{\Omega}, ~ i,j=1,\cdots,n,
		\end{equation}
		and for some $\beta>0$,
		\begin{equation}\label{AA2}
			\sum_{i,j=1}^n a^{ij}(t,x) \xi^i\xi^j\geq \beta |\xi|^2, ~ \text{for} ~ (t,x,\xi) \in [0,T]\times \overline{\Omega}\times \R^n,
		\end{equation}
		where $\xi=(\xi^1,\cdots,\xi^n)\in \R^n$.
		Assume that
		\begin{equation}\label{BB1}
			~ b_0\in C^1([0,T]\times \overline{\Omega}), ~ b_k, ~ \tilde{b}\in L^\infty([0,T]\times \overline{\Omega}), ~ k=1,\cdots,n.
		\end{equation}
	
		Thanks to classical semi-group theory (see \cite{pazy83}), we can obtain that for any $(y^0,y^1)\in \mathcal{H}^1\times L^2$ and $u\in L^2((0,T)\times\Omega)$, System \eqref{Damped Wave Eqn with internal control} admits a global solution $$y\in C(0,T;\mathcal{H}^1)\cap C^1(0,T;L^2).$$
		
		We say that system \eqref{Damped Wave Eqn with internal control} is exactly null controllable in $\mathcal{H}^1\times L^2$, if
		for any given $(y^0,y^1)\in \mathcal{H}^1\times L^2$, there exists a control function $u\in L^2((0,T)\times\Omega)$, such that  $(y(T),y_t(T))=(0,0)$. 
		
		In order to study the null controllability of system \eqref{Damped Wave Eqn with internal control}, we need to consider the following dual system:
		\begin{equation}\label{system:z-HUM}
			\left\{
			\begin{aligned}
				& z_{tt}+b_0z_t-\sum\limits_{i,j=1}^n\big(a^{ij}z_{x_i}\big)_{x_j}+\sum\limits_{k=1}^nb_kz_{x_k}+\tilde{b}z=0, & (t,x)\in & ~ (0,T)\times\Omega, \\
				& z(t,x)=0, & (t,x)\in & ~ (0,T)\times\partial\Omega, \\
				& z(0,x)=z_0, ~ z_t(0,x)=z_1 & x\in & ~ \Omega.
			\end{aligned}
			\right.
		\end{equation}
		
		We say  system \eqref{system:z-HUM} is exactly observable in $\mathcal{H}^1\times L^2$, if for any initial data $(z_0,z_1)\in \mathcal{H}^1\times L^2$, the corresponding solution $z\in C(0,T;  \mathcal{H}^1)\cap C^1(0,T;  L^2)$ of system \eqref{system:z-HUM} holds an observability inequality
		\begin{equation}\label{obs}
			\|z_0\|_{\mathcal{H}^1}^2+\|z_1\|_{L^2}^2\leq C\int_0^T \|z_t\|^2_{L^2(\omega)}\dd t,
		\end{equation}
		where $C$ is a positive constant independent of $(z_0,z_1)$.

		\subsection{Constant case}
		In this subsection, we assume the coefficients are specified as:
		
		\begin{equation}\label{sect2:a}
			a^{ij}=\delta_{ij}, ~ i,j=1,\cdots,n
		\end{equation}
		where $\delta_{ij}$ is  Kronecker delta function and
			\begin{equation}\label{sect2:b}
				b_0=1, b_k=0,\tilde{b}=0.
			\end{equation}

		We introduce an alternative method to prove the following theorem. This method will  be instrumental in the subsequent proofs of our main results.
		\begin{theorem}\label{thm: the example thm2}
			Assume that $a^{ij}$ satisfies \eqref{sect2:a}. Assume that  \eqref{sect2:b} is valid. If System \eqref{system:z-HUM} is exactly observable in $\mathcal{H}^2\times \mathcal{H}^1$, then system \eqref{Damped Wave Eqn with internal control} is exactly null controllable.
		\end{theorem}
		
		\begin{remark}
			Indeed, by using HUM method, it is not difficult to show that system \eqref{Damped Wave Eqn with internal control} is exactly null controllable in the space $ L^2(\Omega) \times H^{-1}(\Omega) $, provided that system \eqref{system:z-HUM} exhibits exact observability in $ \mathcal{H}^1\times L^2$.
		\end{remark}

		Before our proof,  we introduce an intermediate value lemma as follow.
		\begin{lemma}\label{lemma from Evans-example}
			Let $ F: \mathbb{R}^N \rightarrow \mathbb{R}^N $ be a continuous function. Let  $ A,B \in \mathbb{R}^{N \times N} $ be any given symmetric positive definite matrix. Suppose that the inequality
			\begin{equation}\label{inner-product}
				(Bx, AF(x))_{\ell_2}:=Bx\cdot AF(x) \geq 0,
			\end{equation}
			holds for all $ x $ with $ |Bx|_{\ell_2} = r $, for some $ r > 0 $. Then, there exists a point $ x_0 \in \R^N, $ such that $ Bx_0\in B_r $ and $ F(x_0) = 0 $, where $ B_r $ denotes the closed ball in $ \mathbb{R}^N $ with radius $ r $ centered at the origin.
		\end{lemma}
		{\color{black}	\begin{proof}
				The case where $A=B=Id_{N\times N}$ can be found in L. C. Evans \cite{Evans}.  We argue by contradiction and assume the assertion to be false, it would imply that  $F(x)\not=0$ for all $Bx\in B_r$. We define the continuous mapping $w: \overline{B}_r\rightarrow \partial B_r$ as follows:
				\begin{equation}\label{ww}
					w(y):=-\frac{rF(B^{-1}y)}{|F(B^{-1}y)|_{\ell_2}}, ~ \forall ~ y\in \bar{B}_r.
				\end{equation}
				According to Brouwer's Fixed Point theorem, there exists a point $z\in B_r\setminus \{0\}$ with $w(z)=z.$
				
				Now, taking $Bx_1=z, x_1\in \R^N\setminus\{ 0\}$, then by definition \eqref{ww} of $w$, we have
				\begin{equation}
					Bx_1=	w(Bx_1)=-\frac{rF(x_1)}{|F(x_1)|_{\ell_2}}.\end{equation}
				
				Hence, we claim that equation \eqref{inner-product} will lead to a contradiction. We now proceed to analyze the inner product bound of $(Bx_1, ABx_1)_{\ell_{2}}$ as follows:
				\begin{equation}
					0<	(Bx_1,ABx_1)_{\ell_{2}}=(Bx_1,Aw(Bx_1))_{\ell_{2}}=- \frac{r}{|F(x_1)|_{\ell_2}} Bx_1\cdot AF(x_1)\leq 0.
				\end{equation}
				
				This contradiction indicates that our initial assumption is not true, thereby establishing the existence of a point $ x_0 \in B_r $ for which $ F(x_0) = 0 $, and thus concluding the proof.
		\end{proof} }

		We now proceed to establish the null controllability of the system \eqref{Damped Wave Eqn with internal control}. The idea of this method will be used in the proof of Theorem \ref{thm:main1}.
		\begin{proof}
			Let  $\{\varphi_j\}_{j=1}^\infty$ be the eigenfunction of $-\Delta$ with Dirichlet boundary condition coresponding to eigenvalue $\lambda_j^2$. Thanks to elliptic equation theory,  $\{\varphi_j\}_{j=1}^\infty$ actually is the standard orthogonal basis of $L^2(\Omega)$ such that for each $j$,
			\begin{equation}\label{lambdai}
				\left\{
				\begin{aligned}
					&( -\Delta)\varphi_j=\lambda_j^2\varphi_j, & x\in & ~ \Omega, \\
					& \varphi_j=0, & x\in & ~ \partial\Omega,
				\end{aligned}
				\right.
			\end{equation}
			and define the finite energy elements $(y_N^0,y_N^1)$ as follow:
			\begin{equation}\label{yninitial}
				y_N^0=\sum\limits_{j=1}^N(y^0,\varphi_j)_{L^2}\varphi_j, \quad y_N^1=\sum\limits_{j=1}^N(y^1,\varphi_j)_{L^2}\varphi_j.
			\end{equation}

			Let  $(y_N,v_N)$ be given by
			\begin{equation}\label{yvN}
				y_N=\sum\limits_{j=1}^Ng_{jN}(t)\varphi_j, \quad v_N=\sum\limits_{j=1}^Nh_{jN}(t)\varphi_j,
			\end{equation}
			which solves the finite-dimensional system
			\begin{equation}\label{finite dimision yN-example}
				\left\{
				\begin{aligned}
					& \Big(\partial_t^2y_N-\Delta y_N+2\partial_ty_N-\chi_\omega\partial_tv_N,\varphi_i\Big)_{L^2}=0, \quad i=1,2,\cdots,N \\
					& t=0: g_{jN}=(y^0,\varphi_j)_{L^2}, ~ g'_{jN}=(y^1,\varphi_j)_{L^2} \\
				\end{aligned}
				\right.
			\end{equation}
			and the backward system
			\begin{equation}\label{finite dimision vN-example}
				\left\{
				\begin{aligned}
					& \Big(\partial_t^2v_N-\Delta v_N-2\partial_tv_N,\varphi_i\Big)_{L^2}=0, \quad i=1,2,\cdots,N \\
					& t=T: h_{jN}=a_j, ~ h'_{jN}=b_j, \\
				\end{aligned}
				\right.
			\end{equation}
			where $(\vec{a}_N,\vec{b}_N)=(a_1,\cdots,a_N,b_1,\cdots,b_N)\in \mathbb{R}^{2N}$ with
			\begin{equation}\label{abb}
				\sum_{i=1}^N \left(|\lambda_ia_i|^2+|\lambda_i||b_i|^2\right)<\infty.
			\end{equation}
			
			Now we define $\mathcal{F}_G^N:\mathbb{R}^{2N}\rightarrow \mathbb{R}^{2N} $ as follows
			\begin{equation}\label{a map in Galerkin-example2}
				\mathcal{F}_G^N:\big(a_1,\cdots,a_N,b_1,\cdots,b_N\big)\mapsto \big(\lambda_1 g_{1N}(T),\cdots,\lambda_N g_{NN}(T),g'_{1N}(T),\cdots,g'_{NN}(T)\big),
			\end{equation}
			which transforms the final state of $v_N$ to that of $y_N$ at time $T$.  Then we have
			\begin{equation}\label{217}
				B_N\vec{l}\cdot A_N\mathcal{F}^N_G(\vec{l})=\Big(\big(v_N(T),\partial_tv_N(T)\big) ~,~ \mathcal{F}_g\big(v_N(T),\partial_tv_N(T)\big)\Big)_{H^1\times L^2},
			\end{equation}
			for any $\vec{l}=\big(a_1,\cdots,a_N,b_1,\cdots,b_N\big)^\top$,
                where
                $$A_N=Id_{N\times N}, \quad B_N=\text{diag}(\lambda_1, \lambda_2, \cdots \lambda_N, 1, \cdots, 1).$$

			Now our goal is to prove that
			there exists $R>0$, such that
			\begin{equation}\label{218}
				B_N\vec{l}\cdot A_N\mathcal{F}^N_G(\vec{l})\geq 0,
			\end{equation}
			provided $|\vec{l}|_{\ell_2}\geq R$.
			
			In order to obtain \eqref{218}, by recalling the definition of inner product $(\cdot,\cdot)_{\mathcal{H}^1\times L^2}$ and \eqref{217}, we need to prove
			\begin{equation}
				\begin{split}
					&\int_\Omega\Big(\partial_ty_N(T)\partial_tv_N(T)+\nabla y_N(T)\cdot\nabla v_N(T)\Big)\dd x \\ =~&\Big(\big(v_N(T),\partial_tv_N(T)\big) ~,~ \mathcal{F}_g\big(v_N(T),\partial_tv_N(T)\big)\Big)_{\mathcal{H}^1\times L^2}\geqslant0.
				\end{split}				
			\end{equation}
			
			By multiplying the equation in \eqref{finite dimision yN-example} by $h'_{iN}(t)$ and the equation in  \eqref{finite dimision vN-example} by $g'_{iN}(t)$, and summing over $i$,  we derive an energy identity
			\begin{equation}
				\frac{\dd}{\dd t}\int_\Omega\big(\partial_ty_N\partial_tv_N+\nabla y_N\cdot\nabla v_N\big)\dd x=\int_\omega|\partial_tv_N|^2\dd x.
			\end{equation}

			Integrating this over $(0,T)$ with respect to $t$ yields the inequality
			\begin{equation}\label{3.17 example}
				\begin{aligned}
					& \int_\Omega\Big(\partial_ty_N(T)\partial_tv_N(T)+y_N(T)v_N(T)+\nabla y_N(T)\cdot\nabla v_N(T)\Big)\dd x \\
					= & ~ \int_\Omega\Big(\partial_ty_N(0)\partial_tv_N(0)+\nabla y_N(0)\cdot\nabla v_N(0)\Big)\dd x+\int_0^T\int_\omega|\partial_tv_N|^2\dd x\dd t \\
					\geq & ~ \int_0^T\int_\omega|\partial_tv_N|^2\dd x\dd t-\delta_0E(y_N(0))-\frac{1}{4\delta_0}E(v_N(0)),
				\end{aligned}
			\end{equation}
			where $E(v(t))=\int_\Omega(|v_t(t)|^2+|v(t)|^2+|\nabla v(t)|^2)\dd x$.
			
			Now we take $\delta_0>\frac{C}{2}$, $C$ is given by \eqref{obs},  then by the observability inequality,
			\begin{equation}\label{obs1}
				\frac{1}{2\delta_0}E(v_N(0))\leq\int_0^T\int_\omega|\partial_tv_N|^2\dd x\dd t.
			\end{equation}

			Plugging \eqref{obs1} into \eqref{3.17 example}, we get
			\begin{equation}\label{2.14}
				\int_\Omega\Big(\partial_ty_N(T)\partial_tv_N(T)+\nabla y_N(T)\cdot\nabla v_N(T)\Big)\dd x\geq\frac{1}{4\delta_0}E(v_N(0))-\delta_0E(y_N(0)).
			\end{equation}

			Next, multiplying \eqref{finite dimision vN-example} by $2h'_{iN}(t)$ and summing over $i$ yields
			\begin{equation}
				\frac{\dd}{\dd t}E(v_N(t))=2\int_\Omega\big|\partial_tv_N(t)\big|^2\dd x\leq 2E(v_N(t)),
			\end{equation}
			which implies that			
			\begin{equation}
				\frac{\dd}{\dd t}\big(e^{-2t}E(v_N(t))\big)\leq 0.			
			\end{equation}	
				Therefore,
			\begin{equation}
				E(v_N(T))\leq e^{2T}E(v_N(0)).
			\end{equation}

			Combining with \eqref{2.14}, we obtain
			\begin{equation}
				\begin{aligned}
					& \int_\Omega\Big(\partial_ty_N(T)\partial_tv_N(T)+\nabla y_N(T)\cdot\nabla v_N(T)\Big)\dd x \\
					\geq & ~ \frac{1}{4\delta_0}E(v_N(0))-\delta_0E(y_N(0)) \\
					\geq & ~ \frac{1}{4\delta_0e^{2T}}E(v_N(T))-\delta_0E(y_N(0)) .
				\end{aligned}
			\end{equation}
			
			Hence taking $R^2=4\delta_0^2e^{2T}E(y_N(0))$ and	if  $E(u_N(T))\geq R$, then 			
			\begin{equation}\label{Fgge0}	\Big(\big(v_N(T),\partial_tv_N(T)\big) ~,~ \mathcal{F}_g\big(v_N(T),\partial_tv_N(T)\big)\Big)_{\mathcal{H}^1\times L^2}\geq 0.
			\end{equation}

			Now
			we can apply Lemma \ref{lemma from Evans-example}, there exist $\{a_j\}_{j=1}^N$ and $\{b_j\}_{j=1}^N$, such that
			\begin{equation}				\mathcal{F}_G^N\big(\big(a_1,\cdots,a_N,b_1,\cdots,b_N\big)\big)=B_N\big(g_{1N}(T),\cdots,g_{NN}(T),g'_{1N}(T),\cdots,g'_{NN}(T)\big)^\top=0.
			\end{equation}			
			By \eqref{yvN}, this indeed is equivalent to :
			\begin{equation}\label{ytarget}
				y_N(T)=0, ~ \partial_ty_N(T)=0,
			\end{equation}			
			with
			\begin{equation}\label{231}
				E(v_N(T))=\sum_{i=1}^N (\lambda^2_i a_i^2+b_i^2)\leq R^2\leq  CE(y_N(0)).
			\end{equation}
			Let us go back to $v_N$-equation. Multiplying \eqref{finite dimision vN-example}
			by $h_{jNt}$ and summing over $j=1,\cdots, N$,  by integration by parts, we obtain
			\begin{equation}
				E(v_N(0))= E(v_N(T))-2(\chi_\omega\partial_tv_N,\partial_t v_N)_{L^2}\leq E(v_N(T)).
			\end{equation}
			Together with \eqref{231}, this gives
			\begin{equation}
				E(v_N(0))\leq CE(y_N(0)).
			\end{equation}
			
			By using \eqref{3.17 example}, and we can find that
			\begin{equation}
				\int_0^T\int_\omega|\partial_tv_N|^2\dd x\dd t\leq CE(y_N(0))\leq CE(y(0)).
			\end{equation}			
			Consequently, we obtain a bound for $\{\partial_tv_N\}_{N=1}^\infty$  in $L^2(0,T;L^2(\Omega))$.
			
			Moreover, by well-posedness theory of ode systems, we obtain
			\begin{equation}
				\{v_N\}_{N=1}^\infty  \subset L^\infty(0,T;\mathcal{H}^2(\Omega))\cap W^{1,\infty}(0,T;\mathcal{H}^1(\Omega)),
			\end{equation}
			and
			\begin{equation}
				\{y_N\}_{N=1}^\infty  \subset L^\infty(0,T;\mathcal{H}^2(\Omega)), \quad
				\big\{\partial_ty_N\big\}_{N=1}^\infty  \subset L^\infty(0,T;\mathcal{H}^1(\Omega)).
			\end{equation}
			Since $\vec{l}$ is assumed to satisfy \eqref{abb} which is equivalent to that $\|v_N(T)\|_{\mathcal{H}^2}<+\infty$.
			So by using equation \eqref{yvN}, this implies
			\begin{equation}
				\big\{\partial_t^2y_N\big\}_{N=1}^\infty  \subset L^\infty(0,T;L^2(\Omega)).
			\end{equation}
			With the help of classical compactness results (see \cite{S.Zheng}), we can extract a subsequence $\{y_N\}_{N=1}^\infty$ (still denoted by $\{y_N\}_{N=1}^\infty$) such that
			\begin{equation}\begin{cases}
					y_N\stackrel{\ast}{\longrightarrow} y  \text{~in~} L^\infty(0,T;\mathcal{H}^2(\Omega)),\\
					\partial_ty_N\stackrel{\ast}{\longrightarrow} y_t  \text{~in~} L^\infty(0,T;\mathcal{H}^1(\Omega)), \\
					\partial_t^2y_N\stackrel{\ast}{\longrightarrow} y_{tt}  \text{~in~} L^\infty(0,T;L^2(\Omega)),
				\end{cases}			
			\end{equation}
			and
			\begin{equation}
				\partial_tv_N \stackrel{\ast}{\longrightarrow} u  \text{~in~} L^\infty(0,T;L^2(\Omega)),
			\end{equation}
                where $\stackrel{\ast}{\longrightarrow}$ means weakly-$\ast$ convergence.
			
			Combining with initial condition \eqref{yninitial} and  \eqref{ytarget} these convergences  are sufficient to establish that $ y $ is a weak solution to the damped wave equation with
			\begin{equation}
				y(0)=y^0, y_t(0)=y^1, y(T)=0,y_t(T)=0
			\end{equation}
			and
			internal control $ u $.	Thus, we have obtained the null controllability of the system \eqref{Damped Wave Eqn with internal control}.			
		\end{proof}
		
		We finish this part by giving a observation result for linear system, which will be used in the proof of Theorem \ref{thm:semilinear}.
		Consider the following system: \begin{equation}\label{system:z-L}
			\left\{
			\begin{aligned}
				& {\color{black}z_{tt}-Lz_t-\Delta z=0}, & (t,x)\in & ~ (0,T)\times\Omega, \\
				& z(t,x)=0, & (t,x)\in & ~ (0,T)\times\partial\Omega, \\
				&  z(0,x)=z_0, ~ z_t(0,x)=z_1, & x\in & ~ \Omega \\
			\end{aligned}
			\right.
		\end{equation}
		where $L$ is a constant.
		\begin{lem}\label{lem:the example prop}
			Assume that $(\omega,T)$ satisfies GCC. Then there exists a constant $D>L$, such that
			for any initial data $(z_0,z_1)\in \mathcal{H}^1\times L^2(\Omega)$, the corresponding solution $z\in C(0,T;  \mathcal{H}^1)\cap C^1(0,T;  L^{2})$ of system \eqref{system:z-L} holds
			\begin{equation}
				\label{observability inequality semilinear1}
				\|z_0\|_{H_0^1}^2+\|z_1\|_{L^2}^2\leq D \int_0^T \|\nabla z\|^2_{(L^2(\omega))^n}\dd t.
			\end{equation}
            Here $D$ is a constant independent of $z$. Moreover, for any initial data $(z_0,z_1)\in \mathcal{H}^2\times \mathcal{H}^1$, the corresponding solution $z\in C(0,T;  \mathcal{H}^2)\cap C^1(0,T;  \mathcal{H}^1)$ of system \eqref{system:z-L} satisfies
			\begin{equation}\label{observability inequality semilinear2}
				\frac{1}{2}\Big(\big\|\nabla z_t(T)\big\|_{L^2}^2+\|\Delta z(T)\|_{L^2}^2\Big)\leq \tilde{D}\int_0^T\|\nabla z_t\|_{(L^2(\omega))^n}^2\dd t.
			\end{equation}
Here $\tilde{D}$ is a constant independent of $z$.
		\end{lem}
		\begin{proof}
			We first note that equation \eqref{observability inequality semilinear1} is a classical result (see \cite{Bardos}). To prove \eqref{observability inequality semilinear2}, let us  take $v=z_t$, and observe that
		$v$ is  a solution of the system
		$$
            \left\{
            \begin{aligned}
            &v_{tt}-Lv_t-\Delta v=0, & (t,x)\in & ~ (0,T)\times\Omega, \\
            &v=0, & (t,x)\in & ~ (0,T)\times\partial\Omega,
            \end{aligned}
            \right.
            $$
		with initial data
			\begin{align*}
				v(0)=z_1 \in H_0^1, \quad v_t(0)=\Delta z_0+L z_1\in L^2.
			\end{align*}
			From these conditions, we can derive \eqref{observability inequality semilinear2} from \eqref{observability inequality semilinear1}.
		\end{proof}
		
		\subsection{Various case: controllability in $\mathcal{H}^1\times L^2$}
		In this section, we consider the exact null controllability of the system \eqref{Damped Wave Eqn with internal control} in  $\mathcal{H}^1\times L^2$ when the coefficients depend on both space and time. Denote $Q^T:=(0,T)\times\Omega$, we assume that the coefficients $a^{ij}, i,j=1,\cdots,n$  fulfill the conditions  \eqref{AA1}--\eqref{AA2} and additionally satisfy the following bound:
			\begin{equation}\label{coe:aaaaa}
				\|a^{ij}-\delta_{ij}\|_{C^1(\overline{Q^T})}\leq \varepsilon,
			\end{equation}
			for some $\varepsilon$.
			 For simplicity of exposition, we further assume that the coefficients are specified as:
			\begin{equation}\label{sect2:bb}
				b_0=1, b_k=0,\tilde{b}=0.
			\end{equation}
Then we have
\begin{theorem}\label{thm:the example thm}		Assume that \eqref{coe:aaaaa} and \eqref{sect2:bb} are valid.
		Assume that System \eqref{system:z-HUM} is exact observable in $\mathcal{H}^1\times L^2$, then there exists a sufficiently small  $\varepsilon_1$ such that if $\varepsilon\leq \varepsilon_1$, then \eqref{Damped Wave Eqn with internal control} is exactly null controllable in $\mathcal{H}^1\times L^2$.
		\end{theorem}
		\begin{proof}
			{\color{black}Consider the following system:
				\begin{equation}\label{system:z-Picard}
					\left\{
					\begin{aligned}
						& z_{tt}+z_t-\sum\limits_{i,j=1}^n\big(a^{ij}z_{x_i}\big)_{x_j}=0, & (t,x)\in & ~ (0,T)\times\Omega, \\
						& z(t,x)=0, & (t,x)\in & ~ (0,T)\times\partial\Omega, \\
						&  z(0,x)=z_0, ~ z_t(0,x)=z_1, & x\in & ~ \Omega.
					\end{aligned}
					\right.
				\end{equation}
				
				Our strategy involves defining
				\begin{equation}\label{y=w-z linear}
					y(t)=w(t)-z(T-t).
				\end{equation}
				Here  $w$ satisfies
				\begin{equation}\label{the eqn of w linear}
					\left\{
					\begin{aligned}
						& w_{tt}+w_t-\sum\limits_{i,j=1}^n\big(a^{ij}w_{x_i}\big)_{x_j}=-2\chi_{\Omega\setminus\omega}z_t(T-t), & (t,x)\in & ~ (0,T)\times\Omega, \\
						& w(t,x)=0, & (t,x)\in & ~ (0,T)\times\partial\Omega, \\
						& w(0,x)=z(T)+y^0, ~ w_t(0,x)=-z_t(T)+y^1, & x\in & ~ \Omega.
					\end{aligned}
					\right.
				\end{equation}
				It is straightforward to verify that $y$ satisfies
				\begin{equation}\label{the eqn of y linear}
					\left\{
					\begin{aligned}
						& y_{tt}+y_t-\sum\limits_{i,j=1}^n\big(a^{ij}y_{x_i}\big)_{x_j}=2\chi_\omega z_t(T-t), & (t,x)\in & ~ (0,T)\times\Omega, \\
						& y(t,x)=0, & (t,x)\in & ~ (0,T)\times\Omega, \\
						& y(0,x)=y^0, ~ y_t(0,x)=y^1, & x\in & ~ \Omega.
					\end{aligned}
					\right.
				\end{equation}
				Note that $y(T)=w(T)-z_0$ and $~ y_t(T)=w_t(T)+z_1$.
				
				If we can find $(z_0,z_1)$ such that $w(T)=z_0, ~ w_t(T)=-z_1$, then we may take
				\begin{equation}\label{the control u}
					u=2z_t(T-t)
				\end{equation}
				and by the well-posedness of system \eqref{Damped Wave Eqn with internal control}, $u$ will be the control we seek.}
			
			For every $(z_0,z_1)\in \mathcal{H}^1\times L^2$, we define the map
			$$\mathcal{F}:(z_0,z_1)\mapsto\big(w(T),-w_t(T)\big).$$
			
			We aim to show that this map has a fixed point, which would yield the desired conclusion. In the remainder of the proof, we demonstrate that $\mathcal{F}$ is a contraction mapping, and then by contraction mapping theorem, $\mathcal{F}$ has a fixed point.
			
			{\color{black}We  claim that \eqref{obs} implies that
				there exists a constant $\kappa<1$, depending only on $T,\Omega,\omega, a^{ij}$,  such that
				\begin{equation}\label{estimates by observable inequality}
					\begin{aligned}
						& \frac{1}{2}\Big(\int_\Omega a^{ij}(T)z_{x_i}(T)z_{x_j}(T)\dd x +\|z_t(T)\|_{L^2(\Omega)}^2\Big)+e^{\beta T\varepsilon}\int_0^T\|z_t\|^2_{L^2(\Omega\setminus\omega)}\dd t \\
						\leq ~ & \frac{\kappa}{2}\Big(\int_\Omega a^{ij}(0)z_{x_i}(0)z_{x_j}(0)\dd x+\|z_1\|_{L^2(\Omega)}^2\Big).
					\end{aligned}
				\end{equation}
				
				Indeed, by energy equality,
				\begin{equation}
					\begin{aligned}
						& \frac{1}{2}\Big(\int_\Omega a^{ij}(t)z_{x_i}(t)z_{x_j}(t)\dd x+\|z_t(t)\|_{L^2(\Omega)}^2\Big)+\int_0^t\|z_t\|^2_{L^2(\Omega)}\dd \tau \\
						= ~ & \frac{1}{2}\Big(\int_\Omega a^{ij}(0)z_{x_i}(0)z_{x_j}(0)\dd x+\|z_1\|_{L^2(\Omega)}^2\Big)+\frac{1}{2}\int_0^t \int_\Omega a_t^{ij}(\tau)z_{x_i}(\tau)z_{x_j}(\tau)\dd x\dd \tau,
					\end{aligned}
				\end{equation}
				By Gronwall's inequality and the smallness assumption \eqref{coe:aaaaa} on $a^{ij}_t$,
				we have
				\begin{equation}
					\begin{aligned}
						& \frac{1}{2}\Big(\int_\Omega a^{ij}(T)z_{x_i}(T)z_{x_j}(T)\dd x+\|z_t(T)\|_{L^2(\Omega)}^2\Big)+e^{\beta\varepsilon T}\int_0^T\|z_t\|^2_{L^2(\Omega)}\dd t \\
						\leq ~ & \frac{e^{\beta\varepsilon T}}{2}\Big(\int_\Omega a^{ij}(0)z_{x_i}(0)z_{x_j}(0)\dd x+\|z_1\|_{L^2(\Omega)}^2\Big).
					\end{aligned}
				\end{equation}
				 Given that System \eqref{system:z-HUM} is exactly observable, which yields the observability inequality:
				 \begin{equation}
				 	C\int_0^T\|z_t\|^2_{L^2(\omega)}\dd t \geq \Big(\int_\Omega a^{ij}(0)z_{x_i}(0)z_{x_j}(0)\dd x+\|z_1\|_{L^2(\Omega)}^2\Big).
				 \end{equation}
				 Here we utilize the fact that $a^{ij}(0)$ is  positive and bounded.
				 Thus, we have
				 \begin{equation}
				 	\begin{aligned}
				 		& \frac{1}{2}\Big(\int_\Omega a^{ij}(T)z_{x_i}(T)z_{x_j}(T)\dd x+\|z_t(T)\|_{L^2(\Omega)}^2\Big)+e^{\beta\varepsilon T}\int_0^T\|z_t\|^2_{L^2(\Omega\setminus\omega)}\dd t \\
				 		\leq ~ & \left(\frac{e^{\beta\varepsilon T}}{2}-\frac{e^{\beta\varepsilon T}}{C} \right)\Big(\int_\Omega a^{ij}(0)z_{x_i}(0)z_{x_j}(0)\dd x+\|z_1\|_{L^2(\Omega)}^2\Big).
				 	\end{aligned}
				 \end{equation}
			   By  choosing $\varepsilon_1$ such that
				\begin{equation}
					\frac{e^{\beta\varepsilon T}}{2}-\frac{e^{\beta\varepsilon T}}{C}<\frac{1}{2},
				\end{equation}
				 and setting $\kappa=e^{\beta\varepsilon T}(1-\frac{2}{C})$, we obtain  \eqref{estimates by observable inequality}.
				
				Now, multiplying \eqref{the eqn of w linear} by $w_t$ and  integrating by parts, we derive
				\begin{equation}\label{w:esti}
					\begin{aligned}
						& \frac{1}{2}\frac{\dd}{\dd t}\Big(\|w_t\|_{L^2(\Omega)}^2+\int_\Omega a^{ij}(t)w_{x_i}(t)w_{x_j}(t)\dd x\Big)+\|w_t\|_{L^2(\Omega)}^2 \\
						= ~ & -2\int_{\Omega\setminus\omega}w_t(t)z_t(T-t)\dd x + \frac{1}{2}\int_\Omega a_t^{ij}(t)w_{x_i}(t)w_{x_j}(t)\dd x\\
						\leq ~ & \|w_t\|_{L^2(\Omega)}^2+\|z_t(T-t)\|_{L^2(\Omega\setminus\omega)}^2+ \frac{1}{2}\int_\Omega a_t^{ij}(t)w_{x_i}(t)w_{x_j}(t)\dd x.
					\end{aligned}
			\end{equation}}
			
			Integrating \eqref{w:esti} in $t$ from 0 to $T$ and applying Gronwall's inequality, we obtain:
                \begin{align}\label{FF}
                        & \frac{1-C_n\varepsilon}{2}\left\|\mathcal{F}(z_0,z_1)\right\|_{H_0^1\times L^2}^2 \nonumber\\
					\leq ~ & \frac{1}{2}\Big(\|w_t(T)\|_{L^2(\Omega)}^2+\int_\Omega a^{ij}(T)w_{x_i}(T)w_{x_j}(T)\dd x\Big) \nonumber\\
					\leq ~ & e^{CT \varepsilon}\int_0^T\|z_t\|_{L^2(\Omega\setminus\omega)}^2\dd t+\frac{e^{CT \varepsilon}}{2}\left\|-z_t(T)+y^1\right\|_{L^2(\Omega)}^2\nonumber\\
                        &+\frac{e^{CT \varepsilon}}{2}\int_\Omega a^{ij}(0)(z(T)+y^0)_{x_i}(z(T)+y^0)_{x_j}\dd x \nonumber\\
					\leq ~ & e^{CT \varepsilon}\int_0^T\|z_t\|_{L^2(\Omega\setminus\omega)}^2\dd t+\frac{(1+\delta)e^{CT\varepsilon}}{2}\|z_t(T)\|_{L^2(\Omega)}^2+\frac{(1+\delta^{-1})e^{CT\varepsilon}}{2}\|y^1\|_{L^2(\Omega)}^2\nonumber\\
                        &+\frac{e^{CT\varepsilon}}{2}\int_\Omega a^{ij}(0)z_{x_i}(T)z_{x_j}(T)\dd x+\frac{(1+C_n\varepsilon)e^{CT\varepsilon}}{2}\|y^0\|_{H_0^1(\Omega)}^2 \nonumber\\
					&+e^{CT\varepsilon}\int_\Omega a^{ij}(0)z_{x_i}(T) \partial_{x_j}y^0 \, \dd x \nonumber\\
					\leq ~ & e^{CT \varepsilon}\int_0^T\|z_t\|_{L^2(\Omega\setminus\omega)}^2\dd t+\frac{(1+\delta)e^{CT \varepsilon}}{2}\Big(\|z_t(T)\|_{L^2(\Omega)}^2+\int_\Omega a^{ij}(0)z_{x_i}(T)z_{x_j}(T)\dd x\Big) \nonumber\\
					&+\frac{(1+\delta^{-1})(1+C_n\varepsilon)e^{CT \varepsilon}}{2}\Big(\|y^1\|_{L^2(\Omega)}^2+\|y^0\|_{H_0^1(\Omega)}^2\Big)\nonumber \\
					\leq ~ & (1+\delta)e^{CT \varepsilon}\left(\frac{1}{2}\Big(\|z_t(T)\|_{L^2(\Omega)}^2+\int_\Omega a^{ij}(0)z_{x_i}(T)z_{x_j}(T)\dd x\Big)+e^{\beta T\varepsilon}\int_0^T\|z_t\|_{L^2(\Omega\setminus\omega)}^2\dd t\right) \nonumber\\
					&+\frac{(1+\delta^{-1})(1+C_n\varepsilon)e^{CT \varepsilon}}{2}\Big(\|y^1\|_{L^2(\Omega)}^2+\|y^0\|_{H_0^1(\Omega)}^2\Big)\nonumber \\
					\leq ~ & \frac{\kappa(1+\delta)(1+C_n\varepsilon)e^{CT \varepsilon}}{2} \big\|(z_0,z_1)\big\|_{H_0^1\times L^2}^2\nonumber\\
                        &+\frac{(1+\delta^{-1})(1+C_n\varepsilon)e^{CT \varepsilon}}{2}\Big(\|y^1\|_{L^2(\Omega)}^2+\|y^0\|_{H_0^1(\Omega)}^2\Big),
				\end{align}		
			where $C_n$  depends only on $n$. Since $\kappa<1$, we can choose $\varepsilon_1$ such that
                $$
                \kappa \,\frac{1+C_n\varepsilon_1}{1-C_n\varepsilon_1}e^{CT \varepsilon_1}<1,
                $$
                and take $\delta$ sufficiently small  such that
                \begin{equation}\label{cond:111}\kappa(1+\delta)\frac{1+C_n\varepsilon_1}{1-C_n\varepsilon_1}e^{CT \varepsilon_1}<1.\end{equation}
                Then $\mathcal{F}$ is a mapping from the set
			$$\Bigg\{(z_0,z_1)\bigg|\big\|(z_0,z_1)\big\|_{\mathcal{H}^1\times L^2}^2\leq\frac{(1+\delta^{-1})\frac{1+C_n\varepsilon}{1-C_n\varepsilon}e^{CT \varepsilon}}{1-(1+\delta)\kappa\frac{1+C_n\varepsilon}{1-C_n\varepsilon}e^{CT \varepsilon}}\Big(\|y^1\|_{L^2}^2+\|y^0\|_{\mathcal{H}^1}^2 \Big)\Bigg\}$$
			to itself.
			By the definition of $\mathcal{F}$, we know that $\mathcal{F}$ holds
			\begin{equation}
				\mathcal{F}\big(z_0^{(1)},z_1^{(1)}\big)-\mathcal{F}\big(z_0^{(2)},z_1^{(2)}\big)= \mathcal{F}\big(z_0^{(1)}-z_0^{(2)},z_1^{(1)}-z_1^{(2)}\big)
			\end{equation}		
			with the special case that $(y^0,y^1)=(0,0)$.	
			Then due to the above \eqref{FF}, we obtain
			$$\Big\|\mathcal{F}\big(z_0^{(1)},z_1^{(1)}\big)-\mathcal{F}\big(z_0^{(2)},z_1^{(2)}\big)\Big\|_{\mathcal{H}^1\times L^2}^2\leq\kappa(1+\delta)\frac{1+C_n\varepsilon}{1-C_n\varepsilon}e^{CT \varepsilon}\Big\|\big(z_0^{(1)} -z_0^{(2)},z_1^{(1)}-z_1^{(2)}\big)\Big\|_{\mathcal{H}^1\times L^2}^2.$$
		Thus, by \eqref{cond:111}, $\mathcal{F}$ is a contraction map. Hence, by applying contraction mapping theorem, $\mathcal{F}$ has a fixed point, this conclude the proof of our main theorem.
		\end{proof}

		\begin{remark} In contrast to the Hilbert Uniqueness Method (HUM), the presence of damping in the system allows for the identification of the control function in a markedly more straightforward manner. By applying a damping effect, we are able to construct both the control function and the corresponding solutions directly, leveraging the Contraction Mapping Theorem. Nonetheless, the HUM not only ensures the existence of a control function but also yields a wealth of information regarding its properties, such as the $L^2$
			-optimality of the control, the algorithm presented herein does not furnish any guarantees concerning the optimality of the constructed control.
		\end{remark}

\subsection{Various case: controllability in $\mathcal{H}^l\times \mathcal{H}^{l-1}$}

		At the end of this section, let us consider the following linear hyperbolic system.
		\begin{equation}\label{linear hyperbolic observability system}
			\left\{
			\begin{aligned}
				& z_{tt}+b_0z_t-\sum\limits_{i,j=1}^n\big(a^{ij}z_{x_i}\big)_{x_j}+\sum\limits_{k=1}^nb_kz_{x_k}+\tilde{b}z=f, & (t,x)\in & ~ (0,T)\times\Omega, \\
				& z(t,x)=0, & (t,x)\in & ~ (0,T)\times\partial\Omega, \\
				& z(0,x)=z_0, ~ z_t(0,x)=z_1, & x\in & ~ \Omega.
			\end{aligned}
			\right.
		\end{equation}

		{\color{black}
		We begin by stating Theorem \ref{thm:well-reg}, which establishes the well-posedness of the aforementioned linear system and the regularity of its solutions.
		\begin{theorem}
			\label{thm:well-reg}
			Let $T$ be given and $f(t,x)\in C(0,T; \mathcal{H}^1)\cap C^1(0,T;L^2)$.	Assume that $a^{ij}(t,x)=a^{ji}(t,x)\in C^1(\overline{(0,T)\times \Omega})$, and there exists a small constant $\varepsilon_{\mathcal{H}^1}\ll 1$, such that
			\begin{equation}\label{asump:coe2}
				\begin{cases}
					\|a^{ij}-\delta_{ij}\|_{  C^{1}(\overline{Q^T})}<\varepsilon_{\mathcal{H}^1}, ~ i,j=1,\cdots, n, \\
					\|b_0-1\|_{  C^{1}(\overline{Q^T})}<\varepsilon_{\mathcal{H}^1}, ~ \|\tilde{b}\|_{  C^{0}(\overline{Q^T})}<\varepsilon_{\mathcal{H}^1}, \\	  \|b_k\|_{ C^{0}(\overline{Q^T})}<\varepsilon_{\mathcal{H}^1}, ~ k=1,\cdots, n, 
				\end{cases}
			\end{equation}
			where $\delta_{ij}$	 is  Kronecker delta function and $\overline{Q^T}=[0,T]\times \overline{\Omega}$. Then for any initial data $(z_0,z_1)\in \mathcal{H}^2\times \mathcal{H}^1$, system \eqref{linear hyperbolic observability system} admits a unique solution $z\in \cap_{i=0}^2 C^i(0,T; \mathcal{H}^{2-i})$. What is more, the solution $z$ satisfies
			\begin{equation}
				\|z\|_{\cap_{i=0}^2 C^i(0,T; \mathcal{H}^{2-i})} \leq C \left(\|(z_0,z_1)\|_{\mathcal{H}^2\times \mathcal{H}^1}+\|f\|_{C(0,T; \mathcal{H}^1)\cap C^1(0,T;L^2)}\right)
			\end{equation}
			where $C=C(\varepsilon_{\mathcal{H}^1},n,T,\Omega)$ depends on $\varepsilon_{\mathcal{H}^1},n,T,\Omega$.
		\end{theorem}
			

		\begin{proof}
			Let $X:=\mathcal{H}^1\times L^2, Y=\mathcal{H}^2\times \mathcal{H}^1$. Then $Y$ is dense in $X$.
			
			Denote the linear operators as follows:
			\begin{equation*}
					A(t)=\left(
				\begin{array}{cc}
					0& 1 \\
					 \sum\limits_{i,j=1}^n a^{ij}(t,\cdot)\partial^2_{x_ix_j}& 0
				\end{array}\right), ~ B(t)=\left(
				\begin{array}{cc}
					0 & 1 \\
					\tilde{b}(t)+b_{k}(t)\partial_{x_k}+ \sum\limits_{i,j=1}^n (\partial_{x_i}a^{ij}) \partial_{x_j} & b_0
				\end{array}\right),
			\end{equation*}	
		 then 		
		 for any $t\in [0,T]$, $A(t): \mathcal{D}(A(t))\subset X\rightarrow X$ with
		 \begin{equation}\label{222}
		 	\mathcal{D}(A(t))=Y,
		 \end{equation}
		 and $B(t): \mathcal{D}(B(t))\subset X\rightarrow X$ with $\mathcal{D}(B(t))=X$. Moreover, we have $B(t):Y\rightarrow Y$, and thus $\{B(t)\}_{t\in [0,T]}$ is a strongly continuous family of bounded operators on $X$. Therefore, by perturbation theory, it suffices to prove the theorem in the case that $B\equiv0$.
			
			We plan to use \cite[Theorem 5.3]{pazy83} to complete the proof.  In view of the assumptions of \cite[Theorem 5.3]{pazy83},  			
			we only need to verify that
		$\{A(t)\}_{t\in [0, T]}$ satisfies the following conditions:
		\begin{description}
			\item[(1)]  $\{A(t)\}_{t\in [0, T]}$  is a stable family of infinitesimal generators of
			$C_0$ semigroups on $X$;
			\item[(2)] $\mathcal{D}(A(t))=Y$ is independent of $t$;
			\item[(3)]  for any $v\in Y$, $A(t)v$ is continuously differentiable in $X$.
		\end{description}

			Proof of Condition (2) comes from \eqref{222}. Since $a^{ij}\in C^1(\overline{[0,T]\times \Omega})$, condition (3) is also immediately satisfied.
			
			We are left with condition (1). We choose
			$\varepsilon_{\mathcal{H}^1}\ll 1$
			to be sufficiently small such that, by recalling assumption \eqref{coecond:norm1}, we know that
		$(a^{ij})$
			is positive definite. This ensures that the operator
			$a^{ij}(t,\cdot)\partial^2_{x_ix_j}$
			is a second-order elliptic operator that behaves like the Laplacian  $\Delta$.
			
			We observe that for any
			 $t\in [0,T]$,
			$A(t)$ is closed and $(\lambda-A(t))^{-1}$ exists. And  there exist $M>0$ and $w\in \R$, such that for any $\lambda>w$,
			\begin{equation}\label{Att}
				\|(\lambda-A(t))^{-k}\|\leq M(\lambda-w)^{-k}, \quad k\in\mathbb{N}_+.
			\end{equation}
			Then, according to Hille-Yoshida Theorem, we have proved that  $\{A(t)\}_{t\in [0, T]}$ is a family of infinitesimal
			generators of $C_0$ semi-groups on $X$.
			
			Furthermore, since \eqref{Att} is valid for any $t\in [0,T]$, we know that $\{A(t)\}_{t\in [0, T]}$ is stable (see \cite[Definition 2.1]{pazy83} for definition). Hence we have obtained the proof of (1). Now we can apply \cite[Theorem 5.3]{pazy83} to complete the proof.
		\end{proof}

		\begin{corollary}
			\label{cor:well-reg}
			Let $T>0$ be given and $l\geq \frac{n}{2}+1$ be a positive integer.   Assume that $a^{ij}(t,x)=a^{ji}(t,x)\in \cap_{i=0}^l C^i(0,T;\mathcal{H}^{l-i})$, $f(t,x)\in  \cap_{i=0}^{l-1} C^{i}(0,T; \mathcal{H}^{l-i-1})$, and there exists a small constant $\varepsilon_{\mathcal{H}^l}\ll 1$, such that
				\begin{equation}\label{coecond:norm1}
				\begin{cases}
					\|a^{ij}-\delta_{ij}\|_{ \cap_{i=0}^l C^i(0,T;\mathcal{H}^{l-i})}<\varepsilon_{\mathcal{H}^l}, ~ i,j=1,\cdots, n, \\
					\|b_0-1\|_{\cap_{i=0}^l C^i(0,T;\mathcal{H}^{l-i})}<\varepsilon_{\mathcal{H}^l}, ~ \|\tilde{b}\|_{\cap_{i=0}^l C^i(0,T;\mathcal{H}^{l-i})}<\varepsilon_{\mathcal{H}^l}, \\	  \|b_k\|_{\cap_{i=0}^l C^i(0,T;\mathcal{H}^{l-i})}<\varepsilon_{\mathcal{H}^l}, ~ k=1,\cdots, n.
				\end{cases}
			\end{equation}
			Moreover, assume that $a^{ij}, b_k, \tilde{b}$ satisfy that boundary compatibility condition: for any $u\in C^0(0,T;\mathcal{H}^l)\cap C^1(0,T;\mathcal{H}^{l-1})$ and $t\in [0, T)$,
			\begin{equation}\label{assumppp}\begin{cases}
				\sum\limits_{i,j=1}^n\big(a^{ij}\big)_{x_j}u_{x_i}, ~ \sum\limits_{i,j=1}^na^{ij}u_{x_ix_j} \in \mathcal{H}^{l-2},\\ b_0u_t, ~   \sum\limits_{k=1}^nb_ku_{x_k}, ~ \tilde{b}u\in \mathcal{H}^{l-2}.
			\end{cases}
			\end{equation}
			Then for any initial data $(z_0,z_1)\in \mathcal{H}^l\times \mathcal{H}^{l-1}$, \eqref{linear hyperbolic observability system} admits a unique solution $z\in  C(0,T;\mathcal{H}^{l})\cap C^1(0,T;\mathcal{H}^{l-1}) \cap C^2(0,T;\mathcal{H}^{l-2})$ satisfying
			\begin{equation}\label{268}
				\|z\|_{\cap_{k=0}^2 C^k(0,T;\mathcal{H}^{l-k})} \leq C \left(\|(z_0,z_1)\|_{\mathcal{H}^l\times \mathcal{H}^{l-1}}+\|f\|_{\cap_{i=0}^l C^{2}(0,T; \mathcal{H}^{l-i})}\right),
			\end{equation}
			where $C=C(l,n,\varepsilon_{\mathcal{H}^l},T,\Omega)$ depends on $l,n,\varepsilon_{\mathcal{H}^l},T$ and $\Omega$.
		\end{corollary}
		\begin{proof}
			We only sketch the proof here. Let $\tilde{X}:=\mathcal{H}^{l-1}\times \mathcal{H}^{l-2}, \tilde{Y}=\mathcal{H}^l\times \mathcal{H}^{l-1}$.
			
			Noting that when $l>\frac{n}{2}+1$, according to the Sobolev embedding theorem, $a^{ij}(t,x)\in \cap_{i=0}^2 C^i(0,T;\mathcal{H}^{l-i})$ implies $a^{ij}(t,x)\in C^1(\overline{[0,T]\times \Omega})$. Additionally, if taking $\varepsilon_{\mathcal{H}^l}\leq \varepsilon_{\mathcal{H}^1}$, then assumption \eqref{asump:coe2} is clearly satisfied by \eqref{coecond:norm1}.
		    Thanks to assumption \eqref{assumppp}, we still have that for any $t\in [0,T]$, $A(t): \mathcal{D}(A(t))\subset \tilde{X}\rightarrow \tilde{X}$ with
		 \begin{equation}
		 	\mathcal{D}(A(t))=\tilde{Y}.
		 \end{equation}
		 and $B(t): \mathcal{D}(B(t))\subset \tilde{X}\rightarrow \tilde{X}$ with
			$\mathcal{D}(B(t))=\tilde{X}$.
           Thus, similar to Theorem \ref{thm:well-reg}, we can use \cite[Theorem 5.3]{pazy83} to obtain that
		there exists a unique solution $z\in  \cap_{k=0}^2C^k(0,T;\mathcal{H}^{l-k})$ for system \eqref{linear hyperbolic observability system}.
		\end{proof}
		
	}
      At the end, we provide a higher-order energy version of the observability inequality
		\begin{theorem}\label{fully nonlinear observability}
			Assume that $(T,\omega)$ satisfy Assumption \ref{defi:weak-Gamma} for some constant $\varepsilon_0$. Let $l>\frac{n}{2}+2$ be an integer. 	
			Assume \eqref{AA1}, \eqref{AA2} and \eqref{BB1} are valid.
			Then there exists  a small constant $\varepsilon_{obs}=\varepsilon_{obs}(\varepsilon_0)>0$ such that \begin{equation}\label{coecond:norm2}
				\begin{cases}
				\|a^{ij}-\delta_{ij}\|_{ \cap_{i=0}^l C^i(0,T;\mathcal{H}^{l-i})}<\varepsilon_{obs}, ~ i,j=1,\cdots, n \\
				\|b_0-1\|_{\cap_{i=0}^l C^i(0,T;\mathcal{H}^{l-i})}<\varepsilon_{obs}, \quad \|\tilde{b}\|_{\cap_{i=0}^l C^i(0,T;\mathcal{H}^{l-i})}<\varepsilon_{obs}, \\	  \|b_k\|_{\cap_{i=0}^l C^i(0,T;\mathcal{H}^{l-i})}<\varepsilon_{obs}, ~ k=1,\cdots, n,
				\end{cases}
			\end{equation}
			then for any initial data $(z_0,z_1)\in \mathcal{H}^1\times L^2$ and $f\in L^2((0,T)\times \Omega)$, the corresponding solution $z$ of system \eqref{linear hyperbolic observability system}
			holds
			\begin{equation}\label{linear hyperbolic observability inequality}
				\|z_1\|_{L^2(\Omega)}^2+\|z_0\|_{\mathcal{H}^1}^2\leq D_1\left(\int_0^T\int_\omega|z_t|^2\dd x\dd t+\int_0^T \big\| f \big\|_{L^2}^2 \dd t\right),
			\end{equation}
			where $D_1=D_1(T,\omega, \Omega, n,\varepsilon_{obs})>0$ depends on $T,\omega, \Omega, n$ and $\varepsilon_{obs}$.
\end{theorem}
		Utilizing the Duhamel principle, equation \eqref{linear hyperbolic observability inequality} is directly derived from the homogeneous observability inequality of type $(i.e., f \equiv 0)$. Theorem \ref{fully nonlinear observability} possesses its own integrity. We elect to postpone the proof to the appendix.
		
		\section{Proof of Theorem \ref{thm:semilinear}}\label{sec:3}

		This section is devoted to proving Theorem \ref{thm:semilinear}.
		The proof relies on the Galerkin method and a fixed-point Lemma \ref{lemma from Evans-example}, which are introduced in the proof of Theorem \ref{thm: the example thm2}.
		
		Let $\{\varphi_j\}_{j=1}^\infty$ be the eigenfunctions of the Laplacian $-\Delta$ on $\Omega$ corresponding to the eigenvalues $\{\lambda_j^2\}_{j=1}^\infty$ such that
		\begin{equation}\label{elliptic1}
			\left\{
			\begin{aligned}
				&-\Delta\varphi_j=\lambda_j^2\varphi_j, & x\in & ~ \Omega, \\
				& \varphi_j=0, & x\in & ~ \partial\Omega,
			\end{aligned}
			\right.
		\end{equation}
		Due to the classical elliptic operator theory, $\lambda_j$ satisfies
		\begin{equation}\label{lambdajj}
		0 < \lambda_1^2 \leq \lambda_2^2 \leq \cdots < +\infty,
		\end{equation}
		and $\lambda_j^2 \rightarrow \infty$ as $j$ tends to infinity.
		Furthermore, $\{\varphi_j\}_{j=1}^\infty$ is the standard orthogonal basis of $L^2(\Omega)$.
		
		Let $(y_N^0, y_N^1)$ be the asymptotic initial conditions defined by
		\begin{equation}
			y_N^0 = \sum\limits_{j=1}^N (y^0, \varphi_j)_{L^2} \varphi_j, \quad y_N^1 = \sum\limits_{j=1}^N (y^1, \varphi_j)_{L^2} \varphi_j,
		\end{equation}
		then, recalling the initial conditions for the data $(y^0, y^1) \in \mathcal{H}^2 \times \mathcal{H}^1$, we get		
		\begin{equation}\label{1111}
			\begin{split}
				y_N^0 \rightarrow y^0, \quad \text{in} \quad \mathcal{H}^2, \qquad
				y_N^1 \rightarrow y^1, \quad \text{in} \quad \mathcal{H}^1.
			\end{split}
		\end{equation}

		Let $(y_N,v_N)$ be the finite approximation solution defined by
		\begin{equation}\label{yvN-1}
			y_N=\sum\limits_{j=1}^Ng_{jN}(t)\varphi_j, \quad v_N=\sum\limits_{j=1}^Nh_{jN}(t)\varphi_j,
		\end{equation}
		where  the coefficients $(g_{jN},h_{jN})$ solve the finite-dimensional system
		\begin{equation}\label{finite dimision yN}
			\left\{
			\begin{aligned}
				& \Big(\partial_t^2y_N-\Delta y_N+f\big(\partial_ty_N\big)-\chi\cdot\partial_tv_N,\varphi_i\Big)_{L^2}=0, \quad i=1,2,\cdots,N, \\
				& t=0: g_{jN}=(y^0,\varphi_j)_{L^2}, ~ g'_{jN}=(y^1,\varphi_j)_{L^2}, ~ j=1,2,\cdots,N,\\
			\end{aligned}
			\right.
		\end{equation}
		and backward system
		\begin{equation}\label{finite dimision vN}
			\left\{
			\begin{aligned}
				& \Big(\partial_t^2v_N-\Delta v_N-L\partial_tv_N,\varphi_i\Big)_{L^2}=0, \quad i=1,2,\cdots,N, \\
				& t=T: h_{jN}=a_j, ~ h'_{jN}=b_j, ~ j=1,2,\cdots,N. \\
			\end{aligned}
			\right.
		\end{equation}
		Contrasting with the linear damped wave equation case, the term $f(u_t)$ requires a higher-order energy estimate, rather than the one-order energy estimate.
		We need to define  two energy functionals as follows: for any $u(t)\in C^0(0,T; \mathcal{H}^2)\cap C^1(0,T; \mathcal{H}^1) $,
		\begin{equation}			
			E_1(u(t)):=\int_\Omega\big(|u_t(t)|^2+|\nabla u(t)|^2\big)\dd x,\qquad E_2(u(t)):=\int_\Omega\big(|\nabla u_t(t)|^2+|\Delta u(t)|^2\big)\dd x.
		\end{equation}
		
		By \eqref{yvN-1} and the fact that the $\varphi_j$ in \eqref{elliptic1} are orthogonal, we know the norm equivalence relations are given by
	
		\begin{equation}\label{energydecay}
			\begin{split}
				E_1(y_N(0))	=\sum_{j=1}^N \left(|\lambda_j g_{jN}(0)|^2 +| g'_{jN}(0)|^2\right)\sim \|(y_N^0,y_N^1)\|^2_{\mathcal{H}^1\times L^2}\leq 	\|(y^0,y^1)\|^2_{\mathcal{H}^2\times 		 \mathcal{H}^1}, \\
				E_2(y_N(0))	=\sum_{j=1}^N \left(|\lambda_j^2 g_{jN}(0)|^2 +|\lambda_j g'_{jN}(0)|^2\right)\sim \|(y_N^0,y_N^1)\|^2_{\mathcal{H}^2\times \mathcal{H}^1}\leq 	\|(y^0,y^1)\|^2_{\mathcal{H}^2\times \mathcal{H}^1}.
			\end{split}		
		\end{equation}
		We can then define a continuous map $\mathcal{F}_N: \mathbb{R}^{2N}\to \mathbb{R}^{2N}$ by
		\begin{equation}\label{a map in Galerkin2}
			\mathcal{F}_N:\big(a_1,\cdots,a_N,b_1,\cdots,b_N\big)^\top\mapsto \Lambda_N\big(g_{1N}(T),\cdots,g_{NN}(T),g'_{1N}(T),\cdots,g'_{NN}(T)\big)^\top,
		\end{equation}
		where $\Lambda_N=\text{diag}(\lambda_1^2,\cdots,\lambda_N^2,\lambda_1,\cdots,\lambda_N)  \in \R^{2N\times 2N}$  and $|\Lambda_N \big(a_1,\cdots,a_N,b_1,\cdots,b_N\big)^\top|_{\ell_2}<\infty$. Then we state the following lemma, which plays a key role in our proof of Theorem \ref{thm:semilinear}.
		\begin{lemma}\label{lem:Galerkin}
			Under the condition of Theorem \ref{thm:semilinear}, let $\mathcal{F}_N$ be defined by \eqref{a map in Galerkin2}. Then there exists a constant $R$ independent of $N$ and  $x_{0}^N=\big(a_1,\cdots,a_N,b_1,\cdots,b_N\big)\in \mathbb{R}^{2N}$  such that $|\Lambda_Nx_0^N|_{\ell_2}\leq R$ and
			\begin{equation}
				\mathcal{F}_N(x_0^N)=0.
			\end{equation}
		\end{lemma}

		\begin{proof}
			Multiplying equation \eqref{finite dimision yN} by $(\lambda_i^2+\delta^{-1})h'_{iN}(t)$, equation \eqref{finite dimision vN} by $(\lambda_i^2+\delta^{-1})g'_{iN}(t)$, adding them together, and sum this over $i=1,\cdots, N$, we get
			\begin{equation}
				\begin{aligned}
					& \frac{\dd}{\dd t}\int_\Omega\Big(\frac{1}{\delta}\partial_ty_N\partial_tv_N+\nabla\partial_ty_N\cdot\nabla\partial_tv_N\Big)\dd x+\frac{\dd}{\dd t}\int_\Omega\Big(\frac{1}{\delta}\nabla y_N\cdot\nabla v_N+\Delta y_N\Delta v_N\Big)\dd x \\
					&+\int_\Omega\Big(f'\big(\partial_ty_N\big)\nabla\partial_ty_N\cdot \nabla\partial_tv_N-L\nabla\partial_ty_N\cdot\nabla\partial_tv_N\Big)\dd x \\
					&+\frac{1}{\delta}\int_\Omega\big(f\big(\partial_ty_N\big)\partial_tv_N-L\partial_ty_N\partial_tv_N\big)\dd x \\
					= & ~ \frac{1}{\delta}\int_\Omega\chi|\partial_tv_N|^2\dd x-\int_\Omega\chi\partial_tv_N\Delta\partial_tv_N\dd x \\
					= & ~ \int_\Omega\chi\big|\nabla\partial_tv_N\big|^2\dd x+\int_\Omega\Big( \frac{\chi}{\delta}-\frac{\Delta\chi}{2}\Big)|\partial_tv_N|^2\dd x,
				\end{aligned}
			\end{equation}
			where the constant $\delta>0$ will be determined later.
			
			Setting  $\vec{l}=\big(a_1,\cdots,a_N,b_1,\cdots,b_N\big)^\top$, $B_N=\Lambda_N$ and  \begin{equation}\label{ABNN}
				A_N=\text{diag}\left(1+\frac{1}{\delta\lambda_1^2},\cdots,1+\frac{1}{\delta\lambda_N^2}, 1+\frac{1}{\delta},\cdots,1+\frac{1}{\delta}\right)\in \R^{2N\times 2N}.\end{equation}
			Then, integrating the above equation with respect to  $t\in[0,T]$, we get
			\begin{align}\label{3111}
					& (B_N\vec{l},A_NF_N(\vec{l}))_{\ell_2} \nonumber \\
					=~& \frac{1}{\delta}\int_\Omega\Big(\partial_ty_N(T)\partial_tv_N(T)+\nabla y_N(T)\cdot\nabla v_N(T)\Big)\dd x \nonumber \\
					&+ \int_\Omega\Big(\nabla\partial_ty_N(T)\cdot\nabla\partial_tv_N(T)+\Delta y_N(T)\Delta v_N(T)\Big)\dd x \nonumber \\
					=~ & \frac{1}{\delta}\int_\Omega\Big(\partial_ty_N(0)\partial_tv_N(0)+\nabla y_N(0)\cdot\nabla v_N(0)\Big)\dd x \nonumber \\
					&+ \int_\Omega\Big(\nabla\partial_ty_N(0)\cdot\nabla\partial_tv_N(0)+\Delta y_N(0)\Delta v_N(0)\Big)\dd x \nonumber \\
					&+ \frac{1}{\delta}\int_0^T\int_\Omega\big(L\partial_ty_N-f(\partial_ty_N)\big)\partial_tv_N\dd x\dd t+\int_0^T\int_\Omega\big(L-f'(\partial_ty_N)\big)\nabla\partial_ty_N\cdot\nabla \partial_tv_N\dd x\dd t \nonumber \\
					&+ \int_0^T\int_\Omega\chi\big|\nabla\partial_tv_N\big|^2\dd x\dd t+\int_0^T\int_\Omega\Big(\frac{\chi}{\delta}-\frac{\Delta\chi}{2}\Big)|\partial_tv_N|^2\dd x\dd t.
				\end{align}

			Now our goal is to prove that there exists a $R>0$ independent of $N$, such that if for any $\vec{l}\in \R^{2N}$ holds $|B_N\vec{l}|_{\ell_2} \geq R$, then
			\begin{equation}\label{FNl}
				(B_N\vec{l},A_NF_N(\vec{l}))_{\ell_2} \geq 0.
			\end{equation}
			
			Therefore, if equation \eqref{FNl} is valid, combining with \eqref{ABNN} where $A_N,B_N$ are positive define, then we can apply Lemma \ref{lemma from Evans-example} to establish the existence of  $x_0^N$, such that $F_N(x_0^N)=0$, thereby completing the proof.
			
			To get the lower bound of the right hand side of \eqref{3111}, we denote 	\begin{equation}
				\begin{split}							
					J_1= & ~ \frac{1}{\delta}\int_\Omega\Big(\partial_ty_N(0)\partial_tv_N(0)+\nabla y_N(0)\cdot\nabla v_N(0)\Big)\dd x \\
					&+ \int_\Omega\Big(\nabla\partial_ty_N(0)\cdot\nabla\partial_tv_N(0)+\Delta y_N(0)\Delta v_N(0)\Big)\dd x,\\
					J_2=&~\frac{1}{\delta}\int_0^T\int_\Omega\big(L\partial_ty_N-f(\partial_ty_N)\big)\partial_tv_N\dd x\dd t\\ &+\int_0^T\int_\Omega\big(L-f'(\partial_ty_N)\big)\nabla\partial_ty_N\cdot\nabla \partial_tv_N\dd x\dd t.				
				\end{split}				
			\end{equation}
			Since $f$ is Lipschitz continuous,  its derivative $f'$ exists almost everywhere. From conditions \eqref{Lipschitz} and \eqref{tilde L condition}, we have $\tilde{L}\leq f'\leq L$, hence,  Young's inequality yields
			
			\begin{equation}\label{integration terms}
				\begin{aligned}
					|J_1| \leq & ~ \frac{\delta_1}{2\delta(L-\tilde{L})}\int_0^T\int_\Omega\big|L\partial_ty_N-f(\partial_ty_N)\big|^2\dd x\dd t+\frac{L-\tilde{L}}{2\delta\delta_1}\int_0^T\int_\Omega |\partial_tv_N|^2\dd x\dd t \\
					&+\frac{\delta_2(L-\tilde{L})}{2}\int_0^T\int_\Omega\big|\nabla\partial_ty_N\big|^2\dd x\dd t+\frac{L-\tilde{L}}{2\delta_2}\int_0^T\int_\Omega\big|\nabla\partial_tv_N\big|^2\dd x\dd t \\
					\leq & ~ \frac{\delta_1(L-\tilde{L})}{2\delta}\int_0^T\int_\Omega|\partial_ty_N|^2\dd x\dd t+\frac{L-\tilde{L}}{2\delta\delta_1}\int_0^T\int_\Omega|\partial_tv_N|^2\dd x\dd t \\
					&+\frac{\delta_2(L-\tilde{L})}{2}\int_0^T\int_\Omega\big|\nabla\partial_ty_N\big|^2\dd x\dd t+\frac{L-\tilde{L}}{2\delta_2}\int_0^T\int_\Omega\big|\nabla\partial_tv_N\big|^2\dd x\dd t,
				\end{aligned}
			\end{equation}
			and
			\begin{equation}
				|J_2|\leq\frac{\delta_1L}{\delta(L-\tilde{L})}E_0(y_N(0))+\frac{L-\tilde{L}}{4\delta\delta_1L}E_0(v_N(0))+\frac{\delta_2L}{L-\tilde{L}}E_1(y_N(0))+ \frac{L-\tilde{L}}{4\delta_2L}E_1(v_N(0)),
			\end{equation}
			where $\delta_1>0$ and $\delta_2>0$ are  constants to be determined later.
			
			Next, to control the right-hand side of \eqref{integration terms}, we make the standard energy estimate of  $y_N$ and $v_N$.  Multiplying equation \eqref{finite dimision yN} by $g_{iNt}(t)$, adding them together, and summing this over $i=1,\cdots, N$, we obtain the energy estimate of $y_N$
			\begin{equation}\label{energy estimate 0}
				\frac{1}{2}\frac{\dd}{\dd t}E_1(y_N(t))+\int_\Omega f(\partial_ty_N)\partial_ty_N\dd x=\int_\Omega\chi\partial_ty_N\partial_tv_N\dd x.
			\end{equation}
			
			Integrating \eqref{energy estimate 0} from $0$ to $T$ with respect to $t$, we get
			\begin{equation}
				\begin{aligned}
					& \tilde{L}\int_0^T\int_\Omega|\partial_ty_N|^2\dd x\dd t \\
					\leq & ~ \int_0^T\int_\Omega f(\partial_ty_N)\partial_ty_N\dd x\dd t \\
					= & ~ \frac{1}{2}E_1(y_N(0))-\frac{1}{2}E_1(y_N(T))+\int_0^T\int_\Omega\chi\partial_ty_N\partial_tv_N\dd x\dd t \\
					\leq & ~ \frac{1}{2}E_1(y_N(0))+\frac{\tilde{L}}{2}\int_0^T\int_\Omega|\partial_ty_N|^2\dd x\dd t+\frac{1}{2\tilde{L}}\int_0^T\int_\Omega\chi^2|\partial_tv_N|^2\dd x\dd t.
				\end{aligned}
			\end{equation}
			We then obtain
			\begin{equation}\label{estimate for yN1}
				\int_0^T\int_\Omega|\partial_ty_N|^2\dd x\dd t\leq\frac{1}{\tilde{L}}E_1(y_N(0))+\frac{1}{\tilde{L}^2}\int_0^T\int_\Omega\chi^2|\partial_tv_N|^2\dd x\dd t.
			\end{equation}
			
			Multiplying equation  \eqref{finite dimision yN} by $\lambda_i^2g'_{iN}(t)$, adding them together, and summing this over $i=1,\cdots, N$, we get
			\begin{equation}\label{energy estimate 1}
				\frac{1}{2}\frac{\dd}{\dd t}E_2(y_N(t))+\int_\Omega f'(\partial_ty_N)|\nabla\partial_t y_N|^2\dd x=\int_\Omega\nabla \partial_ty_N\cdot\big(\chi\nabla \partial_tv_N+\partial_tv_N\nabla\chi\big)\dd x.
			\end{equation}
			
			Integrating \eqref{energy estimate 1} from $0$ to $T$ with respect to $t$, we   obtain
			\begin{equation}\label{estimate for yN2}
				\begin{aligned}
					& \int_0^T\int_\Omega\big|\nabla\partial_ty_N\big|^2\dd x\dd t \\
					\leq & ~ \frac{1}{\tilde{L}}E_2(y_N(0))+\frac{1}{\tilde{L}^2}\int_0^T\int_\Omega\Big|\chi\nabla\partial_tv_N+\partial_tv_N\nabla\chi\Big|^2\dd x\dd t \\
					\leq & ~ \frac{1}{\tilde{L}}E_2(y_N(0))+\frac{2}{\tilde{L}^2}\int_0^T\int_\Omega\Big(\chi^2\big|\nabla\partial_tv_N\big|^2+|\nabla\chi|^2|\partial_tv_N|^2\Big)\dd x\dd t.
				\end{aligned}
			\end{equation}
			
			Similarly, multiplying the equation \eqref{finite dimision vN} by $h'_{iN}(t)$ and $\lambda_i h'_{iN}(t)$ respectively, and following a similar process to the estimates above for $y_N$, we can obtain:
			\begin{equation}\label{estimate for vN1}
				\int_0^T\int_\Omega|\partial_tv_N|^2\dd x\dd t=\frac{1}{2L}\left(E_1(v_N(T))-E_1(v_N(0))\right)
			\end{equation}			
			and
			\begin{equation}\label{estimate for vN2}
				\int_0^T\int_\Omega\big|\nabla\partial_tv_N\big|^2\dd x\dd t=\frac{1}{2L}\left(E_2(v_N(T))-E_2(v_N(0))\right).
			\end{equation}
			These two equations show that $E_i(v_N)$ for $i =1, 2$ is non-increasing.
			
			Combining  \eqref{integration terms} with \eqref{estimate for yN1}--\eqref{estimate for vN2}, we obtain
			\begin{equation}			
				\begin{aligned}
					& ~ |J_1|+|J_2| \\
					\leq & ~ \frac{L-\tilde{L}}{4\delta\delta_1L}E_0(v_N(T))+\frac{\delta_1(L-\tilde{L})}{2\delta\tilde{L}^2}\int_0^T\int_\Omega\chi^2|\partial_tv_N|^2\dd x\dd t \\
					&+\frac{L-\tilde{L}}{4\delta_2L}E_2(v_N(T))+\frac{\delta_2(L-\tilde{L})}{\tilde{L}^2}\int_0^T\int_\Omega\Big(\chi^2\big|\nabla\partial_tv_N\big|^2+|\nabla\chi|^2|\partial_t v_N|^2\Big)\dd x\dd t \\
					&+\left(\frac{\delta_1(L-\tilde{L})}{2\delta\tilde{L}}+\frac{\delta_1L}{\delta(L-\tilde{L})}\right)E_1(y_N(0))+\left(\frac{\delta_2(L-\tilde{L})}{2\tilde{L}}+\frac{\delta_2L} {L-\tilde{L}}\right)E_2(y_N(0)).
				\end{aligned}
			\end{equation}		
			Hence  by \eqref{3111}, this implies that
			\begin{align}\label{327}
					& \frac{1}{\delta}\int_\Omega\Big(\partial_ty_N(T)\partial_tv_N(T)+\nabla y_N(T)\cdot\nabla v_N(T)\Big)\dd x \nonumber\\
					&+ \int_\Omega\Big(\nabla\partial_ty_N(T)\cdot\nabla\partial_tv_N(T)+\Delta y_N(T)\Delta v_N(T)\Big)\dd x \nonumber\\
					\geq & ~ \int_0^T\int_\Omega\Big(\chi-\frac{\delta_2(L-\tilde{L})}{\tilde{L}^2}\chi^2\Big)\big|\nabla\partial_tv_N\big|^2\dd x\dd t+\int_0^T\int_\Omega\frac{1}{\delta}\Big(\chi- \frac{\delta_1(L-\tilde{L})}{2\tilde{L}^2}\chi^2\Big)|\partial_tv_N|^2\dd x\dd t \nonumber\\
					&-\int_0^T\int_\Omega\Big(\frac{\Delta\chi}{2}+\frac{\delta_2(L-\tilde{L})}{\tilde{L}^2}|\nabla\chi|^2\Big)|\partial_tv_N|^2\dd x\dd t-\frac{L-\tilde{L}}{4\delta\delta_1L}E_1(v_N(T)) -\frac{L-\tilde{L}}{4\delta_2L}E_2(v_N(T)) \nonumber\\
					&-\left(\frac{\delta_1(L-\tilde{L})}{2\delta\tilde{L}}+\frac{\delta_1L}{\delta(L-\tilde{L})}\right)E_1(y_N(0))-\left(\frac{\delta_2(L-\tilde{L})}{2\tilde{L}}+\frac{\delta_2L} {L-\tilde{L}}\right)E_2(y_N(0)).
				\end{align}
			
			Thus, it is now time to estimate each term on the right-hand side of \eqref{327}. We aim to obtain the  lower bounds for the first two positive terms and the upper bounds for the last five negative terms.
			
			First, let us take			
			\begin{equation}
				\delta_1=\tilde{L}\sqrt{\frac{D}{L}},\quad \delta_2=\tilde{L}\sqrt{\frac{D}{2L}}.
                \end{equation}
			Recalling the assumption \eqref{condition of D L and tilde L} on $L,D$ and $\tilde{L}$,	we find that
			\begin{equation}
				0<\frac{\delta_2(L-\tilde{L})}{\tilde{L}^2}\leq \frac{(L-\tilde{L})}{\tilde{L}}<1, \; 0<\frac{\delta_1 (L-\tilde{L})}{2\tilde{L}^2}\leq\frac{(L-\tilde{L})}{2\tilde{L}}<\frac{1}{2}.
			\end{equation}
			Together with the definition of $\chi$, this implies that
			\begin{equation}
            \chi-\frac{\delta_2(L-\tilde{L})}{\tilde{L}^2}\chi^2\geq \Big(1-\frac{(L-\tilde{L})}{\tilde{L}}\Big)\chi, \quad \chi-\frac{\delta_1 (L-\tilde{L})}{2\tilde{L}^2}\chi^2\geq \Big(1-\frac{(L-\tilde{L})}{2\tilde{L}} \Big)\chi.
			\end{equation}
			Thus, we obtain the lower bounds for the first two terms.
			
			Next, using the estimate from \eqref{estimate for vN1}, we find
			\begin{equation}
				\begin{aligned}
					& \int_0^T\int_\Omega\Big(\frac{\Delta\chi}{2}+\frac{\delta_2(L-\tilde{L})}{\tilde{L}^2}|\nabla\chi|^2\Big)|\partial_tv_N|^2\dd x\dd t \\
					\leq & ~ \frac{1}{2L}\Big(\frac{\|\Delta\chi\|_{L^\infty}}{2}+\frac{\delta_2(L-\tilde{L})}{\tilde{L}^2}\|\nabla\chi\|_{L^\infty}^2\Big)E_1(v_N(T)).
				\end{aligned}
			\end{equation}
			
			Furthermore, under the same assumptions as in Theorem \ref{thm:semilinear}, it appears that the assumptions of Lemma \ref{lem:the example prop} are also satisfied. Therefore, we have the following two observability inequalities, \eqref{observability inequality semilinear1} and \eqref{observability inequality semilinear2}, which lead to
			\begin{align}\label{ready to use the lemma}
				& ~ \frac{1}{\delta}\int_\Omega\Big(\partial_ty_N(T)\partial_tv_N(T)+\nabla y_N(T)\cdot\nabla v_N(T)\Big)\dd x \nonumber \\
				& +\int_\Omega\Big(\nabla\partial_ty_N(T)\cdot\nabla\partial_tv_N(T)+\Delta y_N(T)\Delta v_N(T)\Big)\dd x \nonumber \\
				\geq & ~ \Big(1-\frac{(L-\tilde{L})}{\tilde{L}}\sqrt{\frac{D}{2L}}\Big)\int_0^T\int_\omega\big|\nabla\partial_tv_N\big|^2\dd x\dd t-\frac{L-\tilde{L}}{2\tilde{L}\sqrt{2DL}}E_1(v_N (T)) \nonumber \\
				& +\frac{1}{\delta}\Big(1-\frac{(L-\tilde{L})}{2\tilde{L}}\sqrt{\frac{D}{L}}\Big)\int_0^T\int_\omega|\partial_tv_N|^2\dd x\dd t-\frac{L-\tilde{L}}{4\delta\tilde{L}\sqrt{DL}}E_1(v_N(T)) \nonumber \\
				& -\Big(\frac{\|\Delta\chi\|_{L^\infty}}{4L}+\frac{(L-\tilde{L})}{2L\tilde{L}}\sqrt{\frac{D}{2L}}\|\nabla\chi\|_{L^\infty}^2\Big)E_1(v_N(T)) \\
				& -\frac{\tilde{L}}{\delta}\sqrt{\frac{D}{L}}\Big(\frac{L-\tilde{L}}{2\tilde{L}}+\frac{L}{L-\tilde{L}}\Big)E_1(y_N(0))-\tilde{L}\sqrt{\frac{D}{2L}}\Big(\frac{L-\tilde{L}}{2 \tilde{L}}+\frac{L}{L-\tilde{L}}\Big)E_2(y_N(0)) \nonumber \\
				\geq & ~ \left(\frac{1}{2D}-\frac{L-\tilde{L}}{\tilde{L}\sqrt{2DL}}\right)E_2(v_N(T))+\frac{1}{\delta}\left(\frac{1}{2D}-\frac{L-\tilde{L}}{2\tilde{L}\sqrt{DL}}\right)E_1(v_N (T)) \nonumber \\
				& -\Big(\frac{\|\Delta\chi\|_{L^\infty}}{4L}+\frac{(L-\tilde{L})}{2L\tilde{L}}\sqrt{\frac{D}{2L}}\|\nabla\chi\|_{L^\infty}^2\Big)E_1(v_N(T)) \nonumber \\
				& -\frac{\tilde{L}}{\delta}\sqrt{\frac{D}{L}}\Big(\frac{L-\tilde{L}}{2\tilde{L}}+\frac{L}{L-\tilde{L}}\Big)E_1(y_N(0))-\tilde{L}\sqrt{\frac{D}{2L}}\Big(\frac{L-\tilde{L}}{2 \tilde{L}}+\frac{L}{L-\tilde{L}}\Big)E_2(y_N(0)). \nonumber
			\end{align}
			
			{\color{black} Let
				\begin{equation}
					c_1:=\left(\frac{1}{2D}-\frac{L-\tilde{L}}{\tilde{L}\sqrt{2DL}}\right), \; c_2:=\frac{1}{2\delta}\left(\frac{1}{2D}-\frac{L-\tilde{L}}{2\tilde{L}\sqrt{DL}}\right).
				\end{equation}
				Recalling assumption \eqref{condition of D L and tilde L}, we can verify that $c_1 > 0$, $c_2 > 0$. It is possible to choose $\delta$ small enough such that
				\begin{equation}\label{delta-similinear}		\delta\Big(\frac{\|\Delta\chi\|_{L^\infty}}{4L}+\frac{(L-\tilde{L})}{2L\tilde{L}}\sqrt{\frac{D}{2L}}\|\nabla\chi\|_{L^\infty}^2\Big)\leq\frac{c_1}{2}.
				\end{equation}

				Subsequently, if
				\begin{align}\label{112}
						&c_1 E_2(v_N(T))+c_2 E_1(v_N (T)) \nonumber\\
						\geq ~& \frac{\tilde{L}}{\delta}\sqrt{\frac{D}{L}}\Big(\frac{L-\tilde{L}}{2\tilde{L}}+\frac{L}{L-\tilde{L}}\Big)E_1(y_N(0))+\tilde{L}\sqrt{\frac{D}{2L}}\Big(\frac{L-\tilde{L}}{2 \tilde{L}}+\frac{L}{L-\tilde{L}}\Big)E_2(y_N(0))   \nonumber\\
						=:~&d_1E_1(y_N(0))+d_2E_2(y_N(0))
				\end{align}
				is valid, we can derive
				\begin{equation}\label{innerproduct}
					\begin{aligned}
						(B_N\vec{l},A_NF_N(x_0))_{\tilde{l}_2}=~	& \frac{1}{\delta}\int_\Omega\Big(\partial_ty_N(T)\partial_tv_N(T)+\nabla y_N(T)\cdot\nabla v_N(T)\Big)\dd x \\
						 & + \int_\Omega\Big(\nabla\partial_ty_N(T)\cdot\nabla\partial_tv_N(T)+\Delta y_N(T)\Delta v_N(T)\Big)\dd x\geq 0.
					\end{aligned}
				\end{equation}

				Since $\{\varphi_j\}_{j=1}^\infty$ is the standard orthogonal basis of $L^2(\Omega)$ satisfying \eqref{elliptic1}, we have
				\begin{align}
			&c_1 E_2(v_N(T))+c_2 E_1(v_N (T))\nonumber\\
                =~&\sum_{i=1}^N \left((c_1 \lambda_i^4+c_2\lambda_i^2)|a_i|^2+(c_1\lambda_i^2+c_2)|b_i|^2\right)\nonumber \\						
			\geq~& \min\{c_1,c_2\}\sum_{i=1}^N \left( \lambda_i^4|a_i|^2+\lambda_i^2|b_i|^2\right)=\min\{c_1,c_2\} |\Lambda_N \vec{l}|_{\ell_2}^2.	
				\end{align}
				Recalling the initial energy upper bound condition \eqref{energydecay}, we then have
				\begin{equation}\label{111}
					d_1E_1(y_N(0))+d_2E_2(y_N(0))\leq \max\{d_1,d_2\} (E_1(y(0))+E_2(y(0))).
				\end{equation}
				Hence, we define
				\begin{equation}
					R:=\sqrt{\frac{\max\{d_1,d_2\} (E_1(y(0))+E_2(y(0)))}{\min\{c_1,c_2\}}},
				\end{equation}
				and therefore if
				$
				|\Lambda_N \vec{l}|_{\ell_2}\geq R,
				$
				then \eqref{FNl} holds. Moreover, $R$ is independent of $N$.		 }
		\end{proof}
		
		Now we are in a position to prove Theorem \ref{thm:semilinear}.
		
		\begin{proof} For any $N > 0$, by Lemma \ref{lem:Galerkin}, there exists a $\vec{l}_N=(a_1,\cdots,a_N,b_1,\cdots,b_N)$ satisfying
			$\mathcal{F}_N(\vec{l}_N)=0$. Thanks to the  definition of $\mathcal{F}_N$, this indeed implies that $(y_N(T),y_{Nt}(T))=(0,0)$. Then we get
			\begin{equation}
				\begin{aligned}
					&\frac{1}{\delta}\int_\Omega\Big(\partial_ty_N(T)\partial_tv_N(T)+\nabla y_N(T)\cdot\nabla v_N(T)\Big)\dd x \\
					+ & ~ \int_\Omega\Big(\nabla\partial_ty_N(T)\cdot\nabla\partial_tv_N(T)+\Delta y_N(T)\Delta v_N(T)\Big)\dd x=0.
				\end{aligned}
			\end{equation}

			Thus, referring to \eqref{ready to use the lemma}, we find that
			\begin{align}\label{estimate of similinear approximate control}
				& \delta\int_0^T\int_\omega\big|\nabla\partial_tv_N\big|^2\dd x\dd t+\frac{1}{2}\int_0^T\int_\omega|\partial_tv_N|^2\dd x\dd t \nonumber \\
				\leq ~ & C^*\left(E_0(y_N(0))+\delta E_1(y_N(0))\right) \\
				\leq ~ & C^*\left(E_0(y(0))+\delta E_1(y(0))\right), \nonumber
			\end{align}
			where the constant is given by
			\begin{equation}\label{C*}
				\begin{aligned}
					C^*= & ~ \sqrt{\frac{D}{L}}\Big(\frac{L-\tilde{L}}{2}+\frac{L\tilde{L}}{L-\tilde{L}}\Big)\Big(1-\frac{L-\tilde{L}}{\tilde{L}}\sqrt{\frac{2D}{L}}\Big)^{-1} \\
					= & ~ \frac{L^2+\tilde{L}^2}{2(L-\tilde{L})}\frac{\tilde{L}\sqrt{D}}{\tilde{L}\sqrt{L}-(L-\tilde{L})\sqrt{2D}}.
				\end{aligned}
			\end{equation}
			
			It follows from \eqref{estimate of similinear approximate control} that $\{\partial_t v_N\}_{N=1}^\infty$ is bounded in $L^2(0,T;\mathcal{H}^1)$, and hence there exists a subsequence that converges weakly. Furthermore, by the energy estimates \eqref{energy estimate 0} and \eqref{energy estimate 1},
			\begin{equation}\label{y_nbounded}
				\begin{aligned}
					\{y_N\}_{N=1}^\infty & \subset L^\infty(0,T;H^2(\Omega)\cap H_0^1(\Omega)), \\
					\big\{\partial_ty_N\big\}_{N=1}^\infty & \subset L^\infty(0,T;H_0^1(\Omega)), \\
					\big\{f\big(\partial_ty_N\big)\big\}_{N=1}^\infty & \subset L^\infty(0,T;H_0^1(\Omega)),
				\end{aligned}
			\end{equation}
			are bounded sequences. From the system of $y_N$, \eqref{y_nbounded} infers $\{\partial_t^2 y_N\}_{N=1}^\infty \subset L^\infty(0,T;L^2(\Omega))$. Therefore, we can extract a subsequence of $\{y_N\}$ (still denoted as $\{y_N\}$), such that there exist $y \in L^\infty(0,T;H^2 \cap H_0^1)$, $z \in L^\infty(0,T;H_0^1)$, and
			\begin{equation}
				\begin{cases}
					y_N\stackrel{\ast}{\longrightarrow}y, & \quad \text{in} ~ L^\infty(0,T;H^2(\Omega)\cap H_0^1(\Omega)), \\
					\partial_ty_N\stackrel{\ast}{\longrightarrow}y_t, & \quad \text{in} ~ L^\infty(0,T;H_0^1(\Omega)), \\
					\partial_t^2y_N\stackrel{\ast}{\longrightarrow}y_{tt}, & \quad \text{in} ~ L^\infty(0,T;L^2(\Omega)), \\
					f\big(\partial_ty_N\big)\stackrel{\ast}{\longrightarrow}z, & \quad \text{in} ~ L^\infty(0,T;H_0^1(\Omega)).
				\end{cases}
			\end{equation}
			On the other hand, by a compactness argument (refer to \cite{S.Zheng}), we have
			\begin{align}
				\partial_ty_N\longrightarrow y_t, \quad \text{in} ~ L^2(0,T;L^2(\Omega)),
			\end{align}
			and
			\begin{equation}
				\begin{cases}
					y_N(T)\longrightarrow y(T), \quad \text{in} ~ \mathcal{H}^1, \\	\partial_ty_N(T)\longrightarrow y_t(T), \quad \text{in} ~ L^2(\Omega).
				\end{cases}
			\end{equation}
			
			Thus, given that $f$ is a Lipschitz function,
			\begin{align*}
				f\big(\partial_ty_N\big)\longrightarrow f(y_t), \quad \text{in} ~ L^2(0,T;L^2(\Omega)).
			\end{align*}
			
			By the uniqueness of the limit, we conclude that $z = f(y_t)$, implying
			\begin{align*}
				f\big(\partial_ty_N\big)\stackrel{\ast}{\longrightarrow}f(y_t), & \quad \text{in} ~ L^\infty(0,T;H_0^1(\Omega)).
			\end{align*}

			Consequently, the approximation solutions $\{y_N\}_{N=1}^\infty$ converge to a weak solution $y\in C^0((0,T];\mathcal{H}^1)\cap C^1((0,T];L^2)$ of \eqref{semilinear damped wave control} in the sense of $L^2([0,T]; L^2)$. Moreover, the weak limit $u$ of $\{\partial_t v_N\}_{N=1}^\infty$ is the desired control function. Letting $N \to \infty$ in \eqref{estimate of similinear approximate control}, we obtain \eqref{estimate of similinear control function} with $D^*=\frac{C^*}{\delta}$, where $\delta$ and $C^*$ are defined in \eqref{delta-similinear} and \eqref{C*}.
		\end{proof}
		
		\section{Proof of Theorem \ref{thm:main1}}\label{sec:4}
		
		This section is devoted to proving Theorem \ref{thm:main1}. As mentioned in the introduction part, it is sufficient to consider the null controllability problem for quasi-linear damped wave equation
		\begin{equation}\label{system:quasilinear}
			\left\{
			\begin{aligned}
				& y_{tt}+b_0y_t-{\color{black}\sum\limits_{i,j=1}^n\big(a_{ij}y_{x_i}\big)_{x_j}}+\sum\limits_{k=1}^nb_ky_{x_k}+\tilde{b}y=\chi_\omega u, & (t,x)\in & ~ (0,T)\times\Omega, \\
				& y(t,x)=0, & (t,x)\in & ~ (0,T)\times\partial\Omega, \\
				& y(0,x)=y^0, ~ y_t(0,x)=y^1, & x\in & ~ \Omega,
			\end{aligned}
			\right.
		\end{equation}
		with
		\begin{equation} \label{coe:2}
			\begin{cases}
				 a_{ij}=a_{ji}=\delta_{ij}-g^{ij}_2(t,x, y, y_t,\nabla y), \quad i,j=1,\cdots,n,\\
                 b_0=b_0(t,x, y, y_t,\nabla y)=1+\int_0^1\frac{\partial g_1}{\partial y_t}(t,x,\tau y,\tau y_t,\tau\nabla y)\dd\tau, \\
				 b_k=b_k(t,x, y, y_t,\nabla y)=\int_0^1\frac{\partial g_1}{\partial y_{x_k}}(t,x,\tau y,\tau y_t,\tau\nabla y)\dd\tau, \quad k=1,\cdots,n \\
                \tilde{b}= \tilde{b}(t,x, y, y_t,\nabla y)=\int_0^1\frac{\partial g_1}{\partial y}(t,x,\tau y,\tau y_t,\tau\nabla y)\dd\tau.
			\end{cases}
		\end{equation}
				
		\subsection{Existence of solutions of system \eqref{system:quasilinear}}

		Considering the damping term $y_t$, we aim to develop an algorithmic framework that not only establishes the existence of the solution to \eqref{system:quasilinear} but also achieves null controllability for the system described by \eqref{system:quasilinear}. We start by focusing on the linearized version of the system \eqref{system:quasilinear}. To this end, we introduce the following iterative procedure: We initialize with $(z^{(0)},v^{(0)})\equiv (0,0)$. For each $\alpha\geq 1$, given the previous iteration $(z^{(\alpha-1)},v^{(\alpha-1)})$, we define the next iteration $(z^{(\alpha)},v^{(\alpha)})$ as detailed below.
		\begin{equation}\label{eqnofzalpha}
			\left\{
			\begin{aligned}
				& z^{(\alpha)}_{tt}-b^{(\alpha)}_0z^{(\alpha)}_t-{\color{black}\sum\limits_{i,j=1}^n\big(a^{(\alpha)}_{ij}z^{(\alpha)}_{x_i}\big)_{x_j}}+\tilde{b}^{(\alpha)}z^{(\alpha)} & & \\
				& ~ +\sum\limits_{i=1}^n b_i^{(\alpha)}z^{(\alpha)}_{x_i}=0, & (t,x)\in & ~ (0,T)\times\Omega, \\
				& z^{(\alpha)}(t,x)=0, & (t,x)\in & ~ (0,T)\times\partial\Omega, \\
				& z^{(\alpha)}(T,x)=v^{(\alpha-1)}(T,x)+z^{(\alpha-1)}(T,x), & x\in & ~ \Omega, \\
				& z^{(\alpha)}_t(T,x)=v_t^{(\alpha-1)}(T,x)+z_t^{(\alpha-1)}(T,x), & x\in & ~ \Omega,
			\end{aligned}
			\right.
		\end{equation}
		and
\begin{equation}\label{eqnofvalpha}
			\left\{
			\begin{aligned}
				& v^{(\alpha)}_{tt}+b^{(\alpha)}_0v^{(\alpha)}_t-{\color{black}\sum\limits_{i,j=1}^n\big(a^{(\alpha)}_{ij}v^{(\alpha)}_{x_i}\big)_{x_j}}+\tilde{b}^{(\alpha)}v^{(\alpha)} & & \\
				& ~ +\sum\limits_{i=1}^n b^{(\alpha)}_iv^{(\alpha)}_{x_i}=-2\chi\cdot z^{(\alpha)}_t, & (t,x)\in & ~ (0,T)\times\Omega, \\
				& v^{(\alpha)}(t,x)=0, & (t,x)\in & ~ (0,T)\times\partial\Omega, \\
				& v^{(\alpha)}(0,x)=y^0, ~ v^{(\alpha)}_t(0,x)=y^1, & x\in & ~ \Omega
			\end{aligned}
			\right.
		\end{equation}
		where
		\begin{equation}\label{coeabb}
			\begin{cases}
				& a_{ij}^{(\alpha)}=a_{ij}(t,x,v^{(\alpha-1)},v_t^{(\alpha-1)},\nabla v^{(\alpha-1)}), \quad i,j=1,\cdots,n,\\
				& b^{(\alpha)}_i=b_i(t,x,v^{(\alpha-1)},v_t^{(\alpha-1)},\nabla v^{(\alpha-1)}), \quad i=0,\cdots,n, \\
				& \tilde{b}^{(\alpha)}=\tilde{b}(t,x,v^{(\alpha-1)},v_t^{(\alpha-1)},\nabla v^{(\alpha-1)}).
			\end{cases}
		\end{equation}
			We provide some remarks on the assumptions made on the coefficients.
\begin{rem}
			Thanks to the assumptions on $g_2^{ij}$ and $g_1$, we have the following  relations: as $|y|+|\nabla y|+|y_t|\rightarrow 0$, for any $(t,x)\in \overline{(0,T)\times \Omega}$,
			\begin{align} \label{OOO}
				a_{ij}=~&\delta_{ij}+O(|y|+|\nabla y|+|y_t|), \quad b_0=1+ O(|y|+|\nabla y|+|y_t|), \nonumber \\
				b_k=~& O(|y|+|\nabla y|+|y_t|), \quad \tilde{b}= O(|y|+|\nabla y|+|y_t|).
				\end{align}		
		\end{rem}


	\begin{rem}\label{rem:alphabound}
		With the help of \eqref{OOO}, the recurrence relation \eqref{coeabb} can be equivalently written as
		\begin{equation}\begin{split}
			a_{ij}^{(\alpha)}&=\delta_{ij}+a_{ij,0}v^{(\alpha-1)}_t+a_{ij,k}v^{(\alpha-1)}_{x_k}+a_{ij,n+1}v^{(\alpha-1)}, i,j,k=1,\cdots,n\\
				b_0^{(\alpha)}&=1+b_{0,0}v^{(\alpha-1)}_t+b_{0,i}v^{(\alpha-1)}_{x_i}+b_{0,n+1}v^{(\alpha-1)},   i=1,\cdots,n\\
				b_k^{(\alpha)}&=b_{k,0}v^{(\alpha-1)}_t+a_{k,i}v^{(\alpha-1)}_{x_i}+b_{k,n+1}v^{(\alpha-1)}, k,i=1,\cdots,n;\\
				 \tilde{b}^{(\alpha)}&= \tilde{b}_{0}v^{(\alpha-1)}_t+\tilde{b}_{i}v_{x_i}^{(\alpha-1)}+\tilde{b}_{n+1}v^{(\alpha-1)}, k,i=1,\cdots,n\\
			\end{split}
		\end{equation} 	
		where $a_{ij,k}, i,j=1,\cdots,n,  k=0,1,\cdots,n+1$, $b_{i,k}, i,k=0, 1,\cdots,n+1$ and $\tilde{b}_{i}, i=0,1,\cdots, n+1$ are smooth bounded functions.
		
	\end{rem}
		
    Now we state the following proposition:

   \begin{proposition}
			\label{prop:vz}
			Let the sequences $v^{(\alpha)}$ and $z^{(\alpha)}$ be the solutions of \eqref{eqnofvalpha} and \eqref{eqnofzalpha}, respectively. Under the same assumptions as in Theorem \ref{thm:main1},			
			there exists a constant $\varepsilon_{prop}>0$, such that the norm condition for initial data $(y^0,y^1)$ is satisfied:
			\begin{equation*}
				\|y^0\|_{H^s}+\|y^1\|_{H^{s-1}}\leq\varepsilon_{prop},
			\end{equation*}
			where $s\geq \max\{n+2,4\}$.
			Then for any $t\in [0,T]$, we have that as $\alpha \rightarrow \infty$,
			\begin{align*}
				(v^{(\alpha)}(t),v^{(\alpha)}_t(t),v^{(\alpha)}_{tt}(t))&\rightarrow \big(y(t),y_t(t),y_{tt}(t)\big), \quad \text{in} ~ \mathcal{H}^{s-1} \times \mathcal{H}^{s-2}\times \mathcal{H}^{s-3},\\
				(z^{(\alpha)}_t(t),z^{(\alpha)}_{tt}(t))&\rightarrow (u(t),u_t(t)), \quad \text{in} ~ \mathcal{H}^{s-2}\times \mathcal{H}^{s-3},
			\end{align*}
			where limit functions $y\in \cap_{p=0}^2C^p(0,T;\mathcal{H}^{s-p})$ and $u\in \cap_{p=0}^1C^p(0,T;\mathcal{H}^{s-1-p})$ are solutions to the quasilinear system \eqref{system:quasilinear}, subject to the terminal conditions $$(y(T),y_t(T))=(0,0).$$
		\end{proposition}
		
		Notice that the existence part in Theorem \ref{thm:main1} follows from Proposition \ref{prop:vz} directly.

	    In order to obtain the convergence properties of sequences $v^{(\alpha)}$ and $z^{(\alpha)}$, we need to estimate the error of iteration $(v^{(\alpha)},z^{(\alpha)})$, thus,
		we define
		\begin{equation}\label{VZ}
			V^{(\alpha)}=v^{(\alpha)}-v^{(\alpha-1)}, ~ Z^{(\alpha)}=z^{(\alpha)}-z^{(\alpha-1)},\end{equation}
 Then we get $V^{(1)}=v^{(1)}, Z^{(1)}=z^{(1)}$ and for each $\alpha\geq 2$, sequence $V^{(\alpha)}$ and $Z^{(\alpha)}$ solve the equations
		\begin{equation}\label{eqnofValpha}
			\left\{
			\begin{aligned}
				& V^{(\alpha)}_{tt}+b_0^{(\alpha)}V^{(\alpha)}_t-{\color{black}\sum\limits_{i,j=1}^n\big(a_{ij}^{(\alpha)}V^{(\alpha)}_{x_i}\big)_{x_j}}+\tilde{b}^{(\alpha)}V^{(\alpha)} & & \\
				& ~~~~~~~~~~~~~~~~~~~~ +\sum\limits_{i=1}^n b_i^{(\alpha)}V^{(\alpha)}_{x_i}=F^{(\alpha)}-2\chi\cdot Z_t^{(\alpha)}, & (t,x)\in & ~ (0,T)\times\Omega, \\
				& V^{(\alpha)}(t,x)=0, & (t,x)\in & ~ (0,T)\times\partial\Omega, \\
				& V^{(\alpha)}(0,x)=0, ~ V^{(\alpha)}_t(0,x)=0, & x\in & ~ \Omega
			\end{aligned}
			\right.
		\end{equation}
		and
		\begin{equation}\label{eqnofZalpha}
			\left\{
			\begin{aligned}
				& Z^{(\alpha)}_{tt}-b_0^{(\alpha)}Z^{(\alpha)}_t-{\color{black}\sum\limits_{i,j=1}^n\big(a_{ij}^{(\alpha)}Z^{(\alpha)}_{x_i}\big)_{x_j}}+\tilde{b}^{(\alpha)}Z^{(\alpha)} & & \\
				& ~~~~~~~~~~~~~~~~~~~~ +\sum\limits_{i=1}^n b_i^{(\alpha)}Z^{(\alpha)}_{x_i}=H^{(\alpha)}, & (t,x)\in & ~ (0,T)\times\Omega, \\
				& Z^{(\alpha)}(t,x)=0, & (t,x)\in & ~ (0,T)\times\partial\Omega, \\
				& Z^{(\alpha)}(T,x)=v^{(\alpha-1)}(T,x), ~ Z^{(\alpha)}_t(T,x)=v^{(\alpha-1)}_t(T,x) & x\in & ~ \Omega,
			\end{aligned}
			\right.
		\end{equation}
		where
		\begin{equation}\label{Gdefn}
			\begin{aligned}
				F^{(\alpha)}= & \big(b_0^{(\alpha)}-b_0^{(\alpha-1)}\big)v^{(\alpha-1)}_t-\sum\limits_{i,j=1}^n\big[\big(a_{ij}^{(\alpha)}-a_{ij}^{(\alpha-1)}\big)v^{(\alpha-1)}_{x_i}\big]_{x_j}  \\
				& +\big(\tilde{b}^{(\alpha)}-\tilde{b}^{(\alpha-1)}\big)v^{(\alpha-1)}+\sum\limits_{i=1}^n\big(b_i^{(\alpha)}-b_i^{(\alpha-1)}\big)v^{(\alpha-1)}_{x_i},
			\end{aligned}
		\end{equation}

		and
		\begin{equation}\label{Hdefn}
			\begin{aligned}
				H^{(\alpha)}= & -\big(b_0^{(\alpha)}-b_0^{(\alpha-1)}\big)z^{(\alpha-1)}_t-\sum\limits_{i,j=1}^n\big[\big(a_{ij}^{(\alpha)}-a_{ij}^{(\alpha-1)}\big)z^{(\alpha-1)}_{x_i}\big]_{x_j}  \\
				& +\big(\tilde{b}^{(\alpha)}-\tilde{b}^{(\alpha-1)}\big)z^{(\alpha-1)}+\sum\limits_{i=1}^n\big(b_i^{(\alpha)}-b_i^{(\alpha-1)}\big)z^{(\alpha-1)}_i.
			\end{aligned}
		\end{equation}

    The key of the proof of Proposition \ref{prop:vz} is the following estimates: 	
\begin{lemma}\label{lem:VZ}
Let $s\geq \max\{n+2,4\}$ be an integer. There exists a small $\varepsilon_{lem}>0$ and $0<\delta<1$, such that for each $\varepsilon \leq \varepsilon_{lem}$ and for all $\alpha\geq 1$, System \eqref{eqnofzalpha}--\eqref{eqnofvalpha} admits a unique solution $(v^{(\alpha)},z^{(\alpha)})$ with initial data holding
\begin{equation}
				\|y^0\|_{H^s}+\|y^1\|_{H^{s-1}}\leq\varepsilon.
			\end{equation}
 Moreover, for any $t\in [0,T]$ and $\alpha\geq 1$. The sequences $(v^{(\alpha)},z^{(\alpha)})$ and $(V^{(\alpha)},Z^{(\alpha)})$ satisfy
			\begin{equation}\label{induction}
				\begin{aligned}
					& \|V^{(\alpha)}\|_{H^{s-1}}^2+\|V^{(\alpha)}_t(t)\|_{H^{s-2}}^2\leq(1-\delta)^{2\alpha}{C_{V,s}}\varepsilon^2, \\
					& \|Z^{(\alpha)}(t)\|_{H^{s-1}}^2+\|Z^{(\alpha)}_t(t)\|_{H^{s-2}}^2\leq(1-\delta)^{2\alpha}C_{Z,s}\varepsilon^2,
				\end{aligned}
			\end{equation}
			and
			\begin{equation}\label{induction3}
				\begin{aligned}
					& \big\|v^{(\alpha)}(t)\big\|_{H^{s}}^2+\big\|v^{(\alpha)}_t(t)\big\|_{H^{s-1}}^2+\big\|v^{(\alpha)}_{tt}(t)\big\|_{H^{s-2}}^2\leq C_{v,s}\varepsilon^2, \\
					& \big\|z^{(\alpha)}(t)\big\|_{H^{s}}^2+\big\|z^{(\alpha)}_t(t)\big\|_{H^{s-1}}^2+\big\|v^{(\alpha)}_{tt}(t)\big\|_{H^{s-2}}^2 \leq C_{z,s}\varepsilon^2,
				\end{aligned}
			\end{equation}
			 where  $C_{V,s}, C_{Z,s}, C_{v,s}, C_{z,s} $ are positive constants independent of $\alpha,\varepsilon$.
		\end{lemma}

The demonstration of Lemma \ref{lem:VZ} is long, hence we postpone the proof in the next subsection. We give the proof of Proposition \ref{prop:vz} when assuming Lemma \ref{lem:VZ} holds.
		
		\begin{proof}[Proof of Proposition \ref{prop:vz} assuming Lemma \ref{lem:VZ} holds.]		
			
				By the definition of $V^{(\alpha)}$ given in \eqref{VZ}, equation \eqref{induction} entails that
				\begin{equation}\label{412}
					\left\|v^{(\alpha)} - v^{(\beta)}\right\|_{ C^{0}(0,T; \mathcal{H}^{s-1})}^2 \leq \sum_{i=\beta}^\alpha \left\|V^{(i)}\right\|_{C^{0}(0,T; \mathcal{H}^{s-1})}^2 \leq \frac{(1-\delta)^{2\beta}}{1 - (1-\delta)^2} (C_{V,s-1})^2 \varepsilon^2.
				\end{equation}
				Since $0 < \delta < 1$, inequality \eqref{412} and \eqref{induction3} with $k=s-1$ indicate that for each $t\in [0,T]$, the sequence $\{v^{(\alpha)}(t)\}_{\alpha=1}^\infty$ constitutes a Cauchy sequence in $ \mathcal{H}^{s-1}$. Thus, this together with $\{v_t^{(\alpha)}\}_{\alpha=1}^\infty\subset L^\infty(0,T; \mathcal{H}^{s-1})$ implies that there exists $y\in C^0(0,T; \mathcal{H}^{s-1})$ such that $\{v^{(\alpha)}\}_{\alpha=1}^\infty$ converges strongly to $y$ in $C^0(0,T; \mathcal{H}^{s-1})$.
				
				  By utilizing \eqref{induction} and \eqref{induction3}, we can also deduce that $\{v_t^{(\alpha)}\}_{\alpha=1}^\infty\subset L^\infty(0,T; \mathcal{H}^{s-1})$ and $\{v_{tt}^{(\alpha)}\}_{\alpha=1}^\infty\subset L^\infty(0,T; \mathcal{H}^{s-2})$ converge strongly to $\tilde{v}_1\in C(0,T; \mathcal{H}^{s-2})$ and $\tilde{v}_2\in C(0,T; \mathcal{H}^{s-3})$, respectively. Moreover, by \eqref{induction3}, we know $\{v_{tt}^{(\alpha)}\}_{\alpha=1}^\infty\subset L^\infty(0,T; \mathcal{H}^{s-2})$ and $\{v_{ttt}^{(\alpha)}\}_{\alpha=1}^\infty\subset L^\infty(0,T; \mathcal{H}^{s-3})$, so according to compactness argument, we have $\tilde{v}_1=y_t\in C^0(0,T; \mathcal{H}^{s-2})$ and $\tilde{v}_2=y_{tt}\in C^0(0,T; \mathcal{H}^{s-3})$.

				 Similarly, we can establish the existence of  $z\in C^{0}(0,T; \mathcal{H}^{s-1}), z_t\in  C^{0}(0,T; \mathcal{H}^{s-2})$.
				
                Noting that $s\geq \max\{n+2, 4\} \geq \frac{n}{2}+3$, hence by Morrey's embedding inequality, $\mathcal{H}^{s-2}\subset C^1(\Omega)$. This implies that for any $t\in [0,T]$,
                \begin{equation}
			\begin{cases}
				& a_{ij}^{(\alpha)}(t,x,v^{(\alpha)},v^{(\alpha)}_t,\nabla v^{(\alpha)})\rightarrow a_{ij}(t,x,y,y_t,\nabla y), \quad i,j=1,\cdots,n,\\
				& b^{(\alpha)}_i(t,x,v^{(\alpha)},v^{(\alpha)}_t,\nabla v^{(\alpha)})\rightarrow b_i(t,x,y,y_t,\nabla y), \quad i=0,\cdots,n, \\
				& \tilde{b}^{(\alpha)}(t,x,v^{(\alpha)},v^{(\alpha)}_t,\nabla v^{(\alpha)})\rightarrow \tilde{b}(t,x,y,y_ty,\nabla y).
			\end{cases}
		    \end{equation}
                in $C^1(\Omega)$, as $\alpha$ goes to $\infty$.

                By the way, we note that for both initial and terminal values satisfy
				 \begin{equation}
				 (v^{(\alpha)}(0),v_t^{(\alpha)}(0))\rightarrow (y(0),y_t(0)), ~ \text{in}\quad \mathcal{H}^{s-1}\times \mathcal{H}^{s-2},
				 \end{equation}
				 and
				 \begin{equation}
				 	(z^{(\alpha)}(T),z_t^{(\alpha)}(T))\rightarrow (z(T),z_t(T)), ~ \text{in} \quad \mathcal{H}^{s-1}\times \mathcal{H}^{s-2}.
				 \end{equation}

			     Given the initial condition of system \eqref{eqnofvalpha} and the terminal condition of system \eqref{eqnofzalpha}, letting $\alpha $ goes to $\infty$, we obtain
			     \begin{equation}
			     	\begin{split}
			     		(y(0),y_t(0))=&(\lim\limits_{\alpha\rightarrow \infty}v^{(\alpha)}(0), \lim\limits_{\alpha\rightarrow \infty}v^{(\alpha)}(0))=(y^0,y^1), \\
			     		 	(y(T),y_t(T))=&(\lim\limits_{\alpha\rightarrow \infty}v^{(\alpha)}(T), \lim\limits_{\alpha\rightarrow \infty}v^{(\alpha)}(T)), \\
			     		(\lim\limits_{\alpha\rightarrow \infty} z^{(\alpha)}(T),\lim\limits_{\alpha\rightarrow \infty} z_t^{(\alpha)}(T))=&	(\lim\limits_{\alpha\rightarrow \infty} z^{(\alpha-1)}(T),\lim\limits_{\alpha\rightarrow \infty} z_t^{(\alpha-1)}(T)) \\&+(\lim\limits_{\alpha\rightarrow \infty}v^{(\alpha-1)}(T), \lim\limits_{\alpha\rightarrow \infty}v^{(\alpha-1)}(T)).
			     	\end{split}
			     \end{equation}
				This immediately implies that in $\mathcal{H}^{s-1}\times \mathcal{H}^{s-2}$
				\begin{equation}
					(y(0),y_t(0))=(y^0,y^1),  (y(T),y_t(T))=(0,0).
				\end{equation}

				Next, from \eqref{induction3} and the compactness argument, we can deduce that there exists a subsequence of $\{v^{(\alpha)}\}_{\alpha=1}^\infty$ (denoted as $\{v^{(\alpha_1)}\}_{\alpha_1=1}^\infty$) and $\tilde{y}\in \cap_{p=0}^{2} W^{p,\infty}(0,T; \mathcal{H}^{s-p})$, such that
				\begin{equation}\label{411}
					(v^{(\alpha_1)}, v_t^{(\alpha_1)}) \stackrel{\ast}{\longrightarrow} (\tilde{y}, \tilde{y}_t), ~ \text{in } ~ L^\infty(0,T;\mathcal{H}^s) \times L^\infty(0,T;\mathcal{H}^{s-1}),
				\end{equation}
				as $\alpha_1$ goes to $\infty$.
				
				Since the limit is unique, we conclude that  $y=\tilde{y}\in \cap_{p=0}^{2} C^{p}(0,T; \mathcal{H}^{s-p})$.
							By analogous reasoning, we can establish that $\{z^{(\alpha)}\}$ converges strongly to $z\in\cap_{p=0}^{2} C^{p}(0,T; \mathcal{H}^{s-p})$.

				Finally, letting $u=-2z_t$,  these convergence results imply that $(y,u)$ is a solution of System \eqref{system:quasilinear} with initial data $(y^0, y^1)$
				and satisfying the terminal conditions
				 $(y(T), y_t(T)) = (0, 0)$. This completes the proof.
		\end{proof}

\subsection{Proof of Lemma \ref{lem:VZ}}

{\color{black}
        The proof of Lemma \ref{lem:VZ} consists of two points. The first one is to prove the well-posedness of the system \eqref{eqnofzalpha}, \eqref{eqnofvalpha}, \eqref{eqnofValpha}, \eqref{eqnofZalpha} for each $\alpha$. The second one is to show that the corresponding solutions satisfy the estimates \eqref{induction} and \eqref{induction3}.

		Before we state the well-posedness results for System \eqref{eqnofzalpha}--\eqref{eqnofvalpha}, it is imperative to establish a norm bound for composite functions. This estimation is essential for the subsequent analysis of the coefficients within our iterative scheme.

 We can have the following conclusion from the preceding discussion.
\begin{lem}\label{lem:FH} Let $s\geq n+2$ be an integer.
Assume that there exists a constant $\nu_1$ such that 	\begin{equation}
	\sum_{p=0}^{s}\|\partial_t^p z^{(\alpha-1)}\|_{  \mathcal{H}^{s-p}}+\sum_{p=0}^{s-1}\|\partial_t^p(z^{(\alpha-1)}-z^{(\alpha-2)})\|_{\mathcal{H}^{s-1-p}}\leq \nu_1,
\end{equation}
then we have for any $t\in [0,T]$,
\begin{equation}\label{coe:FH}
	\begin{split}
		\sum_{p=0}^{s-2}\|\partial_t^pF^{(\alpha)}\|_{\mathcal{H}^{s-2-p}}\leq C_F\big(\sum_{p=0}^{s-2}\|\partial_t^p(v^{(\alpha-1)}-v^{(\alpha-2)})\|_{\mathcal{H}^{s-1-p}}\big)\big(\sum_{p=0}^{s} \|\partial_t^pv^{(\alpha-1)}\|_{ \mathcal{H}^{s-p}}\big), \\
		\sum_{p=0}^{s-2}\|\partial_t^pH^{(\alpha)}\|_{\mathcal{H}^{s-2-p}}\leq C_H\big(\sum_{p=0}^{s-2}\|\partial_t^p(v^{(\alpha-1)}-v^{(\alpha-2)})\|_{ \mathcal{H}^{s-1-p}}\big)\big(\sum_{p=0}^{s} \|\partial_t^pz^{(\alpha-1)}\|_{  \mathcal{H}^{s-p}}\big),
	\end{split}				
\end{equation}
for some constants $C_F,C_H$ depending on $\nu_1, s, T, n$ and the bounds of $a_{ij,k}, i,j=1,\cdots,n,  k=0,1,\cdots,n+1$, $b_{i,k}, i,k=0, 1,\cdots,n+1$ and $\tilde{b}_{i}, i=0,1,\cdots, n+1$ in Remark \ref{rem:alphabound} independent of $\alpha$.
\end{lem}}

Here, we present a lemma on the norm estimate of a composite function in a bounded domain, which can be referred to as \cite[Lemma 2.1]{LiT2}. This lemma plays a crucial role in estimating the coefficients in the subsequent iterative estimates.

			\begin{lemma}\label{lem:GG}
			Let $l>n$ be an integer and $T>0$.
			Let $G(t,\gamma)=G(t,\gamma_1,\cdots,\gamma_{M})\in C^\infty([0,T]\times\R^M)$  be a bounded smooth function and satisfies
			\begin{equation}\label{GG}
				\|G\|_{C^l([0,T]\times \Omega)}\leq C_\Omega,
			\end{equation}
			for some constant $C_\Omega$ depends on $T$ and $\Omega$.
			If there exists a small positive constant 		     	
			$\nu_1$, such that
			\begin{equation}\label{gammal}
			\sum_{p=0}^l\big\|\partial_t^p\gamma_i\|_{  \mathcal{H}^{l-k}} \leq \nu_1, i=1,\cdots,M,
			\end{equation}	
			then  for any $t\in [0,T]$, for any $u,v\in \cap_{p=0}^lC([0,T]; \mathcal{H}^{l-p})$, we have
			\begin{equation}\label{GUV123}
				\begin{split}
					\sum_{p=0}^l\big\|\partial_t^pG(\gamma)\big\|_{ \mathcal{H}^{l-p}}&\leq C_1, \\
					\sum_{p=0}^l\big\|\partial_t^p(G(\gamma)u)\big\|_{  \mathcal{H}^{l-k}}&\leq C_2(\sum_{p=0}^l\big\|\partial_t^pu\|_{ \mathcal{H}^{l-p}}),\\
					\sum_{p=0}^l\big\|\partial_t^p(G(\gamma)uv)\big\|_{  \mathcal{H}^{l-k}}&\leq C_3(\sum_{p=0}^l\big\|\partial_t^pu\|_{  \mathcal{H}^{l-k}})(\sum_{p=0}^l\big\|\partial_t^pv\|_{ \mathcal{H}^{l-k}}),
				\end{split}
			\end{equation}
			where $C_i=C_i(n, C_\Omega, M, T, \nu_1, s)>0, i=1,2,3$ depend on $n, C_\Omega, M, T, \nu_1$ and $s$.
		\end{lemma}
		
	{\color{black}\begin{proof}[Proof of Lemma \ref{lem:GG}]
	Thanks to \eqref{GG} and \eqref{gammal}, the first inequality is straightforward. Moreover, since  $l-\lfloor\frac{n}{2}\rfloor>\frac{n}{2}$,  the Morrey inequality implies that for any $t\in [0,T]
	$ and $k\leq \lfloor\frac{n}{2}\rfloor$, \begin{equation}\label{GGG}
		\sum_{k=0}^{\lfloor\frac{n}{2}\rfloor}\big\| \partial_t^{k}G(\gamma)\big\|_{C(\Omega)} \leq C\sum_{k=0}^{\lfloor\frac{n}{2}\rfloor}\big\| \partial_t^{k}G(\gamma)\big\|_{H^{l-\lfloor\frac{n}{2}\rfloor}}\leq \tilde{C}_G,
	\end{equation}
	for some constant $\tilde{C}_G=\tilde{C}_G(n, M, C_\Omega, T, \nu_1, l)$.
	
	Next, we focus on verifying the second inequality in \eqref{GUV123}, as the remaining ones can be proven similarly. Let
	$\partial_x=\{\partial_{x_1},\cdots,\partial_{x_n}\}$.	Given two multi-index
	$L, K$ with $ |L|+|K|= l$,  following the method in \cite[Lemma 2.1]{LiT2}, we directly compute
	\begin{equation}\label{GGGGG}
		\partial_t^L\partial_x ^K\big[G(\gamma)u\big]=\sum\limits_{|K_1|+|K_2|=|K|} \big[\sum\limits_{|L_1|+|L_2|=|L|} C_{K_1,K_2,L_1,L_2}\partial_t^{L_1}\partial_x^{K_1} G(\gamma) \partial_t^{L_2}\partial_x^{K_2}u\big],
	\end{equation}
	where
	\begin{equation}\label{KGG}
		\partial_x^{K_1} G(\gamma)=\sum\limits_{\sum\limits_{j=1}^{M}l_j=l, 1\leq m\leq |K_1|} \frac{\partial^{m}_xG}{\partial_x^{l_{1}}\gamma_{1}\cdot\cdot\cdot\partial^{l_{M}}\gamma_{M}} (\partial_x \gamma)^{a_1}\cdot (\partial_x^{|K_1|} \gamma)^{a_{|K_1|}},
	\end{equation}
	with $\sum\limits_{j=1}^{|K_1|}|a_j|=l$ and $\sum\limits_{j=1}^{|K_1|}|j||a_j|=|K_1|$, and $C_{K_1,K_2,L_1,L_2}$ are constants independent of $G$ and $u$.
	
	Next, we divide the terms on the right-hand side of \eqref{GGGGG} into two parts based on $|L_1|+|L_2|\leq l$. The first part is when $|L_2|\geq \lfloor\frac{n}{2}\rfloor$ which  implies $|L_1|\leq \lfloor\frac{n}{2}\rfloor$. Combining this with \eqref{GGG} and taking the $L^2$ norm of  these terms, we obtain
	\begin{align*}
		&\Big\|\sum\limits_{|K_1|+|K_2|=|K|} \big[\sum\limits_{|L_1|+|L_2|=|L|, |L_1|\leq \frac{n}{2}} C_{K_1,K_2,L_1,L_2}\partial_t^{L_1}\partial_x^{K_1} G(\gamma) \partial_t^{L_2}\partial_x^{K_2}u\big]\Big\|_{L^2} \\
        \leq~& C \|u\|_{\cap_{k=0}^lC^k(0,T; \mathcal{H}^{l-k})}
	\end{align*}
	for some constant $C$.
	The second part is when $|L_2|\leq \lfloor\frac{n}{2}\rfloor$, which implies that $\partial_t^{|L_2|} u\in C(0,T;H^{l-|L_2|})$  and $l-|L_2|>\lfloor\frac{n}{2}\rfloor$.
	Then we observe that for any $t\in [0,T]$, the Sobolev space $H^{l-|L_2|}$ is a Banach algebra. Thus, we have
	\begin{equation}
		\|(\partial \gamma)^{a_1} \cdots (\partial^{|K_1|} \gamma)^{a_{|K_1|}} \partial^{K_2} u \|_{L^2} \leq C \|\gamma\|_{\cap_{p=0}^l C^p(0,T; \mathcal{H}^{l - p})}^{|K_1|} \|u\|_{\cap_{p=0}^l C^p(0,T; \mathcal{H}^{l - p})}.
	\end{equation}
Here we also use the fact that $\gamma, u\in \cap_{p=0}^l C^p(0,T; \mathcal{H}^{l - p})\subset \cap_{p=0}^l C^p(0,T; H^{l - p})$ and for any $t\in [0,T]$, $\|u\|_{\cap_{p=0}^l C^p(0,T; \mathcal{H}^{l - p})}\sim \|u\|_{\cap_{p=0}^l C^p(0,T; H^{l - p})}.$\footnote{Here we say that $\|u\|_{\cap_{p=0}^l C^p(0,T; \mathcal{H}^{l - p})}\sim \|u\|_{\cap_{p=0}^l C^p(0,T; H^{l - p})},$ for any $u\in \cap_{p=0}^l C^p(0,T; \mathcal{H}^{l - p})$, means that there exist two positive constants $C_1,C_2$, independent of $u$ such that $C_1 \|u\|_{\cap_{p=0}^l C^p(0,T; H^{l - p})}\leq \|u\|_{\cap_{p=0}^l C^p(0,T; \mathcal{H}^{l - p})} \leq C_2 \|u\|_{\cap_{p=0}^l C^p(0,T; H^{l - p})}$.}
	Combining this with \eqref{GG}, \eqref{gammal}, \eqref{GGGGG}, and \eqref{KGG}, we conclude that for each $k\leq l$,
	\begin{equation}
		\|\partial_t^k \partial^K \big[G(\gamma)u\big]\|_{C(0,T; L^2)} \leq C(n, M, C_\Omega, T, \nu_1, l) \|u\|_{\cap_{p=1}^l C^p(0,T; \mathcal{H}^{l - p})}.
	\end{equation}
	Thus, the proof is complete.
\end{proof}
}

\begin{proof}[Proof of Lemma \ref{lem:FH}.]
Combining the fact that the coefficients have the expanding form seen in Remark \ref{rem:alphabound}, by Lemma \ref{lem:GG}, we first obtain that
\begin{equation}\label{estimate:alphaab}
\begin{split}
	\sum_{p=0}^{s-1}\|\partial_t^p(a_{ij}^{(\alpha)}-\delta_{ij})\|_{\mathcal{H}^{s-1-p}}&\leq C_{ab}\sum_{p=0}^{s}\|v^{(\alpha-1)}\|_{ \mathcal{H}^{s-p}}, ~ i,j=1,\cdots,n, \\
	\sum_{p=0}^{s-1}\|\partial_t^p	b_k^{(\alpha)}\|_{ \mathcal{H}^{s-1-p}}&\leq C_{ab}\sum_{p=0}^{s}\|v^{(\alpha-1)}\|_{ \mathcal{H}^{s-p}}, ~ k=1,\cdots, n, \\
	\sum_{p=0}^{s-1}\|\partial_t^p	(b_0^{(\alpha)}-1)\|_{ \mathcal{H}^{s-1-p}}&\leq C_{ab}\sum_{p=0}^{s}\|v^{(\alpha-1)}\|_{\mathcal{H}^{s-p}}, \\
	\sum_{p=0}^{s-1}\|\partial_t^p	(	\tilde{b}^{(\alpha)})\|_{ \mathcal{H}^{s-1-p}}&\leq C_{ab}\sum_{p=0}^{s}\|v^{(\alpha-1)}\|_{\mathcal{H}^{s-p}},
\end{split}
\end{equation}
where $C_{ab}$ is a constant depending on $\nu_1, s,T, \Omega$ and the bounds of $a_{ij,k}, i,j=1,\cdots,n,  k=0,1,\cdots,n+1$, $b_{i,k}, i,k=0, 1,\cdots,n+1$ and $\tilde{b}_{i}, i=0,1,\cdots, n+1$,  independent of $\alpha$ and $v^{(\alpha-1)}$.

We can then consider the estimations on $F^{(\beta+1)}$ and $H^{(\beta+1)}$. Based on the expressions \eqref{Gdefn} and \eqref{Hdefn}, we need to estimate the following terms:
\begin{equation}\label{coe:bb1}
\begin{aligned}
	\big\|\partial_t^p (b_0^{(\alpha)}-b_0^{(\alpha-1)})\big\|_{\mathcal{H}^{s-1-p}}
	 ,\qquad \big\|\partial_t^p(\tilde{b}^{(\alpha)}-\tilde{b}^{(\alpha-1)})\big\|_{ \mathcal{H}^{s-1-p}},
\end{aligned}
\end{equation}
and
$$\big\|\partial_t^p(b_i^{(\alpha)}-b_i^{(\alpha-1)})\big\|_{ \mathcal{H}^{s-1-p}},\qquad \big\|\partial_t^p(a_{ij}^{(\alpha)}-a_{ij}^{(\alpha-1)})\big\|_{ \mathcal{H}^{s-1-p}},$$
for any non-negative integer $p\leq s-2$.

We only deal with $\big\|\partial_t^p (b_0^{(\alpha)}-b_0^{(\alpha-1)})\big\|_{\mathcal{H}^{s-1-p}} $, other terms are the same.
Note the fact that
\begin{align*}
b_0^{(\alpha)}-b_0^{(\alpha-1)}&=b_0(t,x,v^{(\alpha-1)},v_t^{(\alpha-1)}, \nabla v^{(\alpha-1)})-b_0(t,x,v^{(\alpha-2)},v_t^{(\alpha-2)}, \nabla v^{(\alpha-2)}) \\
                                        &=b_{0,v}(v^{(\alpha-1)}-v^{(\alpha-2)})+b_{0,v_t}(v_t^{(\alpha-1)}-v_t^{(\alpha-2)})+\sum_{i=1}^n b_{0,vi}(v_{x_i}^{(\alpha-1)}-v_{x_i}^{(\alpha-2)}) \\
                                        &=b_{0,v}V^{(\alpha-1)}+b_{0,v_t}V_t^{(\alpha-1)}+\sum_{i=1}^n b_{0,vi}V_{x_i}^{(\alpha-1)}
\end{align*}
where
$$
\begin{cases}
b_{0,v}=\int_0^1\frac{\partial}{\partial v} b_0(t,x, \theta v^{(\alpha-1)},v_t^{(\alpha-1)}, \nabla v^{(\alpha-1)}) \dd\theta,\\ b_{0,v_t}=\int_0^1\frac{\partial}{\partial v_t} b_0(t,x, v^{(\alpha-1)},\theta v_t^{(\alpha-1)}, \nabla v^{(\alpha-1)}) \dd\theta,\\ b_{0,vi}=\int_0^1\frac{\partial}{\partial v_{x_i}} b_0(t,x,  v^{(\alpha-1)},v_t^{(\alpha-1)}, \theta  v_{x_i}^{(\alpha-1)}) \dd\theta.
\end{cases}
$$
Since we have assume that
\begin{equation}
\sum_{p=0}^{s-1}\|\partial_t^pV^{(\alpha-1)}\|_{ \mathcal{H}^{s-1-p}}\leq \nu_1, \quad \sum_{i=1,2}\sum_{p=0}^{s}\|\partial_t^pv^{(\alpha-i)}\|_{\mathcal{H}^{s-p}}\leq \nu_1,
\end{equation}
for some small $\nu_1$.
Then, noting that $s-1\geq n+1$, we can apply Lemma \ref{lem:GG} with $G=b_{0,v}$ (resp. $b_{0,v_t}, b_{0,vi}$) and $l=s-1$, we can obtain
\begin{equation}
\begin{aligned}
	\sum_{p=0}^{s-2}\big\|\partial_t^p(b_0^{(\alpha)}-b_0^{(\alpha-1)})\big\|_{\mathcal{H}^{s-2-p}}& \leq C'_{ab}\sum_{p=0}^{s-1}\|\partial_t^p(v^{(\alpha-1)}-v^{(\alpha-2)})\|_{\mathcal{H}^{s-1-p}} \\
&\leq C'_{ab}\sum_{p=0}^{s-1}\|\partial_t^pV^{(\alpha-1)}\|_{\mathcal{H}^{s-1-p}},
\end{aligned}
\end{equation}
where $C'_{ab}$ is a constant independent with $\alpha$ and  $V^{(\alpha-1)}$.

Thus applying Lemma \ref{lem:GG} again with $G=a_{ij}^{(\alpha)}-a_{ij}^{(\alpha-1)}$ and $l=s-1$, we have
\begin{equation}\label{coe:bb2}
\begin{aligned}
	\sum_{p=0}^{s-2}\big\|\partial_t^p((b_0^{(\alpha)}-b_0^{(\alpha-1)})v_t^{(\alpha-1)})\big\|_{\mathcal{H}^{s-2-p}}& \leq C''_{ab}(\sum_{p=0}^{s-1}\|\partial_t^pV^{(\alpha-1)}\|_{\mathcal{H}^{s-1-p}})(\sum_{p=0}^{s-1}\|\partial_t^pv_t^{(\alpha-1)}\|_{\mathcal{H}^{s-1-p}}) \\
&\leq C''_{ab}(\sum_{p=0}^{s-1}\|\partial_t^pV^{(\alpha-1)}\|_{\mathcal{H}^{s-1-p}})(\sum_{p=0}^{s}\|\partial_t^pv^{(\alpha-1)}\|_{\mathcal{H}^{s-p}}),
\end{aligned}
\end{equation}
where $C''_{ab}$ is a positive constant independent of $\alpha$.
Similarly, we can obtain the estimations of other terms in $F^{(\beta+1)}$ and $H^{(\beta+1)}$. Thus we complete the proof.
\end{proof}

We are now in a position to commence the proof of Lemma \ref{lem:VZ}.
		
		\begin{proof}[Proof of Lemma \ref{lem:VZ}.]
%
           To obtain the spatial norm estimates of the system solutions \eqref{induction} and \eqref{induction3}, we need to establish the following norm estimates:
           There exist constants $\varepsilon_{lem}, M,\delta$, ${C}_{V,mid,l},{C}_{Z,mid,l}, {C}_{v,mid,l},{C}_{z,mid,l}$ for any $l=0,\cdots, s-2$, independent of $\varepsilon$ and $\alpha$, such that for any $\varepsilon<\varepsilon_{lem}$ and any time $t\in [0,T]$, we have

 \begin{equation}\label{induction4}
				\begin{aligned}
					& \|\partial_t^{m+1} V^{(\alpha)}\|_{\mathcal{H}^{l-m}}^2+\|\partial_t^{m}V^{(\alpha)}\|_{\mathcal{H}^{l+1-m}}^2\leq(1-\delta)^{2\alpha-2}{C}_{V,mid,l}^{2}M^{2l}\varepsilon^2,\\
					& \|\partial_t^{m+1} Z^{(\alpha)}\|_{\mathcal{H}^{l-m}}^2+\|\partial_t^{m} Z^{(\alpha)}\|_{\mathcal{H}^{l+l-m}}^2\leq (1-\delta)^{2\alpha-2}{C}_{Z,mid,l}^{2}M^{2l}\varepsilon^2,
				\end{aligned}
			\end{equation}
			and
			\begin{equation}\label{induction5}
				\begin{aligned}
					& \big\|\partial_t^{m+1} v^{(\alpha)}\big\|_{\mathcal{H}^{l+1-m}}^2+\big\|\partial_t^{m}v^{(\alpha)}\big\|_{\mathcal{H}^{l+2-m}}^2 \leq {C}_{v,mid,l}^2M^{2l+2}\varepsilon^2, \\
					& \big\|\partial_t^{m+1} z^{(\alpha)}\big\|_{\mathcal{H}^{l+1-m}}^2+\big\|\partial_t^{m} z^{(\alpha)}\big\|_{\mathcal{H}^{l+2-m}}^2 \leq {C}_{z,mid,l}^{2}M^{2l+2}\varepsilon^2,
				\end{aligned}
			\end{equation}
		for any $m=0,\cdots,l$.

Taking $l=s-2$ and $m=0,1$ directly, we can immediately derive \eqref{induction} and \eqref{induction3} from \eqref{induction4} and \eqref{induction5}. Therefore, we will focus on proving \eqref{induction4} and \eqref{induction5} in the following.

We prove \eqref{induction4} and \eqref{induction5} by deduction.
			{\color{black}The proof of the assertions \eqref{induction4} and \eqref{induction5} will be systematically approached through a sequence of methodical steps, delineated as follows:			
				\begin{enumerate}
					\item Establish the base case by demonstrating that \eqref{induction4}, and \eqref{induction5} hold true when $\alpha=1$;
                    \item Establish the base case by demonstrating that \eqref{induction4} holds true when $l=0$;

					\item  Proceed to the inductive step for \eqref{induction4} and \eqref{induction5}, where it is to be shown that for $\alpha=\beta+1\geq 2, s-2\geq l\geq k+1\geq 2$, the assertion \eqref{induction4}  is valid. Similarly,  for $\alpha=\beta+1\geq 2, s-3\geq l\geq k+1\geq 1$, the assertion \eqref{induction5}  is valid, given that the conditions \eqref{induction4} and \eqref{induction5} are presupposed to be valid for $\alpha\leq \beta, l\leq k\geq 1$;
                    \item Proceed to the inductive step for \eqref{induction5}, where it is to be shown that for $\alpha=\beta+1\geq 2$, the assertion \eqref{induction5} is valid, given that the conditions \eqref{induction4} and \eqref{induction5} are presupposed to be valid for $\alpha=\beta$.
				\end{enumerate}
			}

			\subsubsection{Basic step 1: The case of $ \alpha=1$. }
	
We first note that $V^{(1)}=v^{(1)},Z^{(1)}=z^{(1)}$, which satisfy the following equations:
			\begin{equation}\label{VZ1}
				(\partial_t^2-\Delta+\partial_t) V^{(1)}=-2\chi Z_t^{(1)}, (\partial_t^2-\Delta-\partial_t) Z^{(1)}=0.
			\end{equation}
			Additionally, we have the initial and terminal conditions: \begin{equation}\label{casel00}(V^{(1)}(0),V_t^{(1)}(0))=(v^{(1)}(0),v_t^{(1)}(0)),
			(Z^{(1)}(T),Z_t^{(1)}(T))=(0,0).\end{equation}
			By the well-posedness theory of linear wave equations, there exists a constant $c_0>0$, such that for any $ t\in [0,T]$, for any $k=0,\cdots,s-1,$
			\begin{equation*}
\begin{split}
\|\partial_t^kZ^{(1)}\|_{\mathcal{H}^{s-k}}^2+\|\partial_t^{k+1}Z^{(1)}\|_{\mathcal{H}^{s-k-1}}^2\leq e^{c_0(t-T)}\big(\|\partial_t^kZ^{(1)}(T)\|_{\mathcal{H}^{s-k}}^2+\|\partial_t^{k+1}Z^{(1)}(T)\|_{\mathcal{H}^{s-k-1}}^2\big)=0.
\end{split}			
\end{equation*}
This implies that $-2\chi Z_t^{(1)}\equiv0$. Consequently, using the equation for $V$ in \eqref{VZ1}, we obtain that for any $l\leq s$ and $k=0,\cdots,l$,
		\begin{equation*}
\begin{split}
				\|\partial_t^kV^{(1)}\|_{\mathcal{H}^{l-k}}^2+\|\partial_t^{k+1}V^{(1)}\|_{\mathcal{H}^{l-k-1}}^2&\leq e^{-c_0t}\big(\|\partial_t^kV^{(1)}(0)\|_{\mathcal{H}^{l-k}}^2+\|\partial_t^{k+1}V^{(1)}(0)\|_{\mathcal{H}^{l-k-1}}^2\big).
\end{split}			
\end{equation*}
            Using the relation \eqref{casel00}, we have  
            $$\Delta^m V^{(1)}(0)=\Delta^mv^{(1)}(0), \quad \Delta^m\partial_tV^{(1)}(0)=\Delta^m\partial_tv^{(1)}(0),$$
            for any integer $m\geq 0$. Applying the operator $\partial_t^{k-2}$ the equation for $V^{(1)}$ in \eqref{VZ1} with $Z\equiv0$, we obtain that for any $k\geq 2$,
				\begin{align}\label{V:k:odd:l:0}
					\partial_t^kV^{(1)}+\partial_t^{k-1}V^{(1)}=\Delta(\partial_t^{k-2}V^{(1)}).
			\end{align}
            Using this recursive relation, we can express
            $$\partial^k_tV^{(1)}(0)=\partial_t^{m} \Delta^{\frac{k-m}{2}}V^{(1)}(0)+\sum_{p=0}^{\frac{k-m}{2}} \Delta^{p}\big(C_{k,p,1}V^{(1)}(0)+ C_{k,p,2}\partial_tV^{(1)}(0)\big).$$
            where $m=\frac{1-(-1)^k}{2}$, $C_{k,p,1}$ and $C_{k,p,2}$ are constants depending only on $k,p$.
            By elliptic regularity theory, for any $u\in \mathcal{H}^s$ and $s_1\leq s_2\leq s$, there exists a constant $C_{s_1,s_2}$ depending only on $s_1,s_2$ and $\Omega$, such that  $\|u\|_{\mathcal{H}^{s_1}}\leq C_{s_1,s_2}\|u\|_{\mathcal{H}^{s_2}}$.

            Thus, for any $l=1,\cdots,s-1$ and $k=0,\cdots,l$, we have
            \begin{equation*}
            \begin{split}
            \|\partial_t^{k}V^{(1)}(0)\|_{\mathcal{H}^{l-k}}^2+\|\partial_t^{k+1}V^{(1)}(0)\|_{\mathcal{H}^{l-k-1}}^2& \leq C_l (\|v^{(1)}(0)\|_{\mathcal{H}^{l}}^2+\|v^{(1)}_t(0)\|_{\mathcal{H}^{l-1}}^2)\\ &\leq C_{Vini,l}( \|y^0\|_{\mathcal{H}^s}^2+\|y^1\|_{\mathcal{H}^{s-1}}^2),
            \end{split}
            \end{equation*}
            for some constant $C_l, C_{Vini,l}>0$ depending only on $k$, $s$ and $\Omega$.
              Together with the smallness assumption on the initial data $(y^0,y^1)$, we conclude that if the constants $C_{V,mid,l}, C_{v,mid,l}$ are setting by
              \begin{equation}\label{cond:Vconstant}
              C_{V,mid,l}= 2C_{Vini,l}, \quad C_{v,mid,l}=\frac{2(1-\delta)^2}{1-(1-\delta)^2}C_{Vini,l},
              \end{equation}
              and $C_{Z,mid,l}, C_{z,mid,l}$ are setting by
                \begin{equation}\label{cond:Zconstant}
              C_{Z,mid,l}= \frac{1}{4T}C_{V,mid,l}, \quad C_{z,mid,l}=\frac{1}{4T}C_{v,mid,l},
              \end{equation}
			  then the estimates \eqref{induction4} and \eqref{induction5} hold for the case $\alpha=1$.

            \subsubsection{Basic step 2: To prove \eqref{induction4} and \eqref{induction5} when $l=0$ for any $\alpha\geq 2$}

            When $l=0$, we need to establish estimates for the coefficients. Since \eqref{induction4} and \eqref{induction5} are assumed to hold for $\alpha=\beta\geq 1$,  we can choose
			$\varepsilon_{lem}\leq \varepsilon_{nu}$		
			sufficiently small such that
			\begin{equation}\label{epsilon01}
				\max_{0\leq k\leq s}C_{v,mid,k}M^{k+1} \varepsilon_{nu} \leq 1=:\nu_1.
			\end{equation}
           Here $C_{v,mid,k}, k=0,1,\cdots,s$ are defined by \eqref{cond:Vconstant} and $M$ would be choosing later.

			 By Lemma \ref{lem:GG} and \eqref{estimate:alphaab} in Remark \ref{rem:alphabound} with $\nu_1=1$, for any $t\in [0,T]$,  we obtain for each $\alpha\leq \beta$,
			\begin{align}\label{vertify:1}\begin{split}
				&\big\|b^{(\alpha)}_0-1\big\|_{\cap_{k=0}^{s-1}C^k(0,T;\mathcal{H}^{s-1-k})}+\big\|a^{(\alpha)}_{ij}-\delta_{ij}\big\|_{\cap_{k=0}^{s-1}C^k(0,T;\mathcal{H}^{s-1-k})}\\
              &+\big\|\tilde{b}^{(\alpha)}\big\|_{\cap_{k=0}^{s-1}C^k(0,T;\mathcal{H}^{s-1-k})}+\big\|b^{(\alpha)}_i\big\|_{\cap_{k=0}^{s-1}C^k(0,T;\mathcal{H}^{s-1-k})}\leq 4C_\Omega \max_{k}\{C_{v,mid,k}M^{k+1}\}\varepsilon,
			\end{split}
            \end{align}
			where $C_\Omega$ is a constant depending on the expression of coefficients in Remark \ref{rem:alphabound}, independent of $\alpha$ and $\varepsilon$.

			Next, we choose $\varepsilon_{lem}\leq \min\{\varepsilon_{Hs}, \varepsilon_{nu}\}$ with
			\begin{equation}\label{epsilon02}
				4C_\Omega \max_{k}\{C_{v,mid,k}M^{k+1}\}\varepsilon_{Hs}= \varepsilon_{\mathcal{H}^s},
			\end{equation}
			where $\varepsilon_{\mathcal{H}^s}$ is defined by \eqref{coecond:norm1} with $l=s\geq \lfloor\frac{n}{2}\rfloor+2$.
Thus, the coefficients meet the conditions \eqref{coecond:norm1} with $l=s\geq \lfloor\frac{n}{2}\rfloor+2$. Applying Corollary \eqref{cor:well-reg}, we conclude that System \eqref{eqnofvalpha}-\eqref{eqnofzalpha} admits a unique solution $(v^{(\alpha)},z^{(\alpha)})\in \cap_{k=0}^1C^{k}(0,T;\mathcal{H}^{s-k})\times \cap_{k=0}^1C^{k}(0,T;\mathcal{H}^{s-k})$. Consequently, we have $V^{(\alpha)},Z^{(\alpha)}\in \cap_{k=0}^1C^{k}(0,T;\mathcal{H}^{s-k})$.

                We can use well-posedness of the system of $Z^{(\alpha)}$ to transform the estimate \eqref{induction4}, which holds for any time $t\in [0,T]$, into the following estimates at the terminal time $T$.

             \begin{claim}\label{claim:l0} For any $\alpha\geq 1$,  $Z^{(\alpha)}$ satisfies \begin{equation}\label{induction-Z}
				\begin{aligned}
					 \|\partial_t Z^{(\alpha)}(T)\|_{L^{2}}^2+\| Z^{(\alpha)}(T)\|_{\mathcal{H}^{1}}^2\leq C_{Z_T,mid,0}(1-\delta)^{2\alpha-2}\varepsilon^2,
				\end{aligned}
			\end{equation}
for some constant $C_{Z_T,mid,0}$ independent of $\alpha$ and $\varepsilon$.
             \end{claim}
             \begin{proof}
             Since $Z^{(1)}\equiv 0$ for any $t\in [0,T]$, the estimate \eqref{induction-Z} holds trivially for $\alpha=1$. We now proceed by induction. Assume that \eqref{induction-Z} holds for all $\alpha\leq \beta$.

             Referring to the proof of the linear system, for any $\beta$, we define
		    $$w^{(\beta)}=v^{(\beta)}+z^{(\beta)},\qquad W^{(\beta)}=w^{(\beta)}-w^{(\beta-1)}.$$
		  From the initial data of the $z$-system \eqref{eqnofzalpha}, we derive
		    \begin{equation}
		    	W^{(\beta)}_t(T)=Z^{(\beta+1)}(T), \quad W^{(\beta)}_t(T)=Z^{(\beta+1)}_t(T).
		    \end{equation}
		    For convenience, we define the energy functional: for any  $U\in C(0,T;\mathcal{H}^1)\cap C^1(0,T;L^2)$,
            \begin{equation}
                E_\alpha(U)(t)=\frac{1}{2}\bigg(\int_\Omega |U_t(t)|^2 \dd x+ \sum\limits_{i,j=1}^n\int_\Omega a_{ij}^{(\alpha)}U_{x_i}(t)U_{x_j}(t)\dd x\bigg).
            \end{equation}
           Based on the coefficient estimates \eqref{vertify:1}, and using embedding theory, we derive the $C^1$ estimates for the coefficients:
		    \begin{equation}\label{coe:C1}
		    	\begin{cases}
		    		\|a_{ij}^{(\alpha)}-\delta_{ij}\|_{ C^1((0,T)\times \Omega)}\leq C_{coe,C^1}C_\Omega \max_{k}\{C_{v,mid,k}M^{k+1}\}\varepsilon, \quad i,j=1,\cdots, n, \\
		    		\|b_0^{(\alpha)}-1\|_{C^1((0,T)\times \Omega)}\leq C_{coe,C^1}C_\Omega \max_{k}\{C_{v,mid,k}M^{k+1}\}\varepsilon,\\ \|\tilde{b}^{(\alpha)}\|_{C^1((0,T)\times \Omega)}\leq C_{coe,C^1}C_\Omega \max_{k}\{C_{v,mid,k}M^{k+1}\}\varepsilon,\\	
           \|b_k^{(\alpha)}\|_{C^1((0,T)\times \Omega)}\leq C_{coe,C^1}C_\Omega \max_{k}\{C_{v,mid,k}M^{k+1}\}\varepsilon, \quad k=1,\cdots, n,
		    	\end{cases}
		    \end{equation}
		where 	$C_{coe,C^1}$  depends only on $\Omega$ and $n$.
           Then we can show that for any $\alpha\leq \beta+1$, for any  $U\in C(0,T;\mathcal{H}^1)\cap C^1(0,T;L^2)$,
           \begin{equation}\label{est:EnergyET}
           (2-C_{coe,1}\varepsilon)E_\alpha(U) \leq  \|\partial_t U\|_{L^{2}}^2+\| U\|_{\mathcal{H}^{1}}^2\leq (2+C_{coe,1}\varepsilon)E_\alpha(U),
           \end{equation}
           where $C_{coe,1}=nC_{coe,C^1}C_\Omega \max_{k}\{C_{v,mid,k}M^{k+1}\}$ is independent of $\beta$ and $\varepsilon$.

           We now consider the following expression:
           \begin{equation}\label{EW1}
           \begin{split}
           E_\beta(Z^{(\beta+1)})-E_\beta(Z^{(\beta)})&=
		    E_\beta(W^{(\beta)})-E_\beta(Z^{(\beta)}) \\&=
E_\beta(V^{(\beta)})+\int_\Omega\big(V_t^{(\beta)}Z_t^{(\beta)} +\sum\limits_{i,j=1}^na_{ij}^{(\beta)}V^{(\beta)}_{x_i}\cdot Z^{(\beta)}_{x_j}\big)\dd x.\end{split}\end{equation}

           We first estimate $E_\beta(V^{(\beta)})$. Multiplying the equation for $ V^{(\beta)}$ by $V_t^{(\beta)}$,
we derive the following inequality:
\begin{equation}\label{eqn of partial t k V alpha}
				\begin{aligned}
					 V^{(\beta)}_{tt}V_t^{(\beta)}+b_0^{(\beta)}V^{(\beta)}_tV_t^{(\beta)} &
-\sum\limits_{i,j=1}^n\big(a_{ij}^{(\beta)}V^{(\beta)}_{x_i}\big)_{x_j}V_t^{(\beta)}+ \tilde{b}^{(\beta)}V^{(\beta)}V_t^{(\beta)}
					+  \sum\limits_{i=1}^nb_i^{(\beta)}V^{(\beta)}_{x_i}V_t^{(\beta)} \\
=
&F^{(\beta)}V_t^{(\beta)}-2(\chi Z_t^{(\beta)})V_t^{(\beta)}.
				\end{aligned}
			\end{equation}
            
          By Stokes' formula, we have
          \begin{equation}\label{uni:1}
                \int_\Omega V^{(\beta)}_{tt}V_t^{(\beta)}\dd x=\frac{1}{2}\frac{\dd}{\dd t}\int_\Omega\big| V_t^{(\beta)}\big|^2\dd x.
          \end{equation}

Utilizing the symmetry property  $a^{(\beta+1)}_{ij}=a_{ji}^{(\beta+1)}$, we obtain
    \begin{align}\label{uni:3}
    &-\sum\limits_{i,j=1}^n\int_\Omega\big(a_{ij}^{(\beta+1)}V^{(\beta)}_{x_i}\big)_{x_j}V^{(\beta+1)}_t\dd x \nonumber\\
    =~&\frac{1}{2}\frac{\dd}{\dd t}\int_\Omega \sum\limits_{i,j=1}^na_{ij}^{(\beta+1)}V^{(\beta)}_{x_i} V^{(\beta)}_{x_j}\dd x
            -\frac{1}{2}\sum\limits_{i,j=1}^n \int_\Omega a^{(\beta)}_{ijt}V^{(\beta)}_{x_i}\cdot V^{(\beta)}_{x_j}\dd x.
    \end{align}
      In view of \eqref{coe:C1}, by the inequality $ab\leq \frac{1}{2}(a^2+b^2)$ and Poincar\'{e}'s inequality, we obtain that there exists a constant $C_{coe}$ independent of $\varepsilon$ and $\beta$ such that
      \begin{align}\label{uni:4}
      &\left|\int_\Omega(b^{(\beta)}_0-1)|V_t^{(\beta)}|^2\dd x\right|+\left|\int_\Omega \tilde{b}^{(\beta)}V^{(\beta)}V_t^{(\beta)}\dd x\right| \nonumber \\
      &+  \left|\int_\Omega \sum\limits_{i=1}^nb_i^{(\beta)}V^{(\beta)}_{x_i}V_t^{(\beta)}\dd x\right|  + \left|\frac{1}{2}\sum\limits_{i,j=1}^n \int_\Omega a^{(\beta)}_{ijt}V^{(\beta)}_{x_i}\cdot V^{(\beta)}_{x_j}\dd x\right| \nonumber \\
      \leq~& C_{coe}\varepsilon E_\beta(V^{(\beta)}).
      \end{align}

            To estimate $F^{(\beta)}V^{(\beta)}_t$, we recall
             \eqref{coe:FH} in Lemma \ref{lem:FH}. Combining this with the induction hypothesis that \eqref{induction3} holds for $\alpha=\beta\geq 2$, we obtain that for any $t\in [0,T]$,
		    \begin{equation*}
		    	\|F^{(\beta)}\|_{L^2}\leq \tilde{C}_F (1-\delta)^{\beta-2} \varepsilon^2
		    \end{equation*}
           where  $\tilde{C}_F$ is a constant independent with $\beta$ and $\varepsilon$.
            Applying H\"{o}lder's inequality, we then have for any $t\in [0,T]$,
            \begin{equation}\label{uni:6}
                \int_\Omega|F^{(\beta)}V^{(\beta)}_t|\dd x\leq \tilde{\tilde{C}}_F (1-\delta)^{2\beta-2} \varepsilon^3,
            \end{equation}
            where  $\tilde{\tilde{C}}_F$ is a constant independent with $\beta$ and $\varepsilon$.

             Next, we observe that
\begin{equation}\label{uni:7}
\begin{split}		
\int_\Omega\big|V_t^{(\beta)}\big|^2\dd x
                    +2\int_\Omega V_t^{(\beta)}(\chi Z_t^{(\beta)})\dd x
                     \geq-\int_\Omega\big|\chi\cdot Z_t^{(\beta)}\big|^2\dd x,
                    \end{split}
                    \end{equation}

            Combining \eqref{uni:1}--\eqref{uni:6} and \eqref{uni:7}, we arrive at
            \begin{equation*}
            \frac{\dd}{\dd t}E_\beta(V^{(\beta)}) \leq   \big\|\chi\cdot Z_t^{(\beta)}\big\|_{L^2}^2 + C_{coe} \varepsilon E_\beta(V^{(\beta)})+ \tilde{\tilde{C}}_F (1-\delta)^{2\beta-2} \varepsilon^3.
            \end{equation*}
           Applying Gronwall's inequality, we obtain that  for all
$t\in [0,T]$,
		    \begin{equation}\label{VtoZ}
		    	E_\beta(V^{(\beta)})
		    	\leq C_{VtoZ}(\varepsilon) \left( \int_0^T\int_\Omega\big|\chi\cdot Z_t^{(\beta)}\big|^2\dd x \dd t + T\tilde{\tilde{C}}_F(1-\delta)^{2\beta-2} \varepsilon^3 \right),
		    \end{equation}
		    where $C_{VtoZ}(\varepsilon)=e^{TC_{coe} \varepsilon}$  is a constant that is independent of  $\delta$ and $\beta$.

			We next estimate $$\int_\Omega V_t^{(\beta)}Z_t^{(\beta)} \dd x+\sum\limits_{i,j=1}^n\int_\Omega a_{ij}^{(\beta)}V^{(\beta)}_{x_i}\cdot Z^{(\beta)}_{x_j}\dd x.$$
            Multiplying \eqref{eqnofValpha} by $Z_t^{(\beta)}$, \eqref{eqnofZalpha} by $V_t^{(\beta)}$, and integrating over $\Omega$, we obtain
			\begin{equation}\label{est:E1ZT-1}
			\begin{aligned}
				& \int_\Omega\Big(Z_t^{(\beta)}V^{(\beta)}_{tt}-\sum\limits_{i,j=1}^n Z_t^{(\beta)}\big(a_{ij}^{(\beta)}V^{(\beta)}_{x_i}\big)_{x_j}+\tilde{b}^{(\beta)}Z_t^{(\beta)}V^{(\beta)}+ \sum\limits_{i=1}^nb_i^{(\beta)}Z_t^{(\beta)}V^{(\beta)}_{x_i} \\
				& +V_t^{(\beta)}Z^{(\beta)}_{tt}-\sum\limits_{i,j=1}^nV_t^{(\beta)}\big(a_{ij}^{(\beta)}Z^{(\beta)}_{x_i}\big)_{x_j}+\tilde{b}^{(\beta)}V_t^{(\beta)}Z^{(\beta)}+\sum\limits_{i=1}^n b_i^{(\beta)}V_t^{(\beta)}Z^{(\beta)}_{x_i}\Big)\dd x \\
				& +2\int_\Omega\chi\big|Z_t^{(\beta)}\big|^2\dd x
				=\int_\Omega H^{(\beta)}Z_t^{(\beta)} \dd x.
			\end{aligned}
			\end{equation}
			
			By integration by parts, we derive
			\begin{equation}\label{est:E1ZT-2}
			\begin{aligned}
				& -\int_\Omega\sum\limits_{i,j=1}^n\Big(Z_t^{(\beta)}\big(a_{ij}^{(\beta)}V^{(\beta)}_{x_i}\big)_{x_j}+V_t^{(\beta)}\big(a_{ij}^{(\beta)}Z^{(\beta)}_{x_i}\big)_{x_j}\Big)\dd x \\
				= & ~ \frac{\dd}{\dd t}\int_\Omega\Big(\sum\limits_{i,j=1}^na_{ij}^{(\beta)}Z_{x_i}^{(\beta)}V^{(\beta)}_{x_j}\Big)\dd x-\int_\Omega\sum\limits_{i,j=1}^n\big(\partial_ta_{ij}^{(\beta)} \big)Z_{x_i}^{(\beta)}V^{(\beta)}_{x_j}\dd x,
			\end{aligned}
			\end{equation}
            and
            \begin{equation}\label{est:E1ZT-3}
            \int_\Omega\Big(Z_t^{(\beta)}V^{(\beta)}_{tt}+V_t^{(\beta)}Z^{(\beta)}_{tt}\big)\dd x=\frac{\dd}{\dd t}\int_\Omega V_t^{(\beta)}Z_t^{(\beta)}\dd x.
            \end{equation}

			Using the coefficient estimates \eqref{coe:C1}, the inequality $ab\leq \frac{1}{2}(a^2+b^2)$ and Poincar\'{e}'s inequality, we obtain:
            \begin{equation}\label{est:E1ZT-4}
           \begin{split}
           \left|\int_\Omega \tilde{b}^{(\beta)}Z_t^{(\beta)}V^{(\beta)}\dd x\right|\leq C_1\varepsilon (E_{\beta}(Z^{(\beta)}) +E_{\beta}(V^{(\beta)})), \\
           \left|\int_\Omega\sum\limits_{i=1}^n b_i^{(\beta)}V_t^{(\beta)}Z^{(\beta)}_{x_i}\dd x \right|\leq C_1\varepsilon (E_{\beta}(Z^{(\beta)}) +E_{\beta}(V^{(\beta)})), \\
           \left|\int_\Omega\sum\limits_{i,j=1}^n\big(\partial_ta_{ij}^{(\beta)} Z_{x_i}^{(\beta)}V^{(\beta)}_{x_j}\dd x\right|\leq C_1\varepsilon (E_{\beta}(Z^{(\beta)}) +E_{\beta}(V^{(\beta)})).
            \end{split}
            \end{equation}

			Similar to \eqref{uni:6}, we obtain that for any $t\in [0,T]$
            \begin{equation}\label{uni:8}
            \|H^{(\beta)}\|_{L^2}\leq \tilde{C}_H (1-\delta)^{\beta-2}\varepsilon^2,     \int_\Omega|H^{(\beta)}Z^{(\beta)}_t|\leq \tilde{\tilde{C}}_H (1-\delta)^{2\beta-4} \varepsilon^3,
            \end{equation}
            where $\tilde{C}_H, \tilde{\tilde{C}}_H$ are constants independent with $\beta,\delta$ and $\varepsilon$.

           Combining \eqref{est:E1ZT-1}--\eqref{est:E1ZT-4} with \eqref{uni:8}, we have

			$$
			\begin{aligned}
				& \frac{\dd}{\dd t}\int_\Omega\Big(Z_t^{(\beta)}V^{(\beta)}_t+\sum\limits_{i,j=1}^na_{ij}^{(\beta)}Z_{x_i}^{(\beta)}V^{(\beta)}_{x_j}\Big)\dd x+2\int_\Omega\chi \big|Z_t^{(\beta)}\big|^2\dd x \\
				\leq & ~ 3C_1\varepsilon(E_{\beta}(Z^{(\beta)}) +E_{\beta}(V^{(\beta)}))  + \tilde{\tilde{C}}_H  (1-\delta)^{2\beta-4} \varepsilon^3.
			\end{aligned}
			$$
			
			Integrating with respect to
			$t$ over the interval $(0,T)$ and using the initial condition $(V^{(\beta)}(0), V_t^{(\beta)}(0))=(0,0)$,  we arrive at
			$$
			\begin{aligned}
				& \int_\Omega\Big(Z_t^{(\beta)}V^{(\beta)}_t+\sum\limits_{i,j=1}^na_{ij}^{(\beta)}Z_{x_i}^{(\beta)}V^{(\beta)}_{x_j}\Big)\dd x\Big|_{t=T}+2\int_0^T\int_\Omega \chi\big| Z_t^{(\beta)}\big|^2\dd x\dd t \\
				\leq~ & \int_0^T \left(3C_1\varepsilon(E_{\beta}(Z^{(\beta)}) +E_{\beta}(V^{(\beta)}))  + \tilde{\tilde{C}}_H  (1-\delta)^{2\beta-4} \varepsilon^3\right) \dd t \\
				 =~ & 3C_1\varepsilon \int_0^T (E_{\beta}(Z^{(\beta)}) +E_{\beta}(V^{(\beta)})) \dd t + T\tilde{\tilde{C}}_H  (1-\delta)^{2\beta-4} \varepsilon^3 .
			\end{aligned}
			$$

			Combining this with the energy estimate  \eqref{VtoZ}  of $V^{(\beta)}$, we obtain
			\begin{equation}\label{EZV1}
			\begin{aligned}
				& E_\beta\big(V^{(\beta)}(T)\big)+\int_\Omega V_t^{(\beta)}Z_t^{(\beta)} \dd x+\sum\limits_{i,j=1}^n\int_\Omega a_{ij}^{(\beta)}V^{(\beta)}_{x_i}\cdot Z^{(\beta)}_{x_j}\dd x \\
				\leq & ~ 3C_1\varepsilon \int_0^T (E_{\beta}(Z^{(\beta)}) +E_{\beta}(V^{(\beta)})) \dd t + T(C_{VtoZ}(\varepsilon)\tilde{\tilde{C}}_F+ \tilde{\tilde{C}}_H)  (1-\delta)^{2\beta-2} \varepsilon^3\\
&+C_{VtoZ}(\varepsilon)  \int_0^T\int_\Omega\big|\chi\cdot Z_t^{(\beta)}\big|^2\dd x \dd t-2\int_0^T\int_\Omega \chi\big| Z_t^{(\beta)}\big|^2\dd x\dd t.
			\end{aligned}
			\end{equation}
			
            Since \eqref{induction4} and \eqref{induction5} are assumed to hold for $\alpha=\beta$, we have
            \begin{equation}\label{EZV2}
            (E_{\beta}(Z^{(\beta)}) +E_{\beta}(V^{(\beta)}))\leq C_2(1-\delta)^{2\beta-2} \varepsilon^2
            \end{equation}
            where $C_2:=C^2_{Z,mid,0}+C^2_{V,mid,0}$.

			{\color{black}Taking $\varepsilon_{lem}\leq \varepsilon_{2,0}$, where $\varepsilon_{2,0}$ is sufficiently small, such that $C_{VtoZ}(\varepsilon_{2,0})=e^{TC_{coe}\varepsilon_{2,0} }\leq \frac{3}{2}$}, and noting that $0\leq\chi_\omega(x)\leq 1$, we obtain
  \begin{equation}\label{EZV3}
  C_{VtoZ}(\varepsilon)  \int_0^T\int_\Omega\big|\chi\cdot Z_t^{(\beta)}\big|^2\dd x \dd t-2\int_0^T\int_\Omega \chi\big| Z_t^{(\beta)}\big|^2\dd x\dd t\leq -\frac{1}{2}\int_0^T\int_\Omega \chi\big| Z_t^{(\beta)}\big|^2\dd x\dd t.
  \end{equation}

          Combining \eqref{EZV1}, \eqref{EZV3} with \eqref{EW1}, we deduce that
			\begin{equation}\label{EZV4}
            E_\beta (Z^{(\beta+1)}(T))\leq E_\beta\big(Z^{(\beta)}(T)\big)- \frac{1}{2}\int_0^T\int_\Omega \chi\big| Z_t^{(\beta)}\big|^2\dd x\dd t+C_3 (1-\delta)^{2\beta-4} \varepsilon^3,
            \end{equation}
			where $C_3=3TC_1C_2+\tilde{\tilde{C}}_F+ \tilde{\tilde{C}}_H$ is a constant independent of $\beta$ and $\varepsilon$.

            Choosing $\varepsilon_{lem} $ small enough, such that
            \begin{equation}\label{cond:epsilon:obs}
            C(C_\Omega,\nu_0,n,s)\varepsilon_{lem}\leq \varepsilon_{obs},
            \end{equation}
			where $\varepsilon_{obs}$ is given in  Theorem \ref{fully nonlinear observability}, and recalling \eqref{uni:8}, we can then apply Theorem \ref{fully nonlinear observability} to System \eqref{eqnofZalpha} for $Z^{(\beta)}$, and obtain the following observability inequality
\begin{equation}\label{ineq:obsl0}			
E_\beta (Z^{(\beta)}(T))\leq D\int_0^T\int_\omega\big|Z_t^{(\beta)}\big|^2\dd x\dd t+C_{Z,0}(1-\delta)^{2\beta-4}\varepsilon^4,
\end{equation}
            for some constant $D$ and $C_{Z,0}$. Thus we obtain
			\begin{equation}\label{W1}
				\begin{aligned}
					 E_\beta (Z^{(\beta+1)}(T)))
					\leq & ~ \Big(1-\frac{1}{2D}\Big)E_\beta (Z^{(\beta)}(T))+C_{Z,0}(1-\delta)^{2\beta-4}\varepsilon^4 +  C_3 (1-\delta)^{2\beta-4} \varepsilon^3.
				\end{aligned}
			\end{equation}
					
			Taking $\delta>0$ small enough such that
			\begin{equation}\label{cond:Delta}
				1-\frac{1}{2D}\leq (1-\delta)^3,
			\end{equation}
            and recalling \eqref{est:EnergyET}, we obtain that
			\begin{align*}
				& ~ \big\|Z^{(\beta+1)}(T)\big\|_{\mathcal{H}^1}^2+\big\|Z_t^{(\beta+1)}(T)\big\|_{L^2}^2 \\
				\leq & ~ (1-\delta)^3\frac{2+C_{coe,1}\varepsilon}{2-C_{coe,1}\varepsilon}C_{Z_T,mid,0}(1-\delta)^{2\beta-2}\varepsilon^2\\ &+(2-C_{coe,1}\varepsilon )^{-1} C_{Z,0}(1-\delta)^{2\beta-4}\varepsilon^4 + (2-C_{coe,1}\varepsilon)^{-1} C_3 (1-\delta)^{2\beta-4} \varepsilon^3.
			\end{align*}
             Taking $ \varepsilon_{3}$ small enough such that
             \begin{equation}\label{cond:coe2}
             (1-\delta)\cdot \frac{2+C_{coe,1}\varepsilon_{3}}{2-C_{coe,1}\varepsilon_{3}}\leq  1-\frac{\delta}{2},
             \end{equation}
			 and
\begin{equation}\label{cond:coe3}
					\left(1-\frac{\delta}{2}\right)\left(C_{Z_T,mid,0}+(2-C_{coe,1}\varepsilon_{3} )^{-1}(1-\delta)^{-4}(C_{Z,0}\varepsilon_{3}+C_3)\varepsilon_{3}\right)<C_{Z_T,mid,0}.
			\end{equation}
Therefore,  choosing $\varepsilon_{lem}\leq \varepsilon_{3}$, we obtain
			\begin{align}
				\big\|Z^{(\beta+1)}(T)\big\|_{\mathcal{H}^1}^2+\big\|Z_t^{(\beta+1)}(T)\big\|_{L^2}^2\leq C_{Z_T,mid,0}(1-\delta)^{2\beta}\varepsilon^2.
			\end{align}
            Thus, we complete the proof of Claim \ref{claim:l0}.
            \end{proof}

            Similarly, using the standard energy estimate for $Z^{(\beta+1)}$,
            we obtain
            \begin{align}
                E_\beta(Z^{(\beta+1)})(t)
		    	\leq C_{ZtoZ_T}(\varepsilon) \left( E_\beta(Z^{(\beta+1)})(T) + T\tilde{\tilde{C}}_H(1-\delta)^{2\beta} \varepsilon^3 \right),
            \end{align}
            where $C_{ZtoZ_T}(\varepsilon) =e^{C_{Z_T}T \varepsilon}$ and $C_{Z_T}, \tilde{\tilde{C}}_H$ are constants independent of $\beta$ and $\varepsilon$.

            Choosing $\varepsilon_{lem}<\varepsilon_{ZtoZT}$ small enough and $C_{Z,mid,0}=2C_{Z_T,mid,0}$ such that
            \begin{equation}\label{cond:varepsilon3}
            C_{ZtoZ_T}(\varepsilon_{ZtoZT})\leq  \frac{3}{2},
            \end{equation}
            and
            \begin{equation}
            \frac{3}{2}\cdot\frac{2+C_{coe,1}\varepsilon_{ZtoZT}}{2-C_{coe,1}\varepsilon_{ZtoZT} } C_{Z_T,mid,0}+T\tilde{\tilde{C}}_H(1-\delta)^{-2} \varepsilon_{ZtoZT}\leq C_{Z,mid,0},
            \end{equation}
            we ensure that $Z^{(\beta+1)}$ in \eqref{induction4} holds for $l=0$.

            Finally,  given that $v^{(1)}=V^{(1)}$ and relationships between $v^{(\alpha)}$ and $V^{(\alpha)}$, we can derive the following inequality by \eqref{induction4}:
 $$\big\|\partial_t v^{(\alpha)}\big\|_{L^{2}}^2+\big\| v^{(\alpha)}\big\|_{\mathcal{H}^{1}}^2=\big\|\sum_{\beta=1}^\alpha \partial_t V^{(\beta)}\big\|_{L^{2}}^2+\big\|\sum_{\beta=1}^\alpha  V^{(\beta)}\big\|_{\mathcal{H}^{1}}^2\leq  \frac{(1-\delta)^2-(1-\delta)^{2\alpha}}{1-(1-\delta)^2}C^2_{V,mid,1}\varepsilon^2.$$
				Observing that $\frac{(1-\delta)^2-(1-\delta)^{2\alpha}}{1-(1-\delta)^2}\leq \frac{(1-\delta)^2}{1-(1-\delta)^2}$ and invoking the setting in \eqref{cond:Vconstant}, we can thereby conclude that the inequality \eqref{induction5} holds for the case $l=0$.

            \subsubsection{Inductive step 1: To prove \eqref{induction4} for $ l\geq 1$ and all $\alpha\geq 2$}

           We first consider the equations for $\partial_t^l V^{(\beta+1)}$ and $\partial_t^lZ^{(\beta+1)}$ for $l\geq 1$.
           To achieve this, we apply the differential operator $\partial_t^{k-2}$
  to Equations \eqref{eqnofValpha} and \eqref{eqnofzalpha}, resulting in the following equations:
			\begin{equation}\label{eqn of partial t s V alpha}
				\begin{aligned}
					& \partial_t^{k}V^{(\beta+1)}_{tt}+b_0^{(\beta+1)}\partial_t^{k}V^{(\beta+1)}_t-\sum\limits_{i,j=1}^n\big(a_{ij}^{(\beta+1)}\partial_t^{k}V^{(\beta+1)}_{x_i}\big)_{x_j}+ \tilde{b}^{(\beta+1)}\partial_t^{k}V^{(\beta+1)} \\
					+ & \sum\limits_{i=1}^nb_i^{(\beta+1)}\partial_t^{k}V^{(\beta+1)}_{x_i}=-2\chi\cdot\partial_t^{k}Z_t^{(\beta+1)}+\partial_t^{s-1}F^{(\beta+1)}+F^{(\beta+1,k)},
				\end{aligned}
			\end{equation}
			and
			\begin{equation}\label{eqn of partial t s Z alpha}
				\begin{aligned}
					& \partial_t^{k}Z^{(\beta+1)}_{tt}-b_0^{(\beta+1)}\partial_t^{k}Z^{(\beta+1)}_t-\sum\limits_{i,j=1}^n\big(a_{ij}^{(\beta+1)}\partial_t^{k}Z^{(\beta+1)}_{x_i}\big)_{x_j}+ \tilde{b}_{(\beta+1)}\partial_t^{k}Z^{(\beta+1)} \\
					+ & \sum\limits_{i=1}^nb_i^{(\beta+1)}\partial_t^{k}Z^{(\beta+1)}_{x_i}=\partial_t^{k}H^{(\beta+1)}+H^{(\beta+1,k)},
				\end{aligned}
			\end{equation}
			where
			\begin{align*}
				F^{(\beta+1,k)}= & ~ b_0^{(\beta+1)}\partial_t^kV^{(\beta+1)}_t-\sum\limits_{i,j=1}^n\big(a_{ij}^{(\beta+1)}\partial_t^kV^{(\beta+1)}_{x_i}\big)_{x_j}+\tilde{b}^{(\beta+1)} \partial_t^kV^{(\beta+1)}+\sum\limits_{i=1}^nb_i^{(\beta+1)}\partial_t^kV^{(\beta+1)}_{x_i} \\
				& -\partial_t^k\Big(b_0^{(\beta+1)}V^{(\beta+1)}_t-\sum\limits_{i,j=1}^n\big(a_{ij}^{(\beta+1)}V^{(\beta+1)}_{x_i}\big)_{x_j}+\tilde{b}^{(\beta+1)}V^{(\beta+1)}+\sum\limits_{i=1}^n b_i^{(\beta+1)}V^{(\beta+1)}_{x_i}\Big),
			\end{align*}
			and
			$$
			\begin{aligned}
				H^{(\beta+1,k)}= & ~ b_0^{(\beta+1)}\partial_t^kZ^{(\beta+1)}_t-\sum\limits_{i,j=1}^n\big(a_{ij}^{(\beta+1)}\partial_t^kZ^{(\beta+1)}_{x_i}\big)_{x_j}+\tilde{b}^{(\beta+1)} \partial_t^kZ^{(\beta+1)}+\sum\limits_{i=1}^nb_i^{(\beta+1)}\partial_t^kZ^{(\beta+1)}_{x_i} \\
				& -\partial_t^k\Big(b_0^{(\beta+1)}Z^{(\beta+1)}_t-\sum\limits_{i,j=1}^n\big(a_{ij}^{(\beta+1)}Z^{(\beta+1)}_{x_i}\big)_{x_j}+\tilde{b}^{(\beta+1)}Z^{(\beta+1)}+\sum\limits_{i=1}^n b_i^{(\beta+1)}Z^{(\beta+1)}_i\Big).
			\end{aligned}
			$$

Combining \eqref{coe:FH} in Lemma \ref{lem:FH} with the assumption that \eqref{induction4} and \eqref{induction5} hold for $\alpha\leq \beta$, we obtain that for any $t\in [0,T]$, $k=1,\cdots, s-2$
		    \begin{equation*}\begin{split}
		    	\|\partial_t^k F^{(\beta)}\|_{L^2}\leq \tilde{C}_{F,k}  (1-\delta)^{\beta-1} \varepsilon^2 ,
                \|\partial_t^k H^{(\beta)}\|_{L^2}\leq \tilde{C}_{H,k}  (1-\delta)^{\beta-1} \varepsilon^2 ,
		      \end{split}
            \end{equation*}
            where $\tilde{C}_{F,k} $ and $ \tilde{C}_{H,k}$ are constants independent of $M$, $\beta$ and $\varepsilon$.

            Next, we state the following claim:

		   \begin{claim}\label{lem:FH:est} For any $\beta\geq 1$ and $1\leq k\leq s-2$,  $F^{(\beta+1,k)}, H^{(\beta+1,k)}\in C(0,T;\mathcal{H}^1)$ and
		   there exist constants $C_{F,k},C_{H,k}$, independent of $\beta$ and $\varepsilon$, such that
		    \begin{equation}\label{FH:est}\begin{split}
		    \|F^{(\beta+1,k)}\|_{L^2}&\leq C_{F,k}\varepsilon \big(\sum_{p=0}^{k}\|\partial_t^pV^{\beta+1}\|_{\mathcal{H}^1}+\|\partial_t^{p+1}V^{\beta+1}\|_{L^2}+\|\partial_t^pZ^{(\beta+1)}\|^2_{L^{2}}\big),\\ \|H^{(\beta+1,k)}\|_{L^2}&\leq C_{H,k}\varepsilon \big(\sum_{p=0}^{k}\|\partial_t^pZ^{\beta+1}\|_{\mathcal{H}^1}+\|\partial_t^{p+1}Z^{\beta+1}\|_{L^2}\big).
		    \end{split}
            \end{equation}
		    \end{claim}

		   \begin{proof}
		   	 We first estimate $F^{(\beta+1,k)}$; the estimate for $H^{(\beta+1,k)}$  follows similarly. By expanding
		   	$F^{(\beta+1,k)}$ into commutators, we focus on the first term:
		   	\begin{equation}
		   		b_0^{(\beta+1)}\partial_t^kV^{(\beta+1)}_t-\partial_t^k(b_0^{(\beta+1)}V^{(\beta+1)}_t)=-\sum_{l=0}^{k-1}C_k^l \partial_t^lV^{(\beta+1)}_t\partial_t^{k-l}(b_0^{(\beta+1)}).
		   	\end{equation}
		  Each term in the above expression contains at least the first-order time derivative of   $(b_0^{(\beta+1)}$, and at most the $k-1$-th time derivative of  $V^{(\beta+1)}_t$  and $k$-th time derivative of $(b_0^{(\beta+1)}$, Consequently, by H\"{o}lder's inequality, we have:
		   	\begin{equation}
		   		\|\big [b_0^{(\beta+1)}\partial_t^kV^{(\beta+1)}_t-\partial_t^k(b_0^{(\beta+1)}V^{(\beta+1)}_t))\big ]\|_{L^2}\leq C(s,l)(\sum_{p=0}^k\|\partial_t^p b_0^{(\beta+1)}\|_{L^2)})(\sum_{p=0}^k\|\partial_t^pV^{\beta+1}\|_{L^2})
		   	\end{equation}
		   By \eqref{vertify:1}, we have  $\sum_{p=0}^k\|\partial_t^p b_0^{(\beta+1)}\|_{L^2)} \leq C\varepsilon$ for some constant $C$.

Next, we consider the term
$$\partial_t^k\sum\limits_{i,j=1}^n\big(a_{ij}^{(\beta+1)}V^{(\beta+1)}_{x_i}\big)_{x_j}-\sum\limits_{i,j=1}^n\big(a_{ij}^{(\beta+1)}\partial_t^kV^{(\beta+1)}_{x_i}\big)_{x_j}.$$

When $k\leq s-2$, similar estimates can be derived for the remaining commutators in $F^{(\beta+1,k)}$, except for the term:
\begin{equation*}
    \|\big(\partial_t^k\sum\limits_{i,j=1}^n\big(a_{ij}^{(\beta+1)}V^{(\beta+1)}_{x_i}\big)_{x_j}-\sum\limits_{i,j=1}^n\big(a_{ij}^{(\beta+1)}\partial_t^kV^{(\beta+1)}_{x_i}\big)_{x_j}\big)\|_{L^2}\leq C \varepsilon (\sum_{p=0}^{k-1}\|\partial_t^pV^{(\beta+1)}\|_{\mathcal{H}^{2}}).
\end{equation*}

We observe that
\begin{equation}\label{est:x-to-t}
\sum_{p=0}^{k-1}\|\partial_t^pV^{(\beta+1)}\|_{\mathcal{H}^{2}}\leq C  (\sum_{p=0}^{k}\|\partial_t^pV^{(\beta+1)}\|_{\mathcal{H}^{1}}^2+\sum_{p=0}^{k+1}\|\partial_t^pV^{(\beta+1)}\|^2_{L^{2}}+\sum_{p=0}^{k}\|\partial_t^pZ^{(\beta+1)}\|^2_{L^{2}}
\end{equation}
Indeed, using Equation \eqref{eqnofZalpha} and the conditions \eqref{coe:C1} on the coefficients, we deduce that for any  $k\geq 2$,
             		\begin{equation*}
             \begin{aligned}
                        \big\|\sum\limits_{i,j=1}^n\big(a_{ij}^{(\beta+1)}\partial_t^{k-2}V^{(\beta+1)}_{x_i}\big)_{x_j}\big\|_{L^2}^2\leq &\big\| \partial_t^{k-2}V^{(\beta+1)}_{tt}\big\|_{L^2}^2+\big\|b_0^{(\beta+1)}\partial_t^{k-2}V^{(\beta+1)}_t\big\|_{L^2}^2 \\ &+\|\tilde{b}^{(\beta+1)}\partial_t^{k-2}V^{(\beta+1)}\|_{L^2}^2 +
                        \big\|\sum\limits_{i=1}^nb_i^{(\beta+1)}\partial_t^{k-2}V^{(\beta+1)}_{x_i}\big\|_{L^2}^2 \\ &+\big\|\partial_t^{k-2}F^{(\beta+1)}\|_{L^2}^2+\|F^{(\beta+1,k-2)}\|^2_{L^2}+\big\|2\chi \partial_t^{k-2}Z_t^{(\beta+1)}\|_{L^2}^2.
                    \end{aligned}
                    \end{equation*}
                    By the expression of $F^{(\beta+1,k)}$ and the fact that $F^{(\beta+1,0)}=0$, combined with the conditions \eqref{coe:C1}, we can deduce that
                    $$
                    \big\|\sum\limits_{i,j=1}^n\big(a_{ij}^{(\beta+1)}\partial_t^{k-2}Z^{(\beta+1)}_{x_i}\big)_{x_j}\big\|_{L^2}^2\leq C\big(\sum_{p=0}^{k}\|\partial_t^pV^{(\beta+1)}\|_{\mathcal{H}^{1}}^2+\sum_{p=0}^{k+1}\|\partial_t^pV^{(\beta+1)}\|^2_{L^{2}}\big).
                    $$
                    Since  $\partial_t^{k-2}Z^{(\beta+1)}$ is well-defined and belongs to  $ \mathcal{H}^2$ for each $k=2,\cdots,s$ and due to the conditions \eqref{coe:C1} of $a_{ij}^{(\beta+1)}$,   elliptic theory implies the desired estimate.

Thus, the proof of the claim is complete.
		   \end{proof}

            Thanks to the estimations of $(F^{(\beta+1,k)},H^{(\beta+1,k)})$ in Claim \ref{lem:FH:est},  together with the coefficient estimates \eqref{coe:C1}, we can apply Corollary \ref{cor:well-reg} to obtain the well-posedness of  $\partial_t^kV^{(\beta+1)}$ and $\partial_t^k Z^{(\beta+1)}$ in $C(0,T; \mathcal{H}^{s-k})$.

            To complete the proof of \eqref{induction4}, by induction,  it suffices to prove \eqref{induction4} for the case $l=k,\alpha=\beta+1$ under the assumption that \eqref{induction4} holds for $l\leq k-1, k\geq 1$ and for any $\alpha$, as well as for $l\leq s-1, \alpha\leq \beta$.

            Similar to the case $l=0$, we now prove the following inequality:
           \begin{claim} There exists a small positive constant $\varepsilon_{ZT}$, such that if $\varepsilon\leq \varepsilon_{ZT}$, then for any integer $\alpha\geq 1$, any positive integer $l\leq s-2$,  $Z^{(\alpha)}$ satisfies
           \begin{equation}\label{induction-ZH}
				\begin{aligned}
					 \|\partial_t^{l+1} Z^{(\alpha)}(T)\|_{L^{2}}^2+\|\partial_t^{l} Z^{(\alpha)}(T)\|_{\mathcal{H}^{1}}^2\leq C_{Z_T,mid,l}(1-\delta)^{2\alpha-2}{M^{2l}}\varepsilon^2,
				\end{aligned}
			\end{equation}
where $C_{Z_T,mid,l}$ and $ M$ are positive constants independent of $\alpha$ and $\varepsilon$.
  \end{claim}
\begin{proof}
By employing mathematical induction, we assume that \eqref{induction4} and \eqref{induction5} hold for all $\alpha\leq \beta$ when $l\leq s-1$, as well as for all $\alpha$ when $l\leq k-1$ and $s-1\geq k\geq 1$. Under these assumptions, it remains to prove that \eqref{induction-ZH} holds for $\alpha=\beta+1$ and $l=k$.
             For notational simplicity, we define $\tilde{V}^{(\beta+1,k)}=\partial_t^k V^{(\beta+1)},\tilde{Z}^{(\beta+1,k)}=\partial_t^k Z^{(\beta+1)}$ and for any $t\in [0,T]$, we define
             $$\tilde{W}^{(\beta+1,k)}=\tilde{V}^{(\beta+1,k)}+\tilde{Z}^{(\beta+1,k)}, \tilde{w}^{(\beta+1,k)}=\tilde{v}^{(\beta+1,k)}+\tilde{z}^{(\beta+1,k)}. $$

             The estimates \eqref{induction-Z} for $l=k$ directly come from the estimation of $E_{\beta}(\tilde{W}^{(\beta,k)}(T))-E_{\beta}(\tilde{Z}^{(\beta,k)}(T))$ and the relation between $E_{\beta}(\tilde{W}^{(\beta,k)}(T))$ and $E_{\beta}(\tilde{Z}^{(\beta+1,k)}(T))$.

             The estimate of $E_{\beta}(\tilde{W}^{(\beta,k)}(T))-E_{\beta}(\tilde{Z}^{(\beta,k)}(T))$ follows a similar approach to the case $l=0$. We first observe that:
              \begin{equation}\label{EZVl}
			\begin{aligned}
              & E_{\beta}(\tilde{W}^{(\beta,k)}(T))-E_{\beta}(\tilde{Z}^{(\beta,k)}(T)) \\
              =~&\bigg(\int_\Omega \tilde{V}_t^{(\beta)}\tilde{Z}_t^{(\beta)} \dd x +\sum\limits_{i,j=1}^n\int_\Omega a_{ij}^{(\beta)}\tilde{V}^{(\beta)}_{x_i}\cdot \tilde{Z}^{(\beta)}_{x_j}\dd x\bigg)\Big|_{t=T} +E_\beta\big(\tilde{V}^{(\beta)}(T)\big).
            \end{aligned}
            \end{equation}
            Similarly, by multiplying the equation for
 $\tilde{V}^{(\beta)}$ by $\partial_t\tilde{V}^{(\beta)}$ and integrating over $\Omega$, we obtain:
             \begin{equation*}
		    	\begin{aligned}
		    		 \frac{\dd}{\dd t}E_\beta(\tilde{V}^{(\beta)}) \leq ~&  \big\|\chi\cdot \tilde{Z}_t^{(\beta)}\big\|_{L^2}^2+ C_{V,2,k} \varepsilon E_\beta(\tilde{V}^{(\beta)}) \\
                     &+ \tilde{\tilde{C}}_{F,k} (1-\delta)^{2\beta-2} \varepsilon^3+\bigg|\int_\Omega F^{(\beta,k)}\partial_t\tilde{V}^{(\beta)}\dd x\bigg|,
		    	\end{aligned}
		    \end{equation*}
             for some constants $C_{V,2,k}, \tilde{\tilde{C}}_{F,k}$, independent of $\beta$ and $\varepsilon$.

            Since \eqref{induction4} is assumed to hold for $\alpha=\beta$, by H\"{o}lder's inequality and Lemma \ref{lem:FH:est}, we have:
            $$\bigg|\int_\Omega F^{(\beta,k)}\partial_t\tilde{V}^{(\beta)}\dd x\bigg|\leq C_{V,4,k}(1-\delta)^{2\beta}\varepsilon^3.$$
            for some constant $C_{V,4,k}$, independent of $\beta$ and $\varepsilon$. Thus, by applying Gronwall's inequality, we immediately obtain:
             \begin{equation}\label{Vtov}
		    	E_\beta(\tilde{V}^{(\beta)}(t))
		    	\leq C_{\tilde{V}to\tilde{V}}(\varepsilon) \left(E_\beta(\tilde{V}^{(\beta)}(0))+ \int_0^T\int_\Omega\big|\chi\cdot \tilde{Z}_t^{(\beta)}\big|^2\dd x \dd t + C_{V,5,k}(1-\delta)^{2\beta-2} \varepsilon^3 \right),
		    \end{equation}
 		    where $C_{\tilde{V}to\tilde{V}}(\varepsilon)\leq C_{\tilde{V}to\tilde{V}}(\varepsilon_{2,k})=\frac{3}{2}<2$  is a constant when choosing $\varepsilon_{lem}\leq \varepsilon_{2,k}$.

            The next step is to estimate $E_\beta(\tilde{V}^{(\beta)}(0))$. The equation \eqref{eqnofValpha} of $V^{(\beta)}$ can be written as
			\begin{align}\label{equation:V:beta:rewrite}
				\partial_t^2V^{(\beta)}+\partial_tV^{(\beta)}=\Delta V^{(\beta)}-2\chi\cdot\partial_tZ^{(\beta)}+V_{error}^{(\beta)},
			\end{align}
			where
			{\color{black}\begin{align}\label{V:error}
					V_{error}^{(\beta)}:= & ~ \big(1-b_0^{(\beta)}\big)V_t^{(\beta)}-\big(\tilde{b}^{(\beta)}\big)V^{(\beta)}-\sum\limits_{i=1}^nb_i^{(\beta)}V_{x_i}^{(\beta)} \\
                    &+F^{(\beta)}+\sum\limits_{i,j=1}^n \Big(\big(a_{ij}^{(\beta)}-\delta_{ij} \big)V_{x_i}^{(\beta)}\Big)_{x_j}. \nonumber
				\end{align}
				Since \eqref{induction4} are assumed to hold for $\alpha=\beta$ and recalling the estimates on coefficients, applying Lemma \ref{lem:GG} with $G=V_{error}^{(\beta)}$, we have for any $m\leq s-2-k$
				\begin{align}
					\|\partial_t^kV_{error}^{(\beta)}\|_{\mathcal{H}^{m}}\leq C_{V,error} \big((1-\delta)^{\beta-1}M^{k+m+2}\varepsilon^2\big),
			\end{align}}
			for some constant $C_{V,error}$ independent of $\beta$ and $\varepsilon$.
			Differentiating \eqref{equation:V:beta:rewrite} with respect to $t$ for $k-2$ times, we obtain:

 {\color{black}When $k\geq 1$ is odd,
				\begin{align}\label{V:k:odd}
					\partial_t^kV^{(\beta)}=\partial_t\Delta^{\frac{k-1}{2}}V^{(\beta)}-\sum\limits_{l=0}^{\frac{k-1}{2}}\partial_t^{k-2l-2}\Delta^l\Big(2\chi Z^{(\beta)}_t+\partial_tV^{(\beta)}-V_{error}^{(\beta)}\Big),
			\end{align}
when $k\geq 2$ is even,
				\begin{align}\label{V:k:even}
					\partial_t^kV^{(\beta)}=\Delta^{\frac{k}{2}}V^{(\beta)}-\sum\limits_{l=0}^{\frac{k}{2}-1}\partial_t^{k-2l-2}\Delta^l\Big(2\chi Z^{(\beta)}_t+\partial_tV^{(\beta)} -V_{error}^{(\beta)}\Big).
				\end{align} }

				We note that $\chi$ is a smooth bounded function so that there exist a sequence constants $C_{\chi,p}, p=0,1,\cdots,k$ such that for any $u\in \mathcal{H}^k$,
$$
\sum_{p=0}^{l}\|\Delta^p (2\chi \partial_t^m u)\|_{L^2}\leq C_{\chi,l} \sum_{p=0}^{l}\|\Delta^p \partial_t^m u\|_{L^2}=C_{\chi,l} \sum_{p=0}^{l}\| \partial_t^m u\|_{\mathcal{H}^{2p}}
$$
			
			Noting that from initial data for $V^{(\beta)}$ in \eqref{eqnofValpha}, we have $(\Delta^mV^{(\beta)}(0),\partial_t\Delta^mV^{(\beta)}(0))=(0,0)$ for any integer $m\geq 0$. Moreover, since \eqref{induction4} and \eqref{induction5} hold for $\alpha=\beta$, we obtain
			\begin{align}\label{comb:3}
				& ~ \big\|\partial_t^kV^{(\beta)}(0)\big\|_{L^2}^2+\big\|\partial_t^{k-1}V^{(\beta)}(0)\big\|_{\mathcal{H}^1}^2 \\
				\leq & ~ \bigg(\sum_{i=1}^{k-1} C^2_{V,mid,i}M^{2i}+C_{\chi,k}\sum_{i=1}^{k-1}C^2_{Z,mid,i}M^{2i}\bigg)(1-\delta)^{2\beta-2}\varepsilon^2+ kC^2_{V,error}(1-\delta)^{2\beta-2}M^{2k}\varepsilon^4, \nonumber
			\end{align}
		

           Similarly, by multiplying the equation for $\tilde{V}^{(\beta)}$ by $\partial_t\tilde{Z}^{(\beta)}$ and adding it to the equation for $\tilde{Z}^{(\beta)}$ multiplied by $\partial_t\tilde{V}^{(\beta)}$, and integrating over $\Omega$, we obtain:
           \begin{align*}
           &\int_\Omega \tilde{V}_t^{(\beta)}\tilde{Z}_t^{(\beta)}\dd x +\sum\limits_{i,j=1}^n\int_\Omega a_{ij}^{(\beta)}\tilde{V}^{(\beta)}_{x_i}\cdot \tilde{Z}^{(\beta)}_{x_j}\dd x\Big|_{t=T} \\
           \leq~&  - 2\int_0^T\int_\Omega \chi\big| \tilde{Z}_t^{(\beta)}\big|^2\dd x\dd t
	    +C_{5,k} (1-\delta)^{2\beta-2} \varepsilon^3+\int_\Omega 
           \tilde{V}_t^{(\beta)}\tilde{Z}_t^{(\beta)} \dd x \\ &+\sum\limits_{i,j=1}^n\int_\Omega a_{ij}^{(\beta)}\tilde{V}^{(\beta)}_{x_i}\cdot \tilde{Z}^{(\beta)}_{x_j}\dd x\Big|_{t=0}.
           \end{align*}
           If we regard $\partial_t^kF^{\beta}+F^{(\beta,k)}$ as an external force term $f$, then the equations satisfied by $(\tilde{V}^{(\beta)},\tilde{Z}^{(\beta)})$ have the same coefficients as those for $(V^{(\beta)},Z^{(\beta)})$. Therefore, we can then apply Theorem \ref{fully nonlinear observability} to equation for $\tilde{Z}^{(\beta)}$, to obtain the following observability inequality:
\begin{equation}\label{ineq:obslk}			
E_\beta (\tilde{Z}^{(\beta)}(T))\leq D\int_0^T\int_\omega\big|\tilde{Z}^{(\beta)}_t\big|^2\dd x\dd t+C_6\int_0^T\|\partial_t^kF^{\beta}+F^{(\beta,k)}\|_{L^2}^2\dd t,
\end{equation}
            for the same constant $D$ and some constant $C_6>0$ independent of $\beta$ and $\varepsilon$.

            Recalling \eqref{coe:FH} in Lemma \ref{lem:FH} and \eqref{FH:est} in Lemma \eqref{lem:FH:est}, we have
			$$
\int_0^T\|\partial_t^kF^{\beta}+F^{(\beta,k)}\|_{L^2}^2\dd t\leq C_{F,k}(1-\delta)^{2\beta-2}\varepsilon^4.
            $$

           Next, we estimate the term $\int_\Omega \tilde{V}_t^{(\beta)}\tilde{Z}_t^{(\beta)} \dd x  +\sum\limits_{i,j=1}^n\int_\Omega a_{ij}^{(\beta)}\tilde{V}^{(\beta)}_{x_i}\cdot \tilde{Z}^{(\beta)}_{x_j}\dd x\big|_{t=0}$. By H\"{o}lder's inequality, this term is bounded by
           $$
           4E_\beta(\tilde{Z}^{(\beta)})^\frac{1}{2}E_\beta(\tilde{V}^{(\beta)})^\frac{1}{2}\big|_{t=0}.
           $$
           Using \eqref{induction4} with $\alpha=\beta$ and \eqref{comb:3}, we have
           \begin{align*}
           &\int_\Omega \tilde{V}_t^{(\beta)}\tilde{Z}_t^{(\beta)} \dd x+\sum\limits_{i,j=1}^n\int_\Omega a_{ij}^{(\beta)}\tilde{V}^{(\beta)}_{x_i}\cdot \tilde{Z}^{(\beta)}_{x_j}\dd x\big|_{t=T} \\
           \leq &  - 2\int_0^T\int_\Omega \chi\big| \tilde{Z}_t^{(\beta)}\big|^2\dd x\dd t
	    +\tilde{C}_{5,k} (1-\delta)^{2\beta-2} \varepsilon^3 \\
           &+ C_{Z,mid,k}M^k\bigg(\sum_{i=0}^{k-1} C^2_{V,mid,i}M^{2i}+\sum_{i=0}^{k-1}C_{\chi,i}C^2_{Z,mid,i}M^{2i}\bigg)^{\frac{1}{2}}(1-\delta)^{2\beta-2}\varepsilon^2,
           \end{align*}
           where $\tilde{C}_{5,k}=C_{5,k}+C_{Z,mid,k}C_{V,error}^{\frac{1}{2}}(1-\delta)^2$. Thus combining this, \eqref{comb:3} with \eqref{cond:Delta}, we obtain
           \begin{align}\label{est:EWT-EZT}
           E_\beta(\tilde{W}^{(\beta)})(T)\leq~ &(1-\delta)^3E_\beta(\tilde{Z}^{(\beta)})(T) \nonumber \\
           &+ C_{Z,mid,k}M^k\bigg(\sum_{i=0}^{k-1} C^2_{V,mid,i}M^{2i}+C_{\chi,k}\sum_{i=1}^{k-1}C^2_{Z,mid,i}M^{2i}\bigg)^{\frac{1}{2}}(1-\delta)^{2\beta-2}\varepsilon^2 \nonumber \\
           &+ \bigg(\sum_{i=0}^{k-1} C^2_{V,mid,i}M^{2i}+C_{\chi,k}\sum_{i=0}^{k-1}C^2_{Z,mid,i}M^{2i}\bigg)(1-\delta)^{2\beta-2}\varepsilon^2 \\
           &+\tilde{C}_{5,k}(1-\delta)^{2\beta-2} \varepsilon^3 +C_{V,error}(1-\delta)^{2\beta-2} \varepsilon^4. \nonumber
           \end{align}

%
           To derive the relationship between  $E_\beta(\partial_t^kW^{(\beta)})(T)$ and $E_\beta(\partial_t^kZ^{(\beta+1)})(T)$. we start with the equation for  $W^{(\beta)}$:
			\begin{align}\label{eqn:W:beta}
				\partial_t^2W^{(\beta)}=\Delta W^{(\beta)}+2(1-\chi)\partial_tZ^{(\beta)}-\partial_tV^{(\beta)}+W^{(\beta)}_{error},
			\end{align}
			{\color{black}where
				\begin{align}\label{W:error}
					W^{(\beta)}_{error}:= & ~ \big(1-b_0^{(\beta)}\big)V_t^{(\beta)}-\tilde{b}^{(\beta)}V^{(\beta)}-\sum\limits_{i=1}^nb_i^{(\beta)}V_{x_i}^{(\beta)}+F^{(\beta)} \nonumber \\
					& ~ +\big(b_0^{(\beta)}-1\big)Z_t^{(\beta)}-\tilde{b}^{(\beta)}Z^{(\beta)}-\sum\limits_{i=1}^nb_i^{(\beta)}Z_{x_i}^{(\beta)}+H^{(\beta)} \\
					& ~ +\sum\limits_{i,j=1}^n \Big(\big(a_{ij}^{(\beta)}-\delta_{ij}\big)V_{x_i}^{(\beta)}\Big)_{x_j} +\sum\limits_{i,j=1}^n\Big(\big(a_{ij}^{(\beta)}-\delta_{ij}\big) Z_{x_i}^{(\beta)}\Big)_{x_j}. \nonumber
				\end{align}
				Given that \eqref{induction4} holds for $\alpha=\beta$ and recalling the estimates on the coefficients, we apply Lemma \ref{lem:GG} with $G=W_{error}^{(\beta)}$. This yields:  for any  $m\leq s-2-k$
				\begin{align}
					\|\partial_t^kW^{(\beta)}_{error}\|_{\mathcal{H}^m}=C_{Werror}M^{m+k+2}(1-\delta)^\beta\varepsilon^2,
			\end{align}}
           for some constant $C_{Werror}$ independent of $M$, $\alpha$ and $\varepsilon$.

			Differentiating \eqref{eqn:W:beta} with respect to $t$ for $k-2$ times, {\color{black} we obtain: when $k\geq 1$ is odd,
				\begin{equation}\label{W:k:odd}				\partial_t^kW^{(\beta)}=\Delta^{\frac{k-1}{2}}W_t^{(\beta)}+\sum\limits_{l=0}^{\frac{k-1}{2}}\partial_t^{k-2l-2}\Delta^l\Big(2(1-\chi)\partial_tZ^{(\beta)}-\partial_t V^{(\beta)}+W^{(\beta)}_{error}\Big);
			\end{equation}}				
when $k\geq 2$ is even,
\begin{equation}\label{W:k:even}				\partial_t^kW^{(\beta)}=\Delta^{\frac{k}{2}}W^{(\beta)}+\sum\limits_{l=0}^{\frac{k}{2}-1}\partial_t^{k-2l-2}\Delta^l\Big(2(1-\chi)\partial_tZ^{(\beta)}-\partial_tV^{(\beta)}+W^{(\beta)}_{error}\Big).
				\end{equation}

Noting that for any integer $m\geq 0$,
			$$\Delta^m W^{(\beta)}(T)=\Delta^mZ^{(\beta+1)}(T),\qquad\Delta^m\partial_tW^{(\beta)}(T)=\Delta^m\partial_tZ^{(\beta+1)}(T),$$
			and combining \eqref{eqn:Z:alpha} with $\alpha=\beta+1$ and $t=T$, we obtain: when $k\geq 2$ is odd,
            \begin{equation}\begin{split}
            \partial_t^kW^{(\beta)}(T)=~&\partial_t^k Z^{(\beta+1)}(T)+\sum\limits_{p=0}^{\frac{k}{2}-1}\partial_t^{k-2p-2}\Delta^l\Big(2(1-\chi)\partial_tZ^{(\beta)}-\partial_tV^{(\beta)}+W^{(\beta)}_{error}\Big) \\
            &- \sum\limits_{p=0}^{\frac{k}{2}-1}\partial_t^{k-2p-2}\Delta^p\Big(\partial_tZ^{(\beta+1)}+Z^{(\beta+1)}_{error}\Big);
            \end{split}\end{equation}			

when $k\geq 2$ is even,
\begin{equation}\label{W:k:evenT}			\begin{split}	\partial_t^kW^{(\beta)}=~&\partial_t^kZ^{(\beta+1)}+\sum\limits_{p=0}^{\frac{k}{2}-1}\partial_t^{k-2p-2}\Delta^p\Big(2(1-\chi)\partial_tZ^{(\beta)}-\partial_tV^{(\beta)}+W^{(\beta)}_{error}\Big)\\
&-\sum\limits_{p=0}^{\frac{k}{2}-1}\partial_t^{k-2p-2}\Delta^p\Big(\partial_tZ^{(\beta+1)}+Z^{(\beta+1)}_{error}\Big).
		\end{split}		\end{equation}
				
We note that $\chi$ is a smooth bounded function so that there exist a sequence constants $C_{1-\chi,p}, p=0,1,\cdots,k$ such that for any $u\in \mathcal{H}^k$,
$$
\sum_{p=0}^{l}\|\Delta^p [2(1-\chi) u]\|_{L^2}\leq C_{1-\chi,l} \sum_{p=0}^{l}\|\Delta^p u\|_{L^2}=C_{1-\chi,l} \sum_{p=0}^{l}\| u\|_{\mathcal{H}^{2p}}.
$$

	      Combining this with the induction assumption that \eqref{induction4} and \eqref{induction5} hold for  $\alpha\leq \beta, l\leq s-1$ as well as for any $\alpha$ and $l\leq k-1$, we obtain
\begin{align}\label{est:EZT-EZT}
	&\big\|\partial_t^kZ^{(\beta+1)}(T)\big\|_{L^2}^2  +\big\|\partial_t^{k-1}Z^{(\beta+1)}(T)\big\|_{\mathcal{H}^1}^2 \nonumber\\
    \leq~ &	\big\|\partial_t^kW^{(\beta)}(T)\big\|_{L^2}^2+ \big\|\partial_t^{k-1}W^{(\beta)}(T)\big\|_{\mathcal{H}^1}^2 +C_{1-\chi,k-1}\sum\limits_{p=0}^{k-1} C^2_{Z,mid,p}M^{2p}(1-\delta)^{2\beta-2}\varepsilon^2 \nonumber \\
    &+\sum\limits_{p=0}^{k-1} C^2_{V,mid,p}M^{2p}(1-\delta)^{2\beta-2}\varepsilon^2+kC^2_{Werror}M^{2k}(1-\delta)^{2\beta-2}\varepsilon^4 \\
    &+\sum\limits_{p=0}^{k-1} C^2_{Z,mid,p} M^{2p}(1-\delta)^{2\beta}\varepsilon^2+kC^2_{Zerror}M^{2k}(1-\delta)^{2\beta}\varepsilon^4. \nonumber
\end{align}
			
			Combining with \eqref{comb:3}, and noting the relationship \eqref{est:EnergyET}, we obtain
            \begin{align*}
				&~ \big\|\partial_t^kZ^{(\beta+1)}(T)\big\|_{L^2}^2  +\big\|\partial_t^{k-1}Z^{(\beta+1)}(T)\big\|_{\mathcal{H}^1}^2  \\
				\leq & ~ (1-\delta)^3\frac{2+C_{coe,1}\varepsilon}{2-C_{coe,1}\varepsilon}\big\|\partial_t^kZ^{(\beta)}(T)\big\|_{L^2}^2  +\big\|\partial_t^{k-1}Z^{(\beta)}(T)\big\|_{\mathcal{H}^1}^2 \\
                &+ \frac{2+C_{coe,1}\varepsilon}{2-C_{coe,1}\varepsilon}C_{Z,mid,k}M^k\bigg(\sum_{i=0}^{k-1} \big(C^2_{V,mid,i}+C_{\chi,k-1}C^2_{Z,mid,i}\big)M^{2i}\bigg)^{\frac{1}{2}}(1-\delta)^{2\beta-2}\varepsilon^2\\
           &+\frac{2+C_{coe,1}\varepsilon}{2-C_{coe,1}\varepsilon} \bigg(\sum_{i=0}^{k-1} \big(C^2_{V,mid,i}+C_{\chi,k-1}C^2_{Z,mid,i}\big)M^{2i}\bigg)(1-\delta)^{2\beta-2}\varepsilon^2\\
           &+\frac{2+C_{coe,1}\varepsilon}{2-C_{coe,1}\varepsilon}\Big(\tilde{C}_{V,5,k}(1-\delta)^{2\beta-2}  +C^2_{V,error}\varepsilon\Big)(1-\delta)^{2\beta-2} \varepsilon^3\\
                &+\sum\limits_{p=0}^{k-1} \Big(C_{1-\chi,k-1}C^2_{Z,mid,p}+C^2_{V,mid,p}\Big)M^{2p}(1-\delta)^{2\beta-2}\varepsilon^2\\ &+kC^2_{Werror}M^{2k}(1-\delta)^{2\beta-2}\varepsilon^4 \\
&+\sum\limits_{p=0}^{k-1} C^2_{Z,mid,p} M^{2p}(1-\delta)^{2\beta}\varepsilon^2+kC^2_{Zerror}M^{2k}(1-\delta)^{2\beta}\varepsilon^4.
			\end{align*}
			
			 Noting that when all constants $C_{V,mid,i}, C_{Z,mid,i}, i=0,\cdots, s-1$ and $\delta, C_{Z_T,mid,k}$  are fixed, we can take $M$ large enough such that
             \begin{align}\label{cond:M}
             M\geq~ &\max_{k=0,1,\cdots,s-1}\Bigg\{\frac{10(1-\frac{\delta}{2})}{\delta(1-\delta)^3}\cdot  \frac{C_{Z,mid,k}\sum_{i=0}^{k-1} (C^2_{V,mid,i}+C_{\chi,k}C^2_{Z,mid,i})}{C_{Z_T,mid,k}^2}, \nonumber \\
             &\bigg(\frac{10(1-\frac{\delta}{2})}{\delta(1-\delta)^3}\cdot \frac{\sum_{i=0}^{k-1} (C^2_{V,mid,i}+C_{\chi,k}C^2_{Z,mid,i})}{C_{Z_T,mid,k}^2}\bigg)^{\frac{1}{2}} ,\\
             &\bigg(\frac{10}{\delta(1-\delta)^2} \sum\limits_{p=0}^{k-1} (C_{1-\chi,k}C^2_{Z,mid,p}+C^2_{V,mid,p})\bigg)^{\frac{1}{2}},
                           \bigg(\frac{10}{\delta}\sum\limits_{p=0}^{k-1} C^2_{Z,mid,p} \bigg)^{\frac{1}{2}}       \Bigg\}, \nonumber
             \end{align}
             which together with \eqref{cond:coe2} implies that
            \begin{align}\label{cond:Mk}
          & \big\|\partial_t^kZ^{(\beta+1)}(T)\big\|_{L^2}^2  +\big\|\partial_t^{k-1}Z^{(\beta+1)}(T)\big\|_{\mathcal{H}^1}^2 \nonumber\\
          \leq ~ &     (1-\frac{\delta}{2})(1-\delta)^{2\beta}M^{2k}C^2_{Z_T,mid,k}\varepsilon^2+ \frac{4\delta}{10}(1-\delta)^{2\beta}M^{2k}C^2_{Z_T,mid,k}\varepsilon^2 \\
               & +\frac{2+C_{coe,1}\varepsilon}{2-C_{coe,1}\varepsilon}(\tilde{C}_{V,5,k}(1-\delta)^{2\beta-2}  +C^2_{V,error}\varepsilon)(1-\delta)^{2\beta-2} \varepsilon^3 \nonumber \\
               &+kC^2_{Werror}M^{2k}(1-\delta)^{2\beta-2}\varepsilon^4+kC^2_{Zerror}M^{2k}(1-\delta)^{2\beta}\varepsilon^4. \nonumber
            \end{align}
               We now choose $\varepsilon_{4,k}\leq \varepsilon_3$ such that
               \begin{align}\label{cond:vare4}
                \frac{1-\frac{\delta}{2}}{1-\delta}\tilde{C}_{V,5,k}(1-\delta)^{-2}\varepsilon_{4,k}&+k\Big(\frac{1-\frac{\delta}{2}}{1-\delta}C^2_{V,error}+C^2_{Werror}\Big)M^{2k}(1-\delta)^{-2}\varepsilon^2_{4,k} \\
                &+kC^2_{Zerror}M^{2k}\varepsilon_{4,k}^2
                \leq \frac{\delta}{10}M^{2k}C^2_{Z_T,mid,k}.
               \end{align}
				Thus, by setting $\varepsilon_{lem}\leq \varepsilon_{4,k}$, we obtain
				\begin{align}
					\big\|\partial_tZ^{(\beta+1)}(T)\big\|_{H^{k-1}}^2+\big\|Z^{(\beta+1)}(T)\big\|_{H^k}^2\leq C_{Z_T,mid,k}^2(1-\delta)^{2\beta}M^{2k}\varepsilon^2.
			\end{align}
    Since $l\leq s-1$ is finite, taking $\varepsilon_{ZT}=\min\limits_{k=1,\cdots,s-1}\{\varepsilon_{4,k}\}$, we complete the proof of the claim.
\end{proof}
                 			
           Now, similarly to the energy estimate for $\tilde{V}^{(\beta)}(t)$ in \eqref{Vtov}, we can derive the energy estimate for $\tilde{Z}^{(\beta+1)}(t)=\partial_t^k z^{(\beta+1)}$ as follows:
           \begin{equation}\label{ZtoZ}
		    	E_{\beta+1}(\tilde{Z}^{(\beta+1)}(t))
		    	\leq C_{\tilde{Z}to\tilde{Z}}(\varepsilon) \left(E_\beta(\tilde{Z}^{(\beta+1)}(T)) + C_{Z,5,k}(1-\delta)^{2\beta} \varepsilon^3 \right)
		    \end{equation}
 		    where $C_{\tilde{Z}to\tilde{Z}}(\varepsilon)\leq C_{\tilde{Z}to\tilde{Z}}(\varepsilon_{5,k})=\frac{3}{2}<2$  is a constant when choosing $\varepsilon_{lem}\leq \varepsilon_{5,k}$ and
            $C_{Z,5,k}$ is a constant independent of $\beta$ and $\varepsilon$.


            Putting the estimate \eqref{induction-ZH} into \eqref{ZtoZ}, we have
            for any $t\in [0,T]$,
                \begin{align*}
				&\big\|\partial_t^{k+1}Z^{(\beta+1)}(t)\big\|_{L^{2}}^2+\big\|\partial_t^{k}Z^{(\beta+1)}(t)\big\|_{\mathcal{H}^1}^2 \\
                \leq~& \frac{3}{2}\Big(\frac{1-\frac{\delta}{2}}{1-\delta}C^2_{Z_T,mid,k}M^{2k}+(2+C_{coe,1}\varepsilon)C_{Z,5,k}\varepsilon\Big)(1-\delta)^{2\beta}\varepsilon^2.
			\end{align*}
            Taking this back to \eqref{Vtov} in replace of $\beta$ with $\beta+1$, and noting the same estimate of $\tilde{V}^{(\beta+1)}(0)$ with \eqref{comb:3}, we obtain
            \begin{align*}
            &\big\|\partial_t^{k+1}V^{(\beta+1)}(t)\big\|_{L^{2}}^2+\big\|\partial_t^{k}V^{(\beta+1)}(t)\big\|_{\mathcal{H}^1}^2 \\
            \leq& \, \frac{3}{2}\frac{1-\frac{\delta}{2}}{1-\delta}\Big(
            \Big(\sum_{i=1}^{k-1}C^2_{V,mid,i}M^{2i}+ C_{\chi,k}\sum_{i=1}^{k-1}C^2_{Z,mid,i}M^{2i}\Big)(1-\delta)^{2\beta}\varepsilon^2 + kC^2_{V,error}(1-\delta)^{2\beta}M^{2k}\varepsilon^4\Big) \\
            &+\frac{(18T+9TC_{coe,1}\varepsilon)}{4}\bigg(\frac{1-\frac{\delta}{2}}{1-\delta}C^2_{Z_T,mid,k}M^{2k}+(2+C_{coe,1}\varepsilon)C_{Z,5,k}\varepsilon\bigg)(1-\delta)^{2\beta}\varepsilon^2 \\
            &+\frac{(6+3C_{coe,1}\varepsilon)}{2}  C_{V,5,k}(1-\delta)^{2\beta}\varepsilon^3
            \end{align*}
            Recalling \eqref{cond:M} for $M$ and \eqref{cond:vare4} for $\varepsilon_{lem}$, we obtain that
            \begin{align*}
		&\big\|\partial_t^{k+1}V^{(\beta+1)} 
            (t)\big\|_{L^{2}}^2+\big\|\partial_t^{k}V^{(\beta+1)}(t)\big\|_{\mathcal{H}^1}^2 \\
            \leq ~ & \bigg(\frac{\delta}{10}+\frac{(18T+9TC_{coe,1}\varepsilon)}{4}\frac{1-\frac{\delta}{2}}{(1-\delta)^3}\bigg)C_{Z_T,mid,k}^2M^{2k}(1-\delta)^{2\beta+2}\varepsilon^2 \\
            &+ C_{V,6,k}(1-\delta)^{2\beta+2}\varepsilon^3,
            \end{align*}
            where $C_{V,6,k}=\frac{3kC^2_{V,error}}{2}\frac{1-\frac{\delta}{2}}{(1-\delta)^3}\varepsilon+\frac{(18T+9TC_{coe,1}\varepsilon)}{4(1-\delta)^2}(2+C_{coe,1}\varepsilon)C_{Z,5,k}+	\frac{(6+3C_{coe,1}\varepsilon)}{2(1-\delta)^2}C_{V,5,k}$.
            So according to the relation between $C_{Z_T,mid,k}$ and $C_{V,mid,k}$, we know that
            $$
            \frac{3}{2}\frac{1-\frac{\delta}{2}}{(1-\delta)^2}C^2_{Z_T,mid,k}\leq \frac{3}{5}C^2_{Z,mid,k},
            $$
            and
            $$\bigg(\frac{\delta}{10}+\frac{(18T+9TC_{coe,1}\varepsilon)}{4}\frac{1-\frac{\delta}{2}}{(1-\delta)^3}\bigg)C_{Z_T,mid,k}^2\leq \frac{3}{5}C_{V,mid,k}^2.$$
            And choosing $\varepsilon_{5,k}$ such that
            $$(2+C_{coe,1}\varepsilon_{5,k})C_{Z,5,k}\varepsilon_{5,k}\leq \frac{1}{5}M^{2k}(1-\delta)^2,$$
            and
            $$C_{V,6,k}\varepsilon_{5,k}\leq \frac{1}{5}C^2_{V,mid,k}M^{2k},$$
            then we have
            \begin{align}\label{norm:Zk1}
            \big\|\partial_t^{k+1}Z^{(\beta+1)}(t)\big\|_{L^{2}}^2+\big\|\partial_t^{k}Z^{(\beta+1)}(t)\big\|_{\mathcal{H}^1}^2\leq \frac{4}{5}C^2_{Z,mid,k}(1-\delta)^{2\beta+2}\varepsilon^2,
            \end{align}
            and
            \begin{align}\label{norm:Vk1}
				\big\|&\partial_t^{k+1}V^{(\beta+1)}(t)\big\|_{L^{2}}^2+\big\|\partial_t^{k}V^{(\beta+1)}(t)\big\|_{\mathcal{H}^1}^2\leq \frac{4}{5} C^2_{V,mid,k}M^{2k}(1-\delta)^{2\beta+2}\varepsilon^2.
            \end{align}

         Finally, we establish the relationship between time-space norms.
\begin{claim}
            For any $\alpha\geq 1$, $(V^{(\alpha)}, Z^{(\alpha)})$ satisfy for any integer $0\leq k_1\leq k$:
               \begin{equation}
            \label{Vnormequiv}
               \begin{split}
               & \Big|\|\partial_t^{k_1+1} V^{(\alpha)}\|_{\mathcal{H}^{k-k_1}}+\|\partial_t^{k_1} V^{(\alpha)}\|_{\mathcal{H}^{k+1-k_1}}-\|\partial_t^{k+1}V^{(\alpha)}\|_{L^2}-\|\partial_t^{k}V^{(\alpha)}\|_{\mathcal{H}^1}\Big| \\
               \leq ~& \sum\limits_{p=0}^{k-1} (C_{V,mid,p}+C_{\chi,k}C_{Z,mid,p})  M^{p}(1-\delta)^{\alpha}\varepsilon+ C_{V,1,k}M^k(1-\delta)^{\alpha}\varepsilon^2,
               \end{split}\end{equation}
               and
            \begin{equation}
            \label{Znormequiv}
               \begin{split}
               &\Big| \|\partial_t^{k_1+1} Z^{(\alpha)}\|_{\mathcal{H}^{k-k_1}}+\|\partial_t^{k_1} Z^{(\alpha)}\|_{\mathcal{H}^{k+1-k_1}} -\|\partial_t^{k+1}Z^{(\alpha)}\|_{L^2}-\|\partial_t^{k}Z^{(\alpha)}\|_{\mathcal{H}^1} \Big|\\
               \leq~ &  \sum\limits_{p=0}^{k-1} C_{Z,mid,p} M^{p}(1-\delta)^{\alpha-1}\varepsilon+ C_{Z,1,k}M^k(1-\delta)^{\alpha-1}\varepsilon^2,
               \end{split}\end{equation}
               where $C_{V,1,k}, C_{Z,1,k}$ are constants independent of $M$, $\beta$ and $\varepsilon$.
               \end{claim}

             Thanks to the relations in this claim, we indeed prove that there exist $M_k$ and $\varepsilon_{6,k}$ such that when $M\geq M_k$ and $\varepsilon \leq \varepsilon_{6,k}$,
$$\Big|\|\partial_t^{k_1+1} V^{(\alpha)}\|_{\mathcal{H}^{k-k_1}}+\|\partial_t^{k_1} V^{(\alpha)}\|_{\mathcal{H}^{k+1-k_1}}-\|\partial_t^{k+1}V^{(\alpha)}\|_{L^2}-\|\partial_t^{k}V^{(\alpha)}\|_{\mathcal{H}^1}\Big|\leq \frac{1}{5}C_{V,mid,k}M^k(1-\delta)^{\alpha}\varepsilon $$
and
$$
\Big|\|\partial_t^{k_1+1} Z^{(\alpha)}\|_{\mathcal{H}^{k-k_1}}+\|\partial_t^{k_1} Z^{(\alpha)}\|_{\mathcal{H}^{k+1-k_1}}-\|\partial_t^{k+1}Z^{(\alpha)}\|_{L^2}-\|\partial_t^{k}Z^{(\alpha)}\|_{\mathcal{H}^1}\Big|\leq \frac{1}{5}C_{Z,mid,k}M^k(1-\delta)^{\alpha}\varepsilon
$$
Combining these inequalities with \eqref{norm:Vk1} and \eqref{norm:Zk1}, we can obtain the \eqref{induction4}.

            \begin{proof} We only prove \eqref{Znormequiv}, the \eqref{Vnormequiv} actually is the same.
Rewriting the equation for $Z^{(\alpha)}$, we have
\begin{align}\label{eqn:Z:alpha}
				\partial_t^2Z^{(\alpha)}=\Delta Z^{(\alpha)}-\partial_tZ^{(\alpha)}+Z^{(\alpha)}_{error},
			\end{align}
			{\color{black}where
				\begin{align*}
			Z^{(\alpha)}_{error}:=\big(1-b_0^{(\alpha)}\big) 
                Z_t^{(\alpha)}-\tilde{b}^{(\alpha)}Z^{(\alpha)}-\sum\limits_{i=1}^nb_i^{(\alpha)}V_{x_i}^{(\alpha)}+H^{(\alpha)}+\sum\limits_{i,j=1}^n\Big(\big( a_{ij}^{(\alpha)}-\delta_{ij}\big) Z_{x_i}^{(\alpha)}\Big)_{x_j}.
				\end{align*}
				Since \eqref{induction-ZH} are assumed to hold for any $\alpha\geq 1$ and $l\leq k-1$, as well as for any $\alpha\leq \beta,$ and $l\leq s-1$,  and recalling the estimates on coefficients, applying Lemma \ref{lem:GG} with $G=V_{error}^{(\alpha)}$, we obtain that for any $p\leq k-1$ and $m\leq k-1-p$,
				\begin{align}
					\|\partial_t^pZ^{(\alpha)}_{error}(T)\|_{\mathcal{H}^m}\leq C_{Zerror}M^{p+m+2 }(1-\delta)^\alpha\varepsilon^2,
			\end{align}}
           for some constant $C_{Zerror}$ independent of $M$, $\alpha$ and $\varepsilon$.		

     For any integer $0\leq k_1\leq k$, differentiating \eqref{eqn:Z:alpha} with respect to $t$ for $k-2$ times, we obtain the following results:

              {\color{black}When $k-k_1\geq 2$ is odd, we have
				\begin{align}\label{z:k:odd}					\partial_t^kZ^{(\alpha)}=\partial_t^{k_1+1}\Delta^{\frac{k-k_1-1}{2}}Z^{(\alpha)}+\sum\limits_{p=0}^{\frac{k-k_1-1}{2}}\partial_t^{k-2p-2}\Delta^p\Big(\partial_tZ^{(\alpha)}+Z_{error}^{(\alpha)}\Big),
			\end{align}
            when $k-k_1\geq 2$ is even, we have
				\begin{align}\label{z:k:even}
                \partial_t^kZ^{(\alpha)}=\partial_t^{k_1}
                \Delta^{\frac{k-k_1}{2}}Z^{(\alpha)}+
                \sum\limits_{p=0}^{\frac{k-k_1}{2}-1}\partial_t^{k-2p-2}\Delta^{p}\Big(\partial_tZ^{(\alpha)} +Z_{error}^{(\alpha)}\Big).
				\end{align} }

              Thus, we derive the following estimates:
              \begin{equation}
              \|\partial_t^kZ^{(\alpha)}\|_{L^2}\leq \|Z^{(\alpha)}\|_{\mathcal{H}^k}+ \sum\limits_{p=0}^{k-1} \|\partial_t^{k-1-p}Z^{(\alpha)}\|_{\mathcal{H}^{p}}+k\sum_{p=0}^{k-2}\|\partial_t^pZ_{error}^{(\alpha)}\|_{\mathcal{H}^{k-p-2}},
              \end{equation}
              and
              \begin{equation}
             \|Z^{(\alpha)}\|_{\mathcal{H}^k}\leq  \|\partial_t^kZ^{(\alpha)}\|_{L^2} + \sum\limits_{p=0}^{k-1} \|\partial_t^{k-1-p}Z^{(\alpha)}\|_{\mathcal{H}^{p}}+k\sum_{p=0}^{k-2}\|\partial_t^pZ_{error}^{(\alpha)}\|_{\mathcal{H}^{k-p-2}}.
              \end{equation}

               Since we assume that \eqref{induction4} holds for any $\alpha$ and $l\leq k-1$, we obtain
               \begin{equation}
               \|\partial_t^{k+1}Z^{(\alpha)}\|_{L^2}\leq \|Z^{(\alpha)}\|_{\mathcal{H}^{k+1}}+ \sum\limits_{p=0}^{k-1} C_{Z,mid,p} M^{p}(1-\delta)^{\alpha}\varepsilon+kC_{Zerror}M^k(1-\delta)^{\alpha}\varepsilon^2,
               \end{equation}
               and
               \begin{equation}
               \|Z^{(\alpha)}\|_{\mathcal{H}^{k+1}}\leq \|\partial_t^{k+1}Z^{(\alpha)}\|_{L^2} + \sum\limits_{p=0}^{k-1} C_{Z,mid,p} M^{p}(1-\delta)^{\alpha}\varepsilon+kC_{Zerror}M^k(1-\delta)^{\alpha}\varepsilon^2.
               \end{equation}
               This completes the proof.
               \end{proof}
%

			\subsubsection{ Inductive step 2: To prove \eqref{induction5} for $\alpha=\beta+1\geq 2$.}

			When $k\leq s-1$, \eqref{induction5} follows from \eqref{induction4} and $(v^{(0)},z^{(0)})=(0,0)$ directly. In fact, we observe that for any $k\leq s-1$,
            \begin{align*}
            \big\|\partial_t^m v^{(\alpha)}\big\|_{H^{k-m}}^2=~&\big\|\sum_{\beta=1}^\alpha \partial_t^mV^{(\beta)}\big\|_{H^{k-m}}^2\leq \sum_{\beta=1}^\alpha \|\partial_t^mV^{(\beta)}\|_{H^{k-m}}^2 \\
            \leq~& \frac{(1-\delta)^2-(1-\delta)^{2\alpha}}{1-(1-\delta)^2}C^2_{V,mid,k}M^{2k}\varepsilon^2.
            \end{align*}
		Noting that  $\frac{(1-\delta)^2-(1-\delta)^{2\alpha}}{1- 
            (1-\delta)^2}\leq C_\delta$, thus we can obtain \eqref{induction3} for the case that $k\leq s-1$.

			It remains to estimate $E_\alpha(\partial_t^{s-1}v^{(\alpha)})$ and $E_\alpha(\partial_t^{s-1}z^{(\alpha)})$, i.e., the case of $k=s$. Differentiating the equations \eqref{eqnofvalpha} and \eqref{eqnofzalpha} by $t$ for $s-1$ times, we obtain that
			\begin{equation}\label{eq:vs}
			\begin{aligned}
				& \partial_t^{s+1}v^{(\alpha)}+b^{(\alpha)}_0\partial_t^sv^{(\alpha)}-\sum\limits_{i,j=1}^n\big(a^{(\alpha)}_{ij}\partial_t^{s-1}v^{(\alpha)}_{x_i}\big)_{x_j}+\tilde{b}^{(\alpha)} \partial_t^{s-1}v^{(\alpha)}+\sum\limits_{i=1}^nb^{(\alpha)}_i\partial_t^{s-1}v^{(\alpha)}_{x_i} \\
				+ & ~ 2\chi\cdot\partial_t^{s-1}z^{(\alpha)}_t=g^{(\alpha)},
			\end{aligned}
			\end{equation}
			
			and
\begin{equation}\label{eq:zs}			\partial_t^{s+1}z^{(\alpha)}-b^{(\alpha)}_0\partial_t^sz^{(\alpha)}-\sum\limits_{i,j=1}^n\big(a^{(\alpha)}_{ij}\partial_t^{s-1}z^{(\alpha)}_{x_i}\big)_{x_j}+\tilde{b}^{(\alpha)} \partial_t^{s-1}z^{(\alpha)}+\sum\limits_{i=1}^nb^{(\alpha)}_i\partial_t^{s-1}z^{(\alpha)}_{x_i}=h^{(\alpha)},\end{equation}
			
			where
			$$
			\begin{aligned}
				g^{(\alpha)}= & ~ b^{(\alpha)}_0\partial_t^sv^{(\alpha)}-\sum\limits_{i,j=1}^n\big(a^{(\alpha)}_{ij}\partial_t^{s-1}v^{(\alpha)}_{x_i}\big)_{x_j}+\tilde{b}^{(\alpha)} \partial_t^{s-1}v^{(\alpha)}+\sum\limits_{i=1}^nb_i^{(\alpha)}\partial_t^{s-1}v^{(\alpha)}_{x_i} \\
				- & ~ \partial_t^{s-1}\Big(b_0^{(\alpha)}v^{(\alpha)}_t-\sum\limits_{i,j=1}^n\big(a_{ij}^{(\alpha)}v^{(\alpha)}_{x_i}\big)_{x_j}+\tilde{b}^{(\alpha)}v^{(\alpha)}+\sum\limits_{i=1}^n b_i^{(\alpha)}v^{(\alpha)}_{x_i}\Big),
			\end{aligned}
			$$
			
			and
			$$
			\begin{aligned}
				h^{(\alpha)}= & ~ -b_0^{(\alpha)}\partial_t^sz^{(\alpha)}-\sum\limits_{i,j=1}^n\big(a_{ij}^{(\alpha)}\partial_t^{s-1}z^{(\alpha)}_{x_i}\big)_{x_j}+\tilde{b}^{(\alpha)} \partial_t^{s-1}z^{(\alpha)}+\sum\limits_{i=1}^nb_i^{(\alpha)}\partial_t^{s-1}z^{(\alpha)}_{x_i} \\
				- & ~ \partial_t^{s-1}\Big(-b_0^{(\alpha)}z^{(\alpha)}_t-\sum\limits_{i,j=1}^n\big(a_{ij}^{(\alpha)}z^{(\alpha)}_{x_i}\big)_{x_j}+\tilde{b}^{(\alpha)}z^{(\alpha)}+ \sum\limits_{i=1}^nb_i^{(\alpha)}z^{(\alpha)}_{x_i}\Big).
			\end{aligned}
			$$
            The proof of this case is quite similar to the proof of \eqref{induction3}. We only list the key steps here.
\begin{enumerate}
  \item \textbf{The equations for $\partial_t^{s-1}v^{(\alpha)}$ and $\partial_t^{s-1}z^{(\alpha)}$.} 
  
  Owing to our small assumptions on the coefficients, we rewrite the equations for $\partial_t^{s-1}v^{(\alpha)}$ and $\partial_t^{s-1}z^{(\alpha)}$ as follows:

      {\color{black}When $s$ is odd,
				\begin{align}\label{v:s:odd}
                \partial_t^{s+1}v^{(\alpha)}=\partial_t\Delta^{\frac{s}{2}}v^{(\alpha)}-\sum\limits_{l=0}^{\frac{s}{2}}\partial_t^{s-2l-2}\Delta^l\Big(2\chi Z^{(\alpha)}_t+\partial_tv^{(\alpha)}-v_{error}^{(\alpha)}\Big),
			\end{align}
and
	\begin{align}\label{z:s:odd}
        \partial_t^{s+1}z^{(\alpha)}=\partial_t\Delta^{\frac{s}{2}}z^{(\alpha)}-\sum\limits_{l=0}^{\frac{s}{2}}\partial_t^{s-2l-2}\Delta^l\Big(2\chi z^{(\alpha)}_t+\partial_tv^{(\alpha)}-v_{error}^{(\alpha)}\Big),
	\end{align}
when $s$ is even,
				\begin{align}\label{v:s:even}
					\partial_t^{s+1}v^{(\alpha)}=\Delta^{\frac{s+1}{2}}v^{(\alpha)}-\sum\limits_{l=0}^{\frac{s}{2}}\partial_t^{s-2l-2}\Delta^l\Big(2\chi z^{(\alpha)}_t+\partial_tv^{(\alpha)} -v_{error}^{(\alpha)}\Big),
				\end{align}
      and
				\begin{align}\label{z:s:even}
					\partial_t^{s+1}v^{(\alpha)}=\Delta^{\frac{s+1}{2}}v^{(\alpha)}-\sum\limits_{l=0}^{\frac{s}{2}}\partial_t^{s-2l-2}\Delta^l\Big(2\chi z^{(\alpha)}_t+\partial_tv^{(\alpha)} -v_{error}^{(\alpha)}\Big).
				\end{align} }

      Consequently, we can establish the following relationships between $\partial_t^kv^{(\alpha)}$ and $\Delta^{\frac{k}{2}} v^{(\alpha)} $ and  between $\partial_t^kz^{(\alpha)}$ and $\Delta^{\frac{k}{2}} z^{(\alpha)} $. Specifically, for any  $t\in [0,T]$,
      \begin{equation}
      \Big|\|\partial_t^kv^{(\alpha)}\|_{L^2}-\|v^{(\alpha)} \|_{\mathcal{H}^k}\Big|+\Big|\|\partial_t^kz^{(\alpha)}\|_{L^2}-\|z^{(\alpha)}\|_{\mathcal{H}^k}\Big|=O(M^{k}\varepsilon^2)+O(M^{k-1}\varepsilon).
      \end{equation}
  Recalling the definition $w^{(\alpha)}=v^{(\alpha)}+z^{(\alpha)}$, and combining it with the above relationships, we obtain for any $t\in [0,T]$,
   $\big|\|\partial_t^kw^{(\alpha)}\|_{L^2}-\|w^{(\alpha)}\|_{\mathcal{H}^k}\big|=O(M^{k-1}\varepsilon)$.
  \item \textbf{The non-increasing of $E_{\alpha}(z^{(\alpha)})(T)$}. 
  
  We directly compute the difference $E_{\alpha}(z^{(\alpha)})(T)-E_{\alpha-1}(z^{(\alpha-1)})(T)$ and find that
    \begin{equation}
        E_{\alpha}(\partial_t^{s-1}z^{(\alpha)})(T)-E_{\alpha-1}(\partial_t^{s-1}z^{(\alpha-1)})(T)=E_{\alpha}(w^{(\alpha-1)})(T)-E_{\alpha-1}(z^{(\alpha-1)})(T).
    \end{equation}
    Combining this with the aforementioned relationships, we deduce that the above expression is equal to
    \begin{equation}
    \begin{aligned}
         E_{\alpha}(\partial_t^{s-1}z^{(\alpha-1)})(T)-E_{\alpha-1}(\partial_t^{s-1}z^{(\alpha-1)})(T)
        +E_{\alpha}(\partial_t^{s-1}v^{(\alpha-1)})(T)\\
        +\int_\Omega \partial_t^sz^{(\alpha-1)}\partial_t^sv^{(\alpha-1)}\dd x+\sum_{i,j=1}^n\int_\Omega a^{(\alpha)}_{ij}\partial_{x_i}\partial_t^{s-1}z^{(\alpha-1)}\partial_{x_i}\partial_t^{s-1}v^{(\alpha-1)}\dd x.
    \end{aligned}
    \end{equation}
   Standard energy estimates yield
   \begin{equation}
   E_{\alpha}(\partial_t^{s-1}z^{(\alpha-1)})(T)-E_{\alpha-1}(\partial_t^{s-1}z^{(\alpha-1)})(T)=O(M^{2s}\varepsilon^3),
   \end{equation}
   \begin{equation}
   E_{\alpha}(\partial_t^{s-1}v^{(\alpha-1)})(T)\leq \frac{3}{2}\int_0^T\big\|\chi \partial_t^sz^{(\alpha)}\big\|_{L^2}^2\dd t+O(M^{2s-1}\varepsilon^2),
   \end{equation}
   and
  \begin{equation}\begin{aligned}
  \int_\Omega \partial_t^sz^{(\alpha-1)}\partial_t^sv^{(\alpha-1)}+\sum_{i,j=1}^n\int_\Omega a^{(\alpha)}_{ij}\partial_{x_i}\partial_t^{s-1}z^{(\alpha-1)}\partial_{x_i}\partial_t^{s-1}v^{(\alpha-1)} \\=-2\int_0^T\big\|\chi \partial_t^sz^{(\alpha)}\big\|_{L^2}^2\dd t+O(M^{2s-1}\varepsilon^2).
  \end{aligned}
  \end{equation}
Compared with the coefficients of the equation \eqref{eq:zs} for $\partial_t^{s-1}z^{(\alpha)}$  with those of the equation \eqref{eqnofZalpha} for $Z^{(\alpha)}$, we see that they are identical, with the only difference being in the force term. Consequently, we can apply Theorem \ref{fully nonlinear observability} to the system of   $\partial_t^{s-1}z^{(\alpha)}$ yielding
\begin{equation}			
E\big(\partial_t^{s-1}z^{(\alpha)}(T)\big)\leq D\int_0^T\int_\omega\big|\partial_t^sz^{(\alpha)}\big|^2\dd x\dd t+O(M^{2s}\varepsilon^4).
\end{equation}
Combining these results, we arrive at
\begin{align}
&E_{\alpha}(\partial_t^{s-1}z^{(\alpha)})(T)-E_{\alpha-1}(\partial_t^{s-1}z^{(\alpha-1)})(T) \nonumber\\
\leq~& -\frac{1}{2D}E_{\alpha-1}(z^{(\alpha-1)})(T)+O(M^{2s}\varepsilon^3)+O(M^{2s-1}\varepsilon^2).
\end{align}
Given the assumption $E_{\alpha-1}(z^{(\alpha-1)})(T)=O(M^{2s}\varepsilon^2)$, by choosing $M$ sufficiently large and  $\varepsilon$ sufficiently small, we can ensure that $E_{\alpha}(\partial_t^{s-1}z^{(\alpha)})(T)$ is non-increasing with respect to $\alpha$.

\item \textbf{Uniform Boundedness for $E\big(\partial_t^{s-1}z^{(\alpha)}(t)\big)$ and $E\big(\partial_t^{s-1}v^{(\alpha)}(t)\big)$.} Returning to the equation for  $\partial_t^{s-1}z^{(\alpha)}$, and invoking the well-posedness, for some fixed  $T$,  we have that $E_{\alpha}(\partial_t^{s-1}z^{(\alpha)})(t)$ is uniformly bounded. Consequently, by \eqref{eq:vs},  $E_{\alpha}(\partial_t^{s-1}v^{(\alpha)})(t)$ is also uniformly bounded.
\item \textbf{Completion of the Proof.} Utilizing the result from the first step, we establish the uniform boundedness of the time-space norms of  $\partial_t^{s-1}z^{(\alpha)}(t)$and $\partial_t^{s-1}v^{(\alpha)}(t)$. This completes the proof of \eqref{induction5}.
\end{enumerate}
\end{proof}

        \subsection{Uniqueness}
        In this section, we aim to demonstrate that $y$ is the unique solution to System \eqref{system:quasilinear} that satisfies \eqref{cond:y} with some constant $C_{uni}>0$. To proceed, we assume the existence of another solution $\tilde{y}\in \cap_{i=0}^2 C^i(0,T;\mathcal{H}^s, s\geq \{n+2,4\}$ that also satisfies \eqref{system:quasilinear} and the bound \eqref{cond:y}.
        For notational simplicity, we define:
        \begin{equation*}
        \tilde{b}^v=\tilde{b}(t,x,v,v_t,\nabla v), ~ b_k^v=v_k(t,x,v,v_t,\nabla v), ~ a_{ij}^v=a_{ij}(t,x,v,v_t,\nabla v),
        \end{equation*}
        for any function $v$ and for any $k=0,\cdots,n$ and  $i,j=1,\cdots,n$.
        
        Let
        $w=y-\tilde{y}$. Then, we have
       
        \begin{equation}\label{equ:ww}
			\left\{
			\begin{aligned}
				& \partial_t^2w-\Delta w+w_t=w_{error}, & (t,x)\in & ~ (0,T)\times\Omega, \\
				& w(t,x)=0, & (t,x)\in & ~ (0,T)\times\partial\Omega, \\
				& w(0,x)=0, ~ w_t(0,x)=0, & x\in & ~ \Omega.
			\end{aligned}
			\right.
		\end{equation}
        where
        \begin{align*}
        w_{error}=~& \left(\tilde{b}^yy-\tilde{b}^{\tilde{y}}\tilde{y}\right)+\sum_{i,j=1}^n \left(((a_{ij}^y-\delta_{ij})y_{x_i})_{x_j}-((a_{ij}^{\tilde{y}}-\delta_{ij})\tilde{y}_{x_i})_{x_j}\right)\\
        &+\left((1-b_{0}^y)y_t-(1-b_{0}^{\tilde{y}})\tilde{y}_t\right) +\sum_{i=1}^n\left(b_{i}^yy_{x_i}-b_{i}^{\tilde{y}}\tilde{y}_{x_i}\right) \\
        =:~&I_1+I_2+I_3+I_4.
        \end{align*}
        Next, we expand each  $I_i$ for $i=1,2,3,4$. Stating with $I_1$:
        \begin{equation}
        I_1=\tilde{b}^yy-\tilde{b}^{\tilde{y}}\tilde{y}=(\tilde{b}^y-\tilde{b}^{\tilde{y}})y+\tilde{b}^{\tilde{y}}(y-\tilde{y}).
        \end{equation}
        The first term on the right-hand side can be written as:
        \begin{align*}
        &\tilde{b}(t,x,y,y_t,y_{x_1},\cdots,y_{x_n})-\tilde{b}(t,x,\tilde{y},\tilde{y}_t,\tilde{y}_{x_1},\cdots,\tilde{y}_{x_n}) \\
        =~&\frac{\tilde{b}(y,y_t,y_{x_1},\cdots,y_{x_n})-\tilde{b}(\tilde{y},y_t,y_{x_1},\cdots,y_{x_n})}{y-\tilde{y}}(y-\tilde{y})  \\
        &+\frac{\tilde{b}(\tilde{y},y_t,y_{x_1},\cdots,y_{x_n})-\tilde{b}(t,x,\tilde{y},\tilde{y}_t,y_{x_1},\cdots,y_{x_n})}{y_t-\tilde{y}_t}(y_t-\tilde{y}_t) \\
        &+\frac{\tilde{b}(t,x,\tilde{y},\tilde{y}_t,y_{x_1},\cdots,y_{x_n})-\tilde{b}(t,x,\tilde{y},\tilde{y}_t,\tilde{y}_{x_1},\cdots,y_{x_n})}{y_{x_1}-\tilde{y}_{x_1}}(y_{x_1}-\tilde{y}_{x_1})\\ 
        & +\cdots \\
        &+\frac{\tilde{b}(t,x,\tilde{y},\tilde{y}_t,\tilde{y}_{x_1},\cdots,\tilde{y}_{x_{n-1}},y_{x_n})-\tilde{b}(t,x,\tilde{y},\tilde{y}_t,\tilde{y}_{x_1},\cdots,\tilde{y}_{x_{n-1}},\tilde{y}_{x_n})}{y_{x_n}-\tilde{y}_{x_n}}(y_{x_n}-\tilde{y}_{x_n})\\
        =~& \tilde{b}_yw+\tilde{b}_{y_t}w_t+\sum_{i=1}^n \tilde{b}_{y_i}w_{x_i}.
        \end{align*}
        Therefore, $I_1$ is equal to
        \begin{equation}\label{I11}
        I_1=(y\tilde{b}_y+\tilde{b}^{\tilde{y}})w+y\tilde{b}_{y_t}w_t +y\sum_{k=1}^n\tilde{b}_{y_{x_k}}w_{x_k},
        \end{equation}
        where $\tilde{b}_y,\tilde{b}_{y_t}, \tilde{b}_{y_{x_k}}$ are bounded functions for any $k=1,\cdots,n$.

         Similarly, we can expand $I_2,I_3$ and $I_4$:
         \begin{equation}\label{I12}
         \begin{split}
         I_2&=\sum_{i,j=1}^n \left(\left(y_{x_i}a_{ij,y}w+y_{x_i}a_{ij,y_t}w_t+\sum_{k=1}^ny_{x_i}a_{ij,y_{x_k}}w_{x_k}\right)_{x_j}+((a_{ij}^{\tilde{y}}-\delta_{ij})w_{x_i})_{x_j}\right),\\
         I_3&=y_tb_{0,y}w+\left(y_tb_{0,y_t}+(1-b_{0}^{\tilde{y}})\right)w_t+\sum_{k=1}^ny_tb_{0,y_{x_k}}w_{x_k},\\
         I_4&=  \sum_{i=1}^n\left(y_{x_i}\left(b_{i,y}w+b_{i,y_t}w_t+\sum_{k=1}^nb_{i,y_{x_k}}w_{x_k}\right)+b_{i}^{\tilde{y}}w_{x_i}\right).
         \end{split}
         \end{equation}
         Here $a_{ij,y},a_{ij,y_t},a_{ij,y_{x_k}}, b_{0,y},b_{0,y_t},b_{0,y_{x_k}}, b_{i,y},b_{i,y_t},b_{i,y_{x_k}},$ are bounded functions for any $i,j,k=1,\cdots,n $.
         
         Now, we proceed to prove that $w$ must be zero. We define the energy of the system \eqref{equ:ww} as follows:
         $$E(t)=\frac{1}{2}\int_\Omega \left(|w_t|^2+|\nabla w|^2\right) \dd x.$$
         
         Next, we multiply \eqref{equ:ww} by $w_t$ and integrate over $\Omega$, yielding:
         $$
         E(t)+\int_\Omega |w_t|^2\dd x=\int_\Omega w_{error}w_t\dd x.
         $$
         We need to estimate  $\int_\Omega w_{error}w_t\dd x$,  specifically $\int_\Omega I_iw_t\dd x$ for $ i=1,2,3,4$. 
         Using the expansions \eqref{I11} and \eqref{I12}, along with the Cauchy-Schwarz inequality, we obtain:
         $$
         \Big|\int_\Omega (I_1+I_3+I_4)w_t\dd x\Big|\leq C_1E(t),
         $$
         for some constant $C_1>0$. 
         
         Finally, we need to handle $\int_\Omega I_2 w_t \dd x$, The first two terms of $I_2w_t$  are similar to the previous cases and can be controlled by $E(t)$, that is,
         $$
         \bigg|\int_\Omega\Big(\sum_{i,j=1}^n\left(y_{x_i}a_{ij,y} w+y_{x_i}a_{ij,y_t}w_t\right)_{x_j}\Big)w_t\dd x\bigg|\leq C_2E(t).
         $$
         For the remaining two terms, we can use integration by parts to obtain:
         \begin{align*}
         &\int_\Omega \sum_{i,j=1}^n \bigg(w_t\Big(\sum_{k=1}^ny_{x_i}a_{ij,y_{x_k}}w_{x_k}\Big)_{x_j}+w_t\Big(\big(a_{ij}^{\tilde{y}}-\delta_{ij}\big)w_{x_i}\Big)_{x_j}\bigg)\dd x         \\
         =~& -\bigg(\sum_{i=1}^n\sum_{k,j=1}^n \frac{y_{x_i}a_{ij,y_{x_k}}+y_{x_i}a_{ik,y_{x_j}}}{2}w_{x_i}w_{x_j}\bigg)_t-\bigg(\sum_{i,j=1}^n \frac{(a_{ij}^{\tilde{y}}-\delta_{ij})}{2}w_{x_i}w_{x_j}\bigg)_t \\
         &+ \Big(\sum_{i=1}^n\sum_{k,j=1}^n \frac{y_{x_i}a_{ij,y_{x_k}}+y_{x_i}a_{ik,y_{x_j}}}{2}\Big)_tw_{x_i}w_{x_j}+\frac{1}{2}\sum_{i,j=1}^n (a_{ij}^{\tilde{y}})_tw_{x_i}w_{x_j}.
         \end{align*}
         Given our assumption that the solution satisfies the estimate \eqref{cond:y}, and considering the boundedness of the coefficient functions $a_{ij,y_{x_k}}$, we can conclude that the above terms are bounded by
         $$C_3\varepsilon \frac{\dd E(t)}{\dd t}+C_4E(t),$$
         for some constants $C_3,C_4>0$. Therefore, we have the following inequality:
         $$
         (1-C_3\varepsilon)\frac{\dd E}{\dd t}\leq(C_1+C_2+C_4)E(t)
         $$
         When $\varepsilon$ satisfies $0<1-C_3\varepsilon<1$,  applying Gronwall's inequality and noting that $E(0)=0$,  we conclude that $E(t)\equiv0$ for all $t\geq 0$.
         This completes the proof of the uniqueness of the solution. 
%
%
%

		\section{Proof of Theorem \ref{fully nonlinear controllability}}\label{sec:5}
		
		We consider the local null controllability problem for the fully nonlinear damped wave equations:
		\begin{equation}\label{dampedfullynonlinearwaveeqn}
			\left\{
			\begin{aligned}
				& y_{tt}+2y_t-\Delta y+y=F(y,y_t,\nabla y,\nabla^2y)+\chi\cdot u, & (t,x)\in & ~ (0,T)\times\Omega, \\
				& y(t,x)=0, & (t,x)\in & ~ (0,T)\times\partial\Omega, \\
				& y(0,x)=y^0, ~ y_t(0,x)=y^1, & x\in & ~ \Omega,
			\end{aligned}
			\right.
		\end{equation}
		where $\chi\in C^\infty(\Omega)$ satisfies $0\leq\chi(x)\leq 1$, $\chi|_{\omega}\equiv 1$, and $\chi$ supports in a neighbourhood of $\omega$, with $\omega\subset\Omega\cap O_{\varepsilon_0}(\partial\Omega)$, and
		\begin{equation}
			F(\lambda)=O\big(|\lambda|^2\big), \quad (\lambda\to 0),
		\end{equation}
		with
		$$\lambda=\Big(\lambda',\lambda_0,\lambda_i(i=1,\cdots,n),\lambda_{ij}(i,j=1,\cdots,n)\Big).$$
		
%
Before proceeding with the proof, we offer several remarks here.
		\begin{remark}
			Our nonlinear term $F$ is independent of $\nabla y_t$, this is mainly for simplicity, otherwise we need to deal with terms $v_{tx_i}$, and terms $z_{tx_i}$ in the dual system. But under the assumption of $(T,\omega)$ and $\varepsilon\ll 1$, the observability inequality might also be right.
		\end{remark}
		
		\begin{remark}
			Recall that in the second section, we let $y=e^t\tilde{y}$ to reduce a classical linear wave equation to a damped one. Similarly, for a classical nonlinear wave equation, we can use the same method to reduce it to \eqref{dampedfullynonlinearwaveeqn}.
		\end{remark}
				
			We draw attention to the fact that the outcome articulated in Theorem \ref{fully nonlinear controllability} is characterized by the conditions: $y_t(T)=0, ~ y_{tt}(T)=0,$ as opposed to the conditions: $y(T)=0, ~ y_t(T)=0.$ This discrepancy arises from the inherent complexity associated with fully nonlinear equations, which precludes a direct solution approach. To circumvent this, it is necessary to apply differentiation with respect to $t$ to the equation, thereby converting it into a quasi-linear form. Consequently, the objective of our control strategy is shifted to target the state variables $ (y_t, y_{tt}) $ at the terminal time $ T $, rather than $ (y, y_t) $. Pursuing control over $ (y, y_t) $ may introduce additional layers of complexity. This represents a novel insight that has emerged from our examination of fully nonlinear equations, underscoring the distinctive challenges they present in comparison to their linear counterparts.

		\begin{proof} {\color{black}We first consider \eqref{dampedfullynonlinearwaveeqn} intuitively. Let $v=y_t$, we have
				\begin{equation}\label{eqnofyandv}
					-\Delta y+y=F(y,v,\nabla y,\nabla^2y)-v_t-2v+\chi\cdot u
				\end{equation}
				differentiate \eqref{eqnofyandv} by $t$ formally, we get
				\begin{equation}\label{eqn of v}
					v_{tt}+b_0v_t-\sum\limits_{i,j=1}^na_{ij}v_{x_ix_j}=\tilde{b}v+\sum\limits_{i=1}^nb_iv_{x_i}+\chi\cdot u_t,
				\end{equation}
				where
				\begin{align}
			a_{ij}= & ~ \delta_{ij}+\frac{\partial F}{\partial 
                y_{x_ix_j}}(y,v,\nabla y,\nabla^2y), & b_0= & ~ 2-\frac{\partial F}{\partial v}(y,v,\nabla y,\nabla^2y), \nonumber \\
                b_i= & ~ \frac{\partial F}{\partial y_{x_i}}(y,v,\nabla y,\nabla^2y), & \tilde{b}= & ~ \frac{\partial F}{\partial y}(y,v,\nabla y,\nabla^2y)-1.
				\end{align}
				
				Inspired by \eqref{eqnofyandv} and \eqref{eqn of v}, we set up the following iteration schemes: taking $$(y^{(0)},z^{(0)},v^{(0)})\equiv 0,$$ knowing $(y^{(\alpha-1)},z^{(\alpha-1)}, v^{(\alpha-1)})$, we define $(y^{(\alpha)},z^{(\alpha)},v^{(\alpha)})$ as follows
				\begin{align}\label{eqnofyalpha:1.3}
					-\Delta y^{(\alpha)}+y^{(\alpha)}= & ~ F\big(y^{(\alpha-1)},v^{(\alpha-1)},\nabla y^{(\alpha-1)},\nabla^2y^{(\alpha-1)}\big) \\
					& ~ -v^{(\alpha-1)}_t-2v^{(\alpha-1)}-2\chi\left(z^{(\alpha-1)}(t)-z^{(\alpha-1)}(0)\right), \nonumber
				\end{align}
				\begin{align}\label{eqnofvalpha:1.3}
					v^{(\alpha)}_{tt}+b^{(\alpha)}_0v^{(\alpha)}_t-\sum\limits_{i,j=1}^na^{(\alpha)}_{ij}v^{(\alpha)}_{x_ix_j}=\sum\limits_{i=1}^nb^{(\alpha)}_iv^{(\alpha)}_{x_i}+\tilde{b}^{(\alpha)} v^{(\alpha)}-2\chi\cdot z^{(\alpha)}_t,
				\end{align}
				and
				\begin{align}\label{eqnofzalpha:1.3}
					z^{(\alpha)}_{tt}-b^{(\alpha)}_0z^{(\alpha)}_t-\sum\limits_{i,j=1}^na^{(\alpha)}_{ij}z^{(\alpha)}_{x_ix_j}=\sum\limits_{i=1}^nb_i^{(\alpha)}z^{(\alpha)}_{x_i}+\tilde{b}^{(\alpha)} z^{(\alpha)},
				\end{align}
				with boundary value
				\begin{align}\label{boundary value:1.3}
					y^{(\alpha)}(t,x)=0, ~ v^{(\alpha)}(t,x)=0, ~ z^{(\alpha)}(t,x)=0, \quad (t,x)\in(0,T)\times\partial\Omega,
				\end{align}
				and initial (or final) data
				\begin{align}\label{data:1.3}
					& ~ v^{(\alpha)}(0,x)=y^1, ~ v^{(\alpha)}_t(0,x)=\Delta y^0-y^0-2y^1+F(y^0,y^1,\nabla y^0,\nabla^2y^0), \nonumber \\
					& ~ z^{(\alpha)}(T,x)=v^{(\alpha-1)}(T,x)+z^{(\alpha-1)}(T,x), \\
					& ~ z^{(\alpha)}_t(T,x)=v_t^{(\alpha-1)}(T,x)+z_t^{(\alpha-1)}(T,x), \nonumber
				\end{align}
				where
				\begin{equation}
					\begin{aligned}
						& a_{ij}^{(\alpha)}=a_{ij}(y^{(\alpha-1)},v^{(\alpha-1)},\nabla y^{(\alpha-1)},\nabla^2y^{(\alpha-1)}), \quad i,j=1,\cdots,n\\
						& b^{(\alpha)}_i=b_i(y^{(\alpha-1)},v^{(\alpha-1)},\nabla y^{(\alpha-1)},\nabla^2y^{(\alpha-1)}), \quad i=0,\cdots,n \\
						& \tilde{b}^{(\alpha)}=\tilde{b}^{(\alpha)}(y^{(\alpha-1)},v^{(\alpha-1)},\nabla y^{(\alpha-1)},\nabla^2y^{(\alpha-1)}).
					\end{aligned}
				\end{equation}
				
				By Picard iteration method, we can prove that
				\begin{align}\label{converge:1.3}
					\big(v^{(\alpha)}, ~ v^{(\alpha)}_t\big) \to (v, ~ v_t) \;\; & \text{in} \;\; L^\infty\big(0,T;\mathcal{H}^{s-2}\big)\times L^\infty\big(0,T;\mathcal{H}^{s-3}\big), \nonumber \\
					\big(z^{(\alpha)}, ~ z^{(\alpha)}_t\big) \to (z, ~ z_t) \;\; & \text{in} \;\; L^\infty\big(0,T;\mathcal{H}^{s-2}\big)\times L^\infty\big(0,T;\mathcal{H}^{s-3}\big), \\
					y^{(\alpha)} \to y \;\; & \text{in} \;\; L^\infty\big(0,T;\mathcal{H}^{s-1}\big), \nonumber \\
					y^{(\alpha)}_t \to y_t \;\; & \text{in} \;\; L^\infty\big(0,T;\mathcal{H}^{s-2}\big), \nonumber
				\end{align}
				as $\alpha\to\infty$. The proof is similar to (but more complicated than) that of Theorem \ref{thm:main1}.}
				
				The next step is to prove $v=y_t$. By \eqref{eqnofvalpha:1.3} and \eqref{converge:1.3}, we can check that $v$ satisfies
				\begin{equation}\label{eqn of v finally}
					\left\{
					\begin{aligned}
						& v_{tt}+b_0v_t-\sum\limits_{i,j=1}^na_{ij}v_{x_ix_j}=\tilde{b}v+\sum\limits_{i=1}^nb_iv_{x_i}-2\chi\cdot z_t, & (t,x)\in & ~ (0,T)\times\Omega, \\
						& v(t,x)=0, & (t,x)\in & ~ (0,T)\times\partial\Omega, \\
						& v(0,x)=y^1, ~ v(T,x)=0 , ~ v_t(T,x)=0, & x\in & ~ \Omega, \\
						& v_t(0,x)=-2y^1+\Delta y^0-y^0+F(y^0,y^1,\nabla y^0,\nabla^2y^0), & x\in & ~ \Omega,
					\end{aligned}
					\right.
				\end{equation}
				and $y$ satisfies
				\begin{equation}\label{eqnofyfinally:1.3}
				v_t+2v-\Delta y+y=F\big(y,v,\nabla y,\nabla^2y\big)-2\chi\big(z(t)-z(0)\big).
				\end{equation}
				
				Denote $\bar{v}=y_t$, differentiating \eqref{eqnofyfinally:1.3} by $t$, we get
                \begin{equation*}
                v_{tt}+2v_t-\Delta\bar{v}+\bar{v}=F_y\cdot\bar{v}+F_v\cdot v_t+\sum\limits_{i=1}^nF_{y_{x_i}}\bar{v}_{x_i}+\sum\limits_{i,j=1}^nF_{y_{x_ix_j}}\bar{v}_{x_ix_j}-2\chi\cdot z_t,
                \end{equation*}
                which can be written as
				\begin{equation}
				v_{tt}+b_0v_t-\sum\limits_{i,j=1}^na_{ij}\bar{v}_{x_ix_j}=\tilde{b}\bar{v}+\sum\limits_{i=1}^nb_i\bar{v}_{x_i}-2\chi\cdot z_t.
				\end{equation}
				
				Subtracting from \eqref{eqn of v finally}, we get
				$$
				\left\{
				\begin{aligned}
					& \sum\limits_{i,j=1}^na_{ij}(v-\bar{v})_{x_ix_j}+\sum\limits_{i=1}^nb_i(v-\bar{v})_{x_i}+\tilde{b}(v-\bar{v})=0, & x\in & ~ \Omega, \\
					& v-\bar{v}=0, & x\in & ~ \partial\Omega.
				\end{aligned}
				\right.
				$$
				
				Noting that $a_{ij}, ~ b_i, ~ \tilde{b}$ are functions of $y$ and $v$, this is a linear equation of $v-\bar{v}$. To prove $v=\overline{v}$, we multiply the equation by $-v+\bar{v}$ and make an integration by parts, then we get
				$$\int_\Omega\sum\limits_{i,j=1}^na_{ij}(v-\bar{v})_{x_i}(v-\bar{v})_{x_j}\dd x=\int_\Omega\sum\limits_{i=1}^n\big(b_i-\partial_{x_j}a_{ij}\big)(v-\bar{v})(v-\bar{v})_{x_i}+\tilde{b} (v-\bar{v})^2\dd x.$$
				
				Noting that $|b_i-\partial_ja_{ij}|+|\tilde{b}+1|+|a_{ij}-\delta_{ij}|=O(\varepsilon)$, we have
				
				$$\int_\Omega\big|\nabla(v-\overline{v})\big|^2\dd x\leq\frac{3C\varepsilon-1}{1-2C\varepsilon}\int_\Omega|v-\overline{v}|^2\dd x.$$
				
				Taking $\varepsilon$ small enough such that $C\varepsilon<\frac{1}{3}$, hence we get $v=\bar{v}=y_t$, satisfying
				$$y_t(T)=0, ~ y_{tt}(T)=0.$$
				
				Then
				$$u(t)=-2\chi\left(z(t)-z(0)\right),$$
				is the desired control function.
		\end{proof}
		
		\appendix
		
		\section{Proof of Theorem \ref{fully nonlinear observability}}\label{appdix}
		The appendix is devoted to showing the proof of Theorem \ref{fully nonlinear observability}.
		Denote
		$$Q^T:=(0,T)\times\Omega, \quad \Gamma^T:=(0,T)\times\partial\Omega.$$
		
		Consider the following linear  system
		\begin{equation}\label{linear hyperbolic observability system-A}
			\left\{
			\begin{aligned}
				& z_{tt}+b_0z_t-\sum\limits_{i,j=1}^n\big(a^{ij}z_{x_i}\big)_{x_j}+\sum\limits_{k=1}^nb_kz_{x_k}+\tilde{b}z=0, & (t,x)\in & ~ Q^T, \\
				& z(t,x)=0, & (t,x)\in & ~ \Gamma^T, \\
				& z(0,x)=z_0, ~ z_t(0,x)=z_1 & x\in & ~ \Omega,
			\end{aligned}
			\right.
		\end{equation}
		with
		\begin{equation}\label{coecond:normA}
			\begin{cases}
				\|a^{ij}-\delta_{ij}\|_{  C^{1}(\overline{Q}^T)}<\tilde{\varepsilon}, i,j=1,\cdots, n \\
				\|b_0-1\|_{  C^{1}(\overline{Q}^T)}<\tilde{\varepsilon}, \|\tilde{b}\|_{  C^{0}(\overline{Q}^T)}<\tilde{\varepsilon}\\	  \|b_k\|_{ C^{0}(\overline{Q}^T)}<\tilde{\varepsilon}, k=1,\cdots, n.
			\end{cases}
		\end{equation}
				
		In this section, we will prove Theorem \ref{fully nonlinear observability}, i.e. there exists $\varepsilon_5>0$, such that if $\varepsilon_1\leq \varepsilon_5$,  the observability inequality holds for any solution of \eqref{linear hyperbolic observability system-A}
		\begin{equation}\label{linear hyperbolic observability inequality-A}
			\|z_1\|_{L^2(\Omega)}^2+\|z_0\|_{\mathcal{H}^1(\Omega)}^2\leq D\int_0^T\int_\omega|z_t|^2\dd x\dd t,
		\end{equation}
		for some constant $D>0$ depends on $T, \varepsilon_5$, $n$, $\Omega$ and $\omega$.
		
	We attempt to apply the methodology presented in \cite{Tebou, Fu3} to construct a proof for inequality \eqref{linear hyperbolic observability inequality-A}. The proof is based on Carleman estimate and  primarily divided into two main steps. First, we establish the $H^1$-norm Carleman estimate as detailed in Proposition \ref{internal Carleman theorem}. Following this, to derive the observability inequality, it is essential to eliminate the $z^2$ term appearing on the right-hand side of the inequality. To accomplish this, we proceed to establish an $L^2$-norm Carleman estimate.
		
		In order to secure the Carleman estimate for the system described by \eqref{linear hyperbolic observability system-A}, we commence by confirming that the assumption \eqref{defi:weak-Gamma} concerning $(T, \omega)$ yields the subsequent property. This property is instrumental in laying the groundwork for the derivation of the Carleman estimate in the $H^1$ norm.
		\begin{lem}\label{lem:A1}
		Assume that $(T,\omega)$ satisfies assumption \ref{defi:weak-Gamma}. Assume \eqref{coecond:normA} is valid. Define $\psi(x)=\frac{2\max\limits_{x\in\overline{\Omega}}|x-x_0|^2}{\min\limits_{x\in\overline{\Omega}}|x-x_0|^2}|x-x_0|^2$, there exists a small $\varepsilon_6$ depends on $\varepsilon_0$, $\Omega$ and $n$, such that if in \ref{defi:weak-Gamma}, we have $\tilde{\varepsilon}\leq \varepsilon_6$,  then the following statements are valid.
	\begin{description}
		\item[(a)]	There exists a positive constant $\mu_0>4$, 
            such that for any $(t,x,\xi)\in\overline{Q^T} \times\mathbb{R}^n$,
		\begin{align}\label{assumption:1.1}
            \sum\limits_{j,k=1}^n\sum\limits_{j',k'=1}^n\left[2a^{jk'}(a^{j'k}\psi_{x_{j'}})_{x_{k'}}\right]\xi^j\xi^k\geq\mu_0\sum\limits_{j,k=1}^na^{jk}\xi^j\xi^k,
		\end{align}
		and for any $(t,x)\in\overline{Q^T}$,
            \begin{equation}\label{condition of psi}
		\frac{1}{4}\sum\limits_{j,k=1}^na^{jk}(t,x)\psi_{x_j}\psi_{x_k}\geq 
            \max\limits_{x\in\overline{\Omega}}\psi(x)\geq\min\limits_{x\in\overline{\Omega}}\psi(x)\geq 0.
		\end{equation}
		\item [(b)]Let
		\begin{equation}\label{Gamma t}
			\Gamma_t=\Big\{x\in\partial\Omega:\sum\limits_{j,k=1}^na^{jk}(t,x)\psi_{x_j}n^k>0\Big\},
		\end{equation}
		
		and for $\varepsilon_0>0$,
		$$O_{\varepsilon_0}(\Gamma_t)=\{x\in\mathbb{R}^n:d(x,\Gamma_t)<\varepsilon_0\},$$
		
		we have
		\begin{equation}\label{condition omega-A}
			\Big(\bigcup_{t\in[0,T]}O_{\varepsilon_0}(\Gamma_t)\Big)\cap ~ \Omega\subseteq\omega.
		\end{equation}
		\end{description}
		
		\end{lem}

			\begin{proof}[Proof of (a)]
				First, we prove \eqref{assumption:1.1}. Since  $a^{ij}$ satisfies \eqref{coecond:normA}, we have that
					\begin{equation}\label{aaaa}
					\|a^{ij}-\delta_{ij}\|_{  C^{1}(\overline{Q}^T)}<\varepsilon_1, i,j=1,\cdots, n.
				\end{equation}
				This implies that for any $\mu_0$
				\begin{align}\label{A7}
					\mu_0\sum\limits_{j,k=1}^na^{jk}\xi^j\xi^k\leq\mu_0\big(1+n^2\varepsilon_1\big)|\xi|^2.
				\end{align}
			
			Let $d=\frac{\max\limits_{x\in\overline{\Omega}}|x-x_0|^2}{\min\limits_{x\in\overline{\Omega}}|x-x_0|^2}$. Direct computation shows
				\begin{align}\label{aa9}
					& ~ \sum\limits_{j,k=1}^n\sum\limits_{j',k'=1}^n\left[2a^{jk'}(a^{j'k}\psi_{x_{j'}})_{x_{k'}}\right]\xi^j\xi^k \nonumber \\
					= & ~ \big[8d|\xi|^2+2\sum\limits_{j,k=1}^n\sum\limits_{j',k'=1}^n\left[a^{jk'}(a^{j'k}\psi_{x_{j'}})_{x_{k'}}-4d\delta_{jk}\right]\xi^j\xi^k\big] \\
					\geq & ~ 8|\xi|^2-2\sum\limits_{j,k=1}^n\left|\sum\limits_{j',k'=1}^na^{jk'}(a^{j'k}\psi_{x_{j'}})_{x_{k'}}-4d\delta_{jk}\right||\xi^j\xi^k|. \nonumber
				\end{align}
				
			We observe that
				\begin{align}
					& ~ \sum\limits_{j',k'=1}^na^{jk'}\left(a^{j'k}\psi_{x_{j'}}\right)_{x_{k'}}-4d\delta_{jk} \nonumber \\
					= & ~ 4d\sum\limits_{j',k'=1}^na^{jk'}\left(a^{j'k}\big(x^{j'}-x_0^{j'}\big)\right)_{x_{k'}}-4d\delta_{jk} \\
					= & ~ 4d\sum\limits_{j',k'=1}^na^{jk'}\left(\big(a^{j'k}-\delta_{j'k}\big)\big(x^{j'}-x_0^{j'}\big)\right)_{x_{k'}}+4d\big(a^{jk}-\delta_{jk}\big). \nonumber
				\end{align}
				By \eqref{aaaa}, we obtain
				\begin{align}
				&\bigg|4d\sum\limits_{j',k'=1}^na^{jk'}\left(\big(a^{j'k}-\delta_{j'k}\big)\big(x^{j'}-x_0^{j'}\big)\right)_{x_{k'}}+4d\big(a^{jk}-\delta_{jk}\big)\bigg| \nonumber\\
                \leq~& d\Big(n\sqrt{n} (\tilde{\varepsilon}+1)\tilde{\varepsilon} \sqrt{\max_{x\in \bar{\Omega}}\{\psi(x)\}}+4n(1+\tilde{\varepsilon}) \tilde{\varepsilon}+4\tilde{\varepsilon}\Big).
				\end{align}
				
			Substituting this into \eqref{aa9}, we have
				\begin{align}\label{A10}
					\sum\limits_{j,k=1}^n\sum\limits_{j',k'=1}^n\left[2a^{jk'}(a^{j'k}\psi_{x_{j'}})_{x_{k'}}\right]\xi^j\xi^k\geq (8-C(n,\tilde{\varepsilon},\psi)\tilde{\varepsilon})|\xi|^2,
				\end{align}
				where
				$C(n,\tilde{\varepsilon},\psi)=d\big(n^2\sqrt{n} (\tilde{\varepsilon}+1) \max_{x\in \bar{\Omega}}\{\psi(x)\} +4n^2(1+\tilde{\varepsilon})+4n\big)$.
				
				Combining \eqref{A7} and \eqref{A10}, we know that by choosing $\varepsilon_6$ small enough such that
				\begin{equation}\label{ep61}
					C(n,\varepsilon_6,\psi)\varepsilon_6<4
				\end{equation}
				 then we have \eqref{assumption:1.1} for some $\mu_0>4$.
				
				For the proof of \eqref{condition of psi}, it suffices to check the first part of the inequality. We compute
				\begin{equation}						\frac{1}{4}\sum\limits_{j,k=1}^na^{jk}(t,x)\psi_{x_j}\psi_{x_k}=4d^2 |x-x_0|^2+\frac{1}{4}\sum\limits_{j,k=1}^n\big(a^{jk}(t,x)-\delta^{jk}\big)\psi_{x_j}\psi_{x_k},
				\end{equation}
				which implies that
				\begin{equation}
					\frac{1}{4}\sum\limits_{j,k=1}^na^{jk}(t,x)\psi_{x_j}\psi_{x_k}\geq 4d^2|x-x_0|^2-nd^2\tilde{\varepsilon} |x-x_0|^2.
				\end{equation}
			Choosing  $\varepsilon_6< \frac{2}{n}$ and satisfies \eqref{ep61}, then
				this implies that
				\begin{equation}
						\frac{1}{4}\sum\limits_{j,k=1}^na^{jk}(t,x)\psi_{x_j}\psi_{x_k}\geq 2d^2|x-x_0|^2\geq 2 \max\lim\limits_{x\in \overline{\Omega}}|x-x_0|^2.
				\end{equation}
			Thus, \eqref{condition of psi} is valid for some small $\varepsilon_6>0$
			\end{proof}

			\begin{proof}[Proof of (b)]
				It suffices to demonstrate that for $\varepsilon_0>0$,
				\begin{align}\label{prop:2:need}
					\bigcup_{t\in[0,T]}O_{\varepsilon_0}(\Gamma_t)\subseteq O_{\varepsilon_0}(\Gamma).
				\end{align}
				
				Utilizing the definition of  $\Gamma_t$ and $\psi$, we have
				\begin{align}
					\Gamma_t= & ~ \Big\{x\in\partial\Omega:\sum\limits_{j,k=1}^na^{jk}(t,x)\psi_{x_j}n^k>0\Big\} \nonumber \\
					= & ~ \Big\{x\in\partial\Omega:(x-x_0)\cdot\bm{\nu}>\sum\limits_{j,k=1}^n\big(\delta_{jk}-a^{jk}(t,x)\big)\psi_{x_j}n^k\Big\} 			
				\end{align}
				Recalling \eqref{aaaa}, it follows that
			\begin{equation}			|	\sum\limits_{j,k=1}^n\big(\delta_{jk}-a^{jk}(t,x)\big)\psi_{x_j}n^k|\leq \tilde{\varepsilon}\sqrt{\frac{n}{2}\max_{x\in \bar{\Omega}}\{\psi(x)\} }
				\end{equation}
				Thus, we obtain
				\begin{equation}
						\Gamma_t \subseteq  ~ \Big\{x\in\partial\Omega:(x-x_0)\cdot\bm{\nu}>-C(n,\psi)\varepsilon_1\Big\},
				\end{equation}
			which holds for  all $ ~ t\in[0,T]$. Consequently, by selecting $\varepsilon_6$ sufficiently small such that  $C(n,\psi)\varepsilon_6< \varepsilon_0$, we deduce
				\begin{align}
					\bigcup_{t\in[0,T]}\Gamma_t\subseteq\Gamma,				\end{align}
				and \eqref{condition omega-A} is satisfied. This completes the proof.
		\end{proof}
		
		We now state a Carleman estimate in the $H^1$-norm. Denote
		\begin{equation}\label{definitions in Carleman estimate}
			\left\{
			\begin{aligned}
				& v(t,x)=\theta z, ~ \theta(t,x)=e^{l(t,x)}, ~ l(t,x)=\lambda\phi(t,x), \\
				& \phi(t,x)=\psi(x)-c_1(t-T/2)^2,  c_1\in(0,1).
			\end{aligned}
			\right.
		\end{equation}

	\begin{prop}\label{internal Carleman theorem}
		Assuming that $(T,\omega)$  satisfies the condition given in Assumption \ref{defi:weak-Gamma} and that \eqref{coecond:normA} holds with $\tilde{\varepsilon}\leq \varepsilon_6$ in Lemma \ref{lem:A1}. 			
			Then there exists a constant $\lambda_0>0$ such that for all $\lambda>\lambda_0$, any $z\in H^1_0(Q^T)$ fulfills the following internal Carleman estimate
			\begin{equation}\label{internal Carleman estimate}
				\begin{aligned}
					& \int_{Q^T}\theta^2\Big(\lambda(z_t^2+|\nabla z|^2)+\lambda^3z^2\Big)\dd x\dd t \\
					\leq & ~ C\bigg(\int_{Q^T}\theta^2|z_{tt}-\sum\limits_{i,j=1}^n\big(a^{ij}z_{x_i}\big)_{x_j}|^2\dd x\dd t+\lambda^2\int_0^T\int_{\omega}\theta^2(z_t^2+\lambda^2z^2)\dd x\dd t\bigg).
				\end{aligned}
			\end{equation}
            \end{prop}
		
		\begin{rem}
				Let
			\begin{equation}\label{the choice of T1}
				T_1=\max\left\{2\sqrt{\kappa_1} ~, ~ 1+25s_0(n+2)\sqrt{n}\right\},
			\end{equation}
			where
			$$\kappa_1=\max\limits_{t\in[0,T],x\in\overline{\Omega}}\sum\limits_{j,k=1}^na^{jk}\psi_{x_j}\psi_{x_k}, \quad s_0=\max\limits_{t\in[0,T],x\in\partial\Omega}\sum\limits_{j,k=1}^n a^{jk}\psi_{x_j}n^k.$$
		Direct computation shows that if  $(T,\omega)$ satisfies Assumption \ref{defi:weak-Gamma}, then $T>T_1$.
		\end{rem}
		
		The proof of this Lemma can follow the procedures outlined in \cite[Chapter 4]{Fu2}, and thus we do not provide a detailed proof here.

	As mentioned at the beginning of this section, in order to obtain \eqref{linear hyperbolic observability inequality-A}, we need to eliminate the $z^2$ terms on the right-hand side of \eqref{internal Carleman estimate}. Following the approach in \cite{Fu3}, we need consider the $L^2$-norm Carleman estimate for the following system:
	\begin{equation}\label{hyperbolic observability F-system-B}
		\left\{
		\begin{aligned}
			& z_{tt}+b_0z_t-\sum\limits_{i,j=1}^n\big(a^{ij}z_{x_i}\big)_{x_j}+\sum\limits_{k=1}^nb_kz_{x_k}+\tilde{b}z=F, &  (t,x)\in & ~ Q^T, \\
			& z(t,x)=0, &  (t,x)\in & ~ \Gamma^T,
		\end{aligned}
		\right.
	\end{equation}
	where $F\in L^1(0,T;H^{-1}(\Omega))$ and $a^{ij},b_0,b_i,\tilde{b}, i,j=1,2,\cdots,n$ satisfies \eqref{coecond:normA}.
	
	The $L^2$ estimate requires consideration of the weak solution to system \eqref{hyperbolic observability F-system-B}:
	\begin{defn}
		A function $z\in L^2((0,T)\times\Omega)$ is called a weak solution to \eqref{hyperbolic observability F-system-B} if
		\begin{equation}\label{define weak solution-B}
            \bigg(z ~ , ~ \eta_{tt}-\sum\limits_{j,k=1}^n\big(a^{jk} \eta_{x_j}\big)_{x_k}\bigg)_{L^2(Q^T)}=\int_0^T\langle f(t,\cdot), \eta(t,\cdot)\rangle_{H^{-1}(\Omega), H_0^1(\Omega)} \dd t,
            \end{equation}
        holds for any $\eta\in H_0^2\big(0,T;H^2(\Omega)\cap H_0^1(\Omega)\big)$.
	\end{defn}
	
	Note that there are no initial data in \eqref{hyperbolic observability F-system-B}, so we need the following lemma for weak solutions.
	\begin{lem}\label{lemma4.1 in the book}
		Given $0<t_1<t_2<T$ and $g\in L^2((t_1,t_2)\times\Omega)$. Assume that $z\in L^2(Q^T)$ is a weak solution to \eqref{hyperbolic observability F-system-B} with $z=g$ in $(t_1,t_2)\times\Omega$. Assume that  \eqref{coecond:normA} is valid. Then there exists a small constant $\varepsilon_7$, such that if $\tilde{\varepsilon}\leq \varepsilon_7$ in  \eqref{coecond:normA},  then we have
		$$z\in C([0,T];L^2(\Omega))\cap C^1([0,T];H^{-1}(\Omega)),$$
		
		and there exists a constant $C=C(T,t_1,t_2,\Omega,\tilde{\varepsilon})>0$, such that
		\begin{equation}\label{zfg}
			\|z\|_{C([0,T];L^2(\Omega))\cap C^1([0,T];H^{-1}(\Omega))}\leq C\big(\|f\|_{L^1(0,T;H^{-1}(\Omega))}+\|g\|_{L^2((t_1,t_2)\times\Omega)}\big).
		\end{equation}

	\end{lem}
	
	The proof of this lemma differs from that of \cite[Lemma 5.1]{zhangzua03} solely in that the coefficients are time-dependent.
	This results in additional terms appearing during the regularization process of the solution with respect to time. Nonetheless, by leveraging the smallness assumption \eqref{coecond:normA}, we can achieve the desired conclusion.
	\begin{proof}[Proof of Lemma \ref{lemma4.1 in the book}]
	{\color{black}	
		Fix arbitrary
	 $t_i, i=3,4$ satisfying
		\begin{equation}
			t_1<t_3<t_4<t_2.
		\end{equation}
		For any $\delta\in (0,\min(t_3-t_1,t_2-t_4))$, we have for any $t,x\in (t_3,t_6)\times \Omega$,
		\begin{equation}
			z^{\delta}:=(z*\rho_\delta)(t,x)=\int_{-\infty}^{+\infty}z(s,x)\rho_\delta(t-s)\dd s,
		\end{equation}
		where $\rho_\delta \in C_0^\infty(\R)$ is a Friedrichs mollifier.
		
	  According to equation \eqref{hyperbolic observability F-system-B}, we can verify that $z^\delta\in C^\infty([t_3,t_4];L^2)$ satisfies
	 	\begin{equation}\label{hyperbolic observability F-system-BB}
	 	\left\{
	 	\begin{aligned}
	 		& z^\delta_{tt}+z^\delta_t-\Delta z^\delta=F_1^\delta+F_2^\delta, &  (t,x)\in & ~ Q^T, \\
	 		& z^\delta(t,x)=0, &  (t,x)\in & ~ \Gamma^T,
	 	\end{aligned}
	 	\right.
	 \end{equation}
		where $F^\delta_1=F*\rho_\delta$ and
		\begin{align*}
            F_2^\delta=&\int_{-\infty}^{+\infty}(b_0(t)-b_0(s))z_t(s,x)\rho_\delta(t-s)\dd s \\
		&-\int_{-\infty}^{+\infty}\sum\limits_{i,j=1}^n\big((a^{ij}(s,x)- 
            a^{ij}(t,x))z_{x_i}(s,x)\big)_{x_j}\rho_\delta(t-s)\dd s \\
            &+ \int_{-\infty}^{+\infty}(b_k(t)-b_k(s))z_{x_k}(s,x)\rho_\delta(t-s)\dd s+\int_{-\infty}^{+\infty}(\tilde{b}(t)-\tilde{b}(s))z(s,x)\rho_\delta(t-s)\dd s.
            \end{align*}

		Since for any $t \in (t_3, t_4)$, $\rho_\delta(t-s)$ has compact support. Then we can use integration by parts in the sense of distributions, to deduce that
		\begin{equation}
				\int_{-\infty}^{+\infty}(b_0(t)-b_0(s))z_t(s,x)\rho_\delta(t-s)\dd s= \int_{-\infty}^{+\infty}z(s,x)\partial_s\big((b_0(t)-b_0(s))\rho_\delta(t-s)\big)\dd s.
		\end{equation}

		Denote $(-\Delta)^{-1}$ as the inverse of the Laplacian operator  $-\Delta$ with Dirichlet boundary conditions. Thus, for any
		 $u,v\in L^2(\overline\Omega), a \in C^1(\overline\Omega), b, c\in C^0(\overline\Omega)$, if $(-\Delta)^{-1}v=v=u=0$ on $\partial\Omega$, then
		\begin{equation}
			\begin{split}
				((au_{x_i})_{x_j},(\Delta)^{-1}v)_{H^{-2},H^2}\leq~& C_1\|a\|_{C^1}\|u\|_{L^2}\|v\|_{L^2}, \\
				((bu_{x_i}),(\Delta)^{-1}v)_{H^{-2},H^2}\leq~& C_2\|b\|_{C^0}\|u\|_{L^2}\|v\|_{H^{-1}}, \\	((cu),(\Delta)^{-1}v)_{H^{-2},H^2}\leq~& C_3\|c\|_{C^0}\|u\|_{L^2}\|v\|_{H^{-2}}\leq C_4\|c\|_{C^0}\|u\|_{L^2}\|v\|_{L^{2}},
			\end{split}
		\end{equation}
		 where $C_i,i=1,2,3,4$ depend only on $\Omega$ and $n$.
		 Multiplying the equation for
		$z^\delta$
		 in system \eqref{hyperbolic observability F-system-BB} by $\Delta^{-1}z^\delta$, integrating over $\Omega$, and using \eqref{coecond:normA}, we obtain:
		 \begin{align}\label{zdelta:1}
		 	&\|z^\delta\|_{C^0(t_3,t_4; L^2)\cap C^1(t_3,t_4; H^{-1}) }^2\nonumber\\
    \leq ~&\tilde{C}\big(\|F_1^\delta\|_{L^2(t_3,t_4; L^2)}	\|z^\delta\|_{C^0(t_3,t_4; L^2) }+\tilde{\varepsilon} \|z\|_{L^2(t_1,t_2; L^2)}	\|z^\delta\|_{C^0(t_3,t_4; L^2)\cap C^1(t_3,t_4; H^{-1})}\big).
		 \end{align}
		 Here $\tilde{C} >0$ is a constant independent with $\delta, z,z^\delta, F^\delta_1$ and $\tilde{\varepsilon}$.
		
		Thus, together with the assumption that  $z=g$ in $(t_1,t_2)\times\Omega$, it immediately implies that
		  \begin{equation}\label{zdelta:2}
		 	\|z^\delta\|_{C^0(t_3,t_4; L^2)\cap C^1(t_3,t_4; H^{-1}) }^2\leq \tilde{C}_1\big(\|F_1^\delta\|_{L^2(t_3,t_4; L^2)}^2	+ \|g\|_{L^2(t_1,t_2; L^2)}^2	\big).
		 \end{equation}
		
		  Letting $\delta$ tends to zero and using the properties of the Friedrichs mollifier $\rho_\delta$, we can conclude that  $z\in C([t_3,t_6];L^2(\Omega))\cap C^1([t_3,t_6];H^{-1}(\Omega))$ and
		    \begin{equation}\label{zdelta2}
		  	\|z\|_{C^0(t_3,t_6; L^2)\cap C^1(t_3,t_6; H^{-1}) }^2\leq \tilde{C}_1\big(\|F\|_{L^2(t_1,t_2; L^2)}^2	+ \|g\|_{L^2(t_1,t_2; L^2)}^2	\big).
		  \end{equation}
		
		 Since \eqref{hyperbolic observability F-system-B} is a linear system, using the well-posedness theory of linear wave equations, we can get \eqref{zfg}. Therefore, we complete the proof.}
	\end{proof}

	Our Carleman estimate for the above hyperbolic operators in $L^2$-norm is as follows.
                \begin{prop}\label{prop:L2 Carleman theorem}
				Assuming that $(T,\omega)$  satisfies the condition given in Assumption \ref{defi:weak-Gamma} and that \eqref{coecond:normA} holds with $\tilde{\varepsilon}\leq \varepsilon_6$ in Lemma \ref{lem:A1}.  Let $T_1$ be given in \eqref{the choice of T1}. Then there exists a constant $\lambda_0^*>0$ such that for $\forall ~ T>T_1$ and $\lambda>\lambda_0^*$, and every solution $z\in C^0([0,T]; L^2(\Omega))$ satisfying $z(0,x)=z(T,x)=0, ~ x\in\Omega$ and $$z_{tt}-\sum\limits_{j,k=1}^n\big(a^{jk}z_{x_j}\big)_{x_k}\in H^{-1}\big(Q^T\big),$$
                it holds
			\begin{equation}\label{L2 Carleman estimate}
				\begin{aligned}
					& \lambda\int_{Q^T}\theta^2z^2\dd x\dd t \\
					\leq & ~ C\bigg(\Big\|\theta\Big(z_{tt}-\sum\limits_{j,k=1}^n\big(a^{jk}z_{x_j}\big)_{x_k}+2z_t+z\Big)\Big\|_{H^{-1}(Q^T)}^2+\lambda^2\int_0^T\int_\omega\theta^2 z^2\dd x\dd t\bigg).
				\end{aligned}
			\end{equation}
            \end{prop}
		
		We will first assume Proposition \ref{prop:L2 Carleman theorem} and then provide the proof for Theorem \ref{fully nonlinear observability} . The proof for Proposition \ref{prop:L2 Carleman theorem} will be presented later.
		\begin{proof}[Proof of Theorem \ref{fully nonlinear observability}]
			
		For any $(z_0,z_1)\in H_0^1(\Omega)\times L^2(\Omega)$, the system \eqref{linear hyperbolic observability system-A} admits a unique solution
		$$z\in C([0,T];H_0^1(\Omega))\cap C^1([0,T];L^2(\Omega)).$$
		
		Define the energy of the system by
		$$\mathcal{E}(t)=\frac{1}{2}\int_\Omega\Big[|z_t(t)|^2+\sum\limits_{j,k=1}^na^{jk}z_{x_j}(t)z_{x_k}(t)+|z(t)|^2\Big]\dd x.$$
        
        Multiplying the system \eqref{linear hyperbolic observability system-A} by $z_t$, integrating it on $\Omega$, and using integration by parts, we get
		\begin{equation}\label{energy estimate half}
			\mathcal{E}'(t)+2\int_\Omega b_0z_t^2\dd x=-\int_\Omega\Big(\sum\limits_{k=1}^nb_kz_tz_{x_k}+\tilde{b}zz_t\Big)\dd x\geq-C\varepsilon\mathcal{E}(t).
		\end{equation}
		
		Since that
		$$\int_\Omega b_0z_t^2\dd x\leq(2+C\varepsilon)\mathcal{E}(t),$$
		we have
		$$\mathcal{E}'(t)+(2+C\varepsilon)\mathcal{E}(t)=e^{-(2+C\varepsilon)t}\frac{\dd}{\dd t}\big(e^{(2+C\varepsilon)t}\mathcal{E}(t)\big)\geq 0.$$
		
		Integrating the above inequality on $(0,T)$, we get
		\begin{equation}\label{E0 and ET}
			e^{(2+C\varepsilon)T}\mathcal{E}(T)\geq\mathcal{E}(0).
		\end{equation}
		
		\textbf{Step 1.} We put
		$$\tilde{T}_j=\Big(\frac{1}{2}-\varepsilon_j\Big)T, \quad \tilde{T}'_j=\Big(\frac{1}{2}+\varepsilon_j\Big)T, \quad j=0,1$$
		for constants $0<\varepsilon_0<\varepsilon_1<\frac{1}{2}$.
		
		Then we choose a nonnegative cut-off function $\tilde{\zeta}\in C_0^2([0,T])$ such that
		\begin{equation}\label{zeta is 1}
			\tilde{\zeta}(t)\equiv 1, \qquad \forall ~ t\in[\tilde{T}_1,\tilde{T}'_1].
		\end{equation}
		
		Set $\tilde{z}(t,x)=\tilde{\zeta}(t)z_t(t,x)$ for $(t,x)\in Q^T$. Then $\tilde{z}$ solves
		\begin{equation}\label{system:z}
			\left\{
			\begin{aligned}
				& \tilde{z}_{tt}-\sum\limits_{j,k=1}^n(a^{jk}\tilde{z}_{x_j})_{x_k}+2\tilde{z}_t+\tilde{z}=\tilde{\zeta}_{tt}z_t+\tilde{\zeta}_tz_t+2\tilde{\zeta}_tz_{tt}+\tilde{\zeta}(2-b_0)z_{tt} & & \\
				& +\tilde{\zeta}\bigg[\sum\limits_{j,k=1}^n(a^{jk}_tz_{x_j})_{x_k}-\sum\limits_{k=1}^n(b_kz_{x_k})_t-(\tilde{b}-1+\partial_tb_0)z_t-\tilde{b}_tz\bigg], & (t,x)\in & ~ Q^T, \\
				& \tilde{z}(t,x)=0, & (t,x)\in & ~ \Gamma_T, \\
				& \tilde{z}(0,x)=\tilde{z}(T,x)=0, & x\in & ~ \Omega.
			\end{aligned}
			\right.
		\end{equation}
		
		Let $T_1$ and $\phi$ be given by \eqref{the choice of T1} and \eqref{definitions in Carleman estimate}. Then by Proposition \ref{prop:L2 Carleman theorem}, there exists $\lambda_0^*>0$ such that for all $T>T_1$ and $\lambda\geq \lambda_0^*$, it holds that
		\begin{equation}\label{4.133 in the book}
			\begin{aligned}
				& \lambda\int_{Q^T}\theta^2\tilde{z}^2\dd x\dd t \\
				\leq ~ & C\bigg(\Big\|\theta\big(\tilde{\zeta}_{tt}z_t+\tilde{\zeta}_tz_t+2\tilde{\zeta}_tz_{tt}+\tilde{\zeta}(2-b_0)z_{tt}\big)\Big\|_{H^{-1}(Q^T)}^2+\lambda^2\int_0^T \int_\omega\theta^2\tilde{z}^2\dd x\dd t \\
				+ & \bigg\|\theta\tilde{\zeta}\Big[\sum\limits_{j,k=1}^n(a^{jk}_tz_{x_j})_{x_k}-\sum\limits_{k=1}^n(b_kz_{x_k})_t-(\tilde{b}-1+\partial_tb_0)z_t-\tilde{b}_tz\Big]\bigg\|_{H^{-1} (Q^T)}^2\bigg).
			\end{aligned}
		\end{equation}
		
		Using H$\mathrm{\ddot{o}}$lder inequality and Sobolev embedding theorem, we find that
		\begin{equation}\label{4.134 in the book}
			\left\{
			\begin{aligned}
				\big\|\theta(\tilde{\zeta}_{tt}+\tilde{\zeta}_t)z_t\big\|_{H^{-1}(Q^T)} & \leq C\|\theta z_t\|_{L^2(\tilde{Q})} \\
				\big\|2\theta\tilde{\zeta}_tz_{tt}\big\|_{H^{-1}(Q^T)} & \leq C(1+\lambda)\|\theta z_t\|_{L^2(\tilde{Q})} \\
				\big\|\theta\tilde{\zeta}(2-b_0)z_{tt}\big\|_{H^{-1}(Q^T)} & \leq C(1+\lambda)\varepsilon\|\theta z_t\|_{L^2(Q^T)},
			\end{aligned}
			\right.
		\end{equation}
		where $\tilde{Q}=\big((0,\tilde{T}_1)\cup(\tilde{T}'_1,T)\big) 
            \times\Omega$, and
		\begin{equation}\label{Addition:5}
			\begin{aligned}
				& \bigg\|\theta\tilde{\zeta}\Big[\sum\limits_{j,k=1}^n(a^{jk}_tz_{x_j})_{x_k}-\sum\limits_{k=1}^n(b_kz_{x_k})_t-(\tilde{b}-1+\partial_tb_0)z_t-\tilde{b}_tz\Big]\bigg\|_{H^{-1} (Q^T)} \\
				\leq ~ & C(1+\lambda)\varepsilon\big(\|\theta\nabla z\|_{L^2(Q^T)}+\|\theta z_t\|_{L^2(Q^T)}\big).
			\end{aligned}
		\end{equation}
		
		Combining \eqref{system:z}--\eqref{Addition:5}, we have
		\begin{equation}\label{4.135 in the book}
			\begin{aligned}
				& \lambda\|\theta\tilde{z}\|_{L^2(Q^T)}^2 \\
				\leq ~ & C\lambda^2\|\theta z_t\|_{L^2(\tilde{Q})}^2+C\lambda^2\|\theta\tilde{z}\|_{L^2((0,T)\times\omega)}^2+C\lambda^2\varepsilon^2\Big(\|\theta\nabla z\|_{L^2(Q^T)}^2 +\|\theta z_t\|_{L^2(Q^T)}^2\Big) \\
				\leq ~ & C\lambda^2\|\theta z_t\|_{L^2(\tilde{Q})}^2+C\lambda^2\int_0^T\int_\omega\theta^2z_t^2\dd x\dd t+C\lambda^2\varepsilon^2\Big(\|\theta\nabla z\|_{L^2(Q^T)}^2+\|\theta z_t\|_{L^2(Q^T)}^2\Big).
			\end{aligned}
		\end{equation}
		
		On the other hand, by \eqref{zeta is 1}, we find that
		$$\|\theta\tilde{z}\|_{L^2(Q^T)}^2\geq\int_{\tilde{T}_1}^{\tilde{T}'_1}\int_\Omega\theta^2z_t^2\dd x\dd t.$$ 
        Thus we have
		\begin{equation}\label{4.136 in the book}
			\|\theta z_t\|_{L^2(Q^T)}^2\leq\|\theta\tilde{z}\|_{L^2(Q^T)}^2+\|\theta z_t\|_{L^2(\tilde{Q})}^2.
		\end{equation}
		
		It follows from \eqref{4.135 in the book} and \eqref{4.136 in the book} that
		\begin{equation}\label{4.137 in the book}
			\|\theta z_t\|_{L^2(Q^T)}^2\leq C\lambda\Big(\|\theta z_t\|_{L^2(\tilde{Q})}^2+\varepsilon^2\|\theta\nabla z\|_{L^2(Q^T)}^2+\int_0^T\int_\omega\theta^2z_t^2\dd x\dd t\Big).
		\end{equation}
		
		\textbf{Step 2.} We set
		$$R_0=\min\limits_{x\in\overline{\Omega}}\sqrt{\psi(x)},\quad R_1=\max\limits_{x\in\overline{\Omega}}\sqrt{\psi(x)}.$$
        By the definition \eqref{definitions in Carleman estimate} of the function $\phi$, we can see there exists an $\varepsilon_1\in(0,1/2)$, such that
		\begin{equation}\label{4.84 in the book}
			\phi(t,x)\leq\frac{R_1^2}{2}-\frac{c_1T^2}{8}<0, \quad \forall ~ (t,x)\in\tilde{Q}.
		\end{equation}
		
		Further, since that
		$$\phi\Big(\frac{T}{2},x\Big)=\psi(x)\geq R_0^2, \quad\forall ~ x\in\Omega,$$
        one can find an $\varepsilon_0\in(0,1/2)$, such that
		\begin{equation}\label{4.85 in the book}
			\phi(t,x)\geq\frac{R_0^2}{2}, \quad \forall ~ (t,x)\in\big(\tilde{T}_0,\tilde{T}'_0\big)\times\Omega:=Q_0.
		\end{equation}
		
		Combining \eqref{4.137 in the book}--\eqref{4.85 in the book}, we obtain that
		$$
		\begin{aligned}
			& e^{\lambda R_0^2}\|z_t\|_{L^2(Q_0)}^2 \\
			\leq ~ & C\lambda\Big(e^{\lambda(R_1^2-cT^2/4)}\|z_t\|_{L^2(\tilde{Q})}^2+\varepsilon^2e^{2\lambda R_1^2}\|\nabla z\|_{L^2(Q^T)}^2+e^{2\lambda R_1^2}\int_0^T \int_\omega z_t^2\dd x\dd t\Big). \\
		\end{aligned}
		$$
        Noting that
		$$\|z_t\|_{L^2(\tilde{Q})}^2+\|\nabla z\|_{L^2(Q^T)}^2\leq 2T\sup\limits_{t\in[0,T]}\mathcal{E}(t)\leq 2Te^{(1+\varepsilon)T}\mathcal{E}(T),$$
        hence we have
		\begin{equation}\label{4.139 in the book}
			\|z_t\|_{L^2(Q_0)}^2\leq C\lambda\Big(e^{\lambda(R_1^2-R_0^2-cT^2/4)}\mathcal{E}(T)+e^{2\lambda R_1^2}\int_0^T\int_\omega z_t^2\dd x\dd t\Big).
		\end{equation}
		
		\textbf{Step 3.} We choose a nonnegative function $\zeta\in C^1([\tilde{T}_0,\tilde{T}'_0])$ with $\zeta(\tilde{T}_0)=\zeta(\tilde{T}'_0)=0$. Multiplying the equation in \eqref{linear hyperbolic observability system-A} by $\zeta z$, integrating it in $Q_0$ and using integration by parts, we get
		$$
		\begin{aligned}
			& \int_{Q_0}\zeta\Big(z_t^2+\sum\limits_{j,k=1}^na^{jk}z_{x_j}z_{x_k}+z^2\Big)\dd x\dd t=2\int_{\tilde{T}_0}^{\tilde{T}'_0}\zeta(t)\mathcal{E}(t)\dd t \\
			= & ~ 2\int_{Q_0}\zeta z_t^2\dd x\dd t+\int_{Q_0}\zeta_tzz_t\dd x\dd t-\int_{Q_0}\zeta z\Big(b_0z_t+\sum\limits_{k=1}^nb_kz_{x_k}+(\tilde{b}-1)z\Big)\dd x\dd t \\
			\leq & ~ C\int_{Q_0}z_t^2\dd x\dd t+C\varepsilon\int_{Q_0}\zeta|\nabla z|^2\dd x\dd t \\
			\leq & ~ C\int_{Q_0}z_t^2\dd x\dd t+C\varepsilon\int_{\tilde{T}_0}^{\tilde{T}'_0}\zeta(t)\mathcal{E}(t)\dd t.
		\end{aligned}
		$$
        Thus we obtain
		\begin{equation}\label{minE}
			\min\limits_{t\in[0,T]}\mathcal{E}(t)\leq C\int_{Q_0}z_t^2\dd x\dd t.
		\end{equation}
        Note that by \eqref{energy estimate half}, we also have
		$$\mathcal{E}'(t)+\int_\Omega b_0z_t^2\dd x=-\int_\Omega\Big(\sum\limits_{k=1}^nb_kz_tz_{x_k}+\tilde{b}zz_t\Big)\dd x\leq C\varepsilon\mathcal{E}(t),$$
        hence we get
		$$\frac{\dd}{\dd t}\big(e^{-C\varepsilon t}\mathcal{E}(t)\big)\leq -e^{-C\varepsilon t}\int_\Omega b_0z_t^2\dd x\leq 0.$$
        Then we have
		\begin{equation}\label{almost decrease of E}
			\mathcal{E}(t)\geq e^{-C\varepsilon t}\mathcal{E}(t)\geq e^{-C\varepsilon T}\mathcal{E}(T), \qquad \forall ~ t\in[0,T].
		\end{equation}
		
		Combining \eqref{minE} and \eqref{almost decrease of E}, we have
		\begin{equation}\label{4.140 in the book}
			\mathcal{E}(T)\leq C\int_{Q_0}z_t^2\dd x\dd t.
		\end{equation}
		
		It follows from \eqref{4.139 in the book} and \eqref{4.140 in the book} that
		\begin{equation}\label{4.141 in the book}
			\mathcal{E}(T)\leq C\lambda\Big(e^{\lambda(R_1^2-R_0^2-cT^2/4)}\mathcal{E}(T)+e^{2\lambda R_1^2}\int_0^T\int_\omega z_t^2\dd x\dd t\Big).
		\end{equation}
		
		Noting that $R_1^2-R_0^2-cT^2/4<0$, let $\lambda$ be large enough such that
		$$C\lambda e^{\lambda(R_1^2-R_0^2-cT^2/4)}\leq\frac{1}{2},$$
        then can deduce from \eqref{4.141 in the book} that
		\begin{equation}\label{observability but ET}
			\mathcal{E}(T)\leq C_1e^{C_1}\int_0^T\int_\omega z_t^2\dd x\dd t,
		\end{equation}
		where $C_1$ is a positive constant independent of initial data.
		
		Combining \eqref{observability but ET} and \eqref{E0 and ET}, we obtain
		$$\|z_1\|_{L^2(\Omega)}^2+\|z_0\|_{H^1(\Omega)}^2\leq C\mathcal{E}(0)\leq C_2e^{C_2}\int_0^T\int_\omega z_t^2\dd x\dd t,$$
		with a constant $C_2>0$ independent of initial data. Thus we obtain the desired inequality.
		\end{proof}
		
			\subsection{Carleman estimate in $L^2$-norm}

		This subsection is devoted to prove Proposition \ref{prop:L2 Carleman theorem}.
		
		Throughout this subsection, we fix the function $\phi$ in \eqref{definitions in Carleman estimate}, a parameter $\lambda>0$, and a function $z\in C([0,T];L^2(\Omega))$ holding $z(0,x)=z(T,x)=0$ for $x\in\Omega$. For any $K>1$, we choose a function $\rho(x)\in C^2(\overline{\Omega})$ with $\min\limits_{x\in\overline{\Omega}}\rho(x)=1$ so that
		\begin{equation}
			\rho(x)=
			\left\{
			\begin{aligned}
				1, & ~ x\in \omega, \\
				K, & ~ d(x,\omega)\geq \frac{1}{\ln K},
			\end{aligned}
			\right.
		\end{equation}
		
		For any integer $m\geq 3$, let $h=\frac{T}{m}$. Define
		\begin{equation}
			z_m^i=z_m^i(x)=z(ih,x), ~ \phi_m^i=\phi_m^i(x)=\phi(ih,x), \quad i=0,1,\cdots,m.
		\end{equation}
		and
		\begin{equation}
			a^{jk}_i=a^{jk}_i(x)=a^{jk}(ih,x),\quad i=0,1,\cdots,m; ~ j,k=1,\cdots,n.
		\end{equation}
		
		Let $\{(w_m^i,r_{1m}^i,r_{2m}^i),r_m^i\}_{i=0}^m\in (H_0^1(\Omega)\times (L^2(\Omega))^3)^{m+1}$ satisfy the following system:
		\begin{equation}\label{system:one}
			\left\{
			\begin{aligned}
				& \frac{w_m^{i+1}-2w_m^i+w_m^{i-1}}{h^2}-\sum_{j_1,j_2=1}^n\partial_{j_2}(a^{j_1,j_2}_i\partial_{j_1}w_m^i)\\
				= & \frac{r_{1m}^{i+1}-r^i_{1m}}{h}+r_{2m}^i+\lambda z_m^ie^{2\lambda\phi_m^i}+r_m^i, \quad 1\leq i\leq m-1, ~ x\in\Omega, \\
				& w_m^i=0, \qquad 0\leq i\leq m, ~ x\in\partial\Omega, \\
				& w_m^0=w_m^m=r_{2m}^0=r_{2m}^m=r_m^0=r_m^m=0, ~ r_{1m}^0=r_{1m}^1, \quad x\in\Omega.
			\end{aligned}
			\right.
		\end{equation}
		The set of admissible sequences for \eqref{system:one} is defined as
		\begin{equation}
			\begin{aligned}
				\mathcal{A}_{ad}:=\Big\{ & \{(w_m^i,r_{1m}^i,r_{2m}^i),r_m^i\}_{i=0}^m\in (H_0^1(\Omega)\times (L^2(\Omega))^3)^{m+1}\Big| \\
				& \{(w_m^i,r_{1m}^i,r_{2m}^i),r_m^i\}_{i=0}^m ~ \text{satisfy} ~ \eqref{system:one}\Big\}.
			\end{aligned}
		\end{equation}
		Note that we can easily see the set $\mathcal{A}_{ad}\not=\emptyset$ because $\big\{(0,0,0,-\lambda z_m^ie^{2\lambda\phi_m^i})\big\}_{i=0}^m\in\mathcal{A}_{ad}$.
		
		Now, let us introduce the cost functional
		\begin{equation}\begin{aligned}
				& J\Big(\{(w_m^i,r_{1m}^i,r_{2m}^i),r_m^i\}_{i=0}^m\Big)=\frac{h}{2}\int_\Omega\rho\frac{|r_{1m}^m|}{\lambda^2}e^{-2\lambda\phi_m^m}\dd x \\
				+ & \frac{h}{2}\sum_{i=1}^{m-1}\left[\int_\Omega|w_m^i|^2e^{-2\lambda\phi_m^i}\dd x+\int_\Omega\rho\left(\frac{|r_{1m}^i|^2}{\lambda^2}+\frac{|r_{2m}^i|^2}{\lambda^4}\right)e^{-2 \lambda\phi_m^i}+K\int_\Omega|r_m^i|^2\dd x\right].
			\end{aligned}
		\end{equation}
		
		Let us consider the following optimal problem:
		\begin{equation}\label{problem}
			\inf_{\{(w_m^i,r_{1m}^i,r_{2m}^i),r_m^i\}_{i=0}^m\in \mathcal{A}_{ad}}J(\{(w_m^i,r_{1m}^i,r_{2m}^i),r_m^i\}_{i=0}^m)=d.
		\end{equation}
		
		We have the following key proposition.
		\begin{prop}\label{auxiliary OP}
			For any $K>1$ and $m\geq 3$, problem \eqref{problem} admits a unique solution $\{(\hat{w}_m^i,\hat{r}_{1m}^i,\hat{r}_{2m}^i),\hat{r}_m^i\}_{i=0}^m\in \mathcal{A}_{ad}$, such that
			$$J\Big(\{(\hat{w}_m^i,\hat{r}_{1m}^i,\hat{r}_{2m}^i),\hat{r}_m^i\}_{i=0}^m\Big)=\min_{\{(w_m^i,r_{1m}^i,r_{2m}^i),r_m^i\}_{i=0}^m\in\mathcal{A}_{ad}}J\Big(\{(w_m^i,r_{1m}^i, r_{2m}^i),r_m^i\}_{i=0}^m\Big).$$
			
			Furthermore, for
			\begin{equation}
				p_m^i=p_m^i(x):=K\hat{r}_m^i(x), \qquad 0\leq i\leq m,
			\end{equation}
			one has
			\begin{equation}
				\begin{aligned}
					& \hat{w}_m^0=\hat{w}_m^m=p_m^0=p_m^m=0, \qquad x\in\Omega, \\
					& \hat{w}_m^i,p_m^i\in H^2(\Omega)\cap H_0^1(\Omega), \quad 1\leq i\leq m-1
				\end{aligned}
			\end{equation}
			and the following optimality conditions:
			\begin{equation}\label{p:one}
				\left\{
				\begin{aligned}
					& \frac{p_m^i-p_m^{i-1}}{h}+\rho\frac{\hat{r}_{1m}^i}{\lambda^2}e^{-2\lambda\phi_m^i}=0, \\
					& p_m^i-\rho\frac{\hat{r}_{2m}^i}{\lambda^4}e^{-2\lambda\phi_m^i}=0,
				\end{aligned}
				\right.
				\qquad 1\leq i\leq m, ~ x\in\Omega
			\end{equation}
				and
			\begin{equation}\label{p:two}
				\left\{
				\begin{aligned}
					& \frac{p_m^i-2p_m^{i-1}+p_m^{i-1}}{h^2}-\sum_{j_1,j_2=1}^n\partial_{j_2}(a^{j_1,j_2}_i\partial_{j_1}p_m^i) \\
					& +\hat{z}_m^i e^{-2\lambda\phi_m^i}=0, ~ x\in\Omega \\
					& p_m^i=0, \qquad x\in\partial\Omega.
				\end{aligned}
				\right.
				\quad 1\leq i\leq m-1.
			\end{equation}
			
			Moreover, there is a constant $C=C(K,\lambda)>0$, independent of $m$, such that
			\begin{equation}\label{estimation:one}
				h\sum_{i=1}^{m-1}\int_\Omega\Big[|\hat{w}_m^i|^2+|\hat{r}_{1m}^i|^2+|\hat{r}_{2m}^i|^2+K|\hat{r}_m^i|^2\Big]\dd x+h\int_\Omega|\hat{r}_{1m}^m|^2\leq C
			\end{equation}
            and
			\begin{equation}\label{estimation:two}
				\sum_{i=1}^{m-1}\int_\Omega\Big[\frac{(\hat{w}_m^{i+1}-\hat{w}_m^{i})^2}{h^2}+\frac{(\hat{r}_{1m}^{i+1}-\hat{r}_{1m}^{i})^2}{h^2}+\frac{(\hat{r}_{2m}^{i+1}-\hat{r}_{2m}^{i})^2} {h^2}+K\frac{(\hat{r}_{m}^{i+1}-\hat{r}_{m}^{i})^2}{h^2}\Big]\dd x\leq\frac{C}{h}.
			\end{equation}
		\end{prop}
		\begin{rem}
			For any $\big\{(w_m^i,r_{1m}^i,r_{2m}^i),r_m^i\big\}_{i=0}^m\in\mathcal{A}_{ad}$, since $(a^{j_1,j_2}_i)$ is positive definite, by standard regularity results of elliptic equations, we obtain $w_m^i\in H^2(\Omega)\cap H_0^1(\Omega)$.
		\end{rem}
		\begin{proof}
			The proof is divided into several steps.
			
			\textbf{Step 1. Existence and uniqueness of $\{(\hat{w}_m^i,\hat{r}_{1m}^i,\hat{r}_{2m}^i),\hat{r}_m^i\}_{i=0}^m\in \mathcal{A}_{ad}$.}
			
			Let $\{\{(w_m^{i,j},r_{1m}^{i,j},r_{2m}^{i,j}),r_m^{i,j}\}_{i=0}^m\}_{j=1}^\infty\subset\mathcal{A}_{ad}$ be a minimizing sequence of $J$. Due to the coercivity of $J$ and noting that $w_m^{i,j}$ solves an elliptic equation, it can be shown that
			$$\{\{(w_m^{i,j},r_{1m}^{i,j},r_{2m}^{i,j}),r_m^{i,j}\}_{i=0}^m\}_{j=1}^\infty$$
			is bounded in $\mathcal{A}_{ad}$. Therefore, there exists a subsequence of $\{\{(w_m^{i,j},r_{1m}^{i,j},r_{2m}^{i,j}),r_m^{i,j}\}_{i=0}^m\}_{j=1}^\infty$ converging weakly to some
			$$\{(\hat{w}_m^i,\hat{r}_{1m}^i,\hat{r}_{2m}^i),\hat{r}_m^i\}_{i=0}^m\in(H_0^1(\Omega)\times (L^2(\Omega))^3)^{m+1}.$$
			
			Note that the constraint condition \eqref{system:one} is a linear system, we obtain
			$$\{(\hat{w}_m^i,\hat{r}_{1m}^i,\hat{r}_{2m}^i),\hat{r}_m^i\}_{i=0}^m\in \mathcal{A}_{ad}.$$
			and $\hat{w}_m^0=\hat{w}_m^m=p_m^0=p_m^m=0, ~ x\in\Omega$.
			
			Since $J$ is strictly convex, this optimal target is the unique solution of \eqref{problem}.
			
			\textbf{Step 2. The proof of \eqref{p:one} and \eqref{p:two}.}
			
			Fix any
			$$\delta_{0m}^i\in H^2\cap H_0^1, ~ \delta_{1m}^i\in L^2, ~ \delta_{2m}^i\in L^2, \quad i=0,1,\cdots,m$$
			with $\delta_{0m}^0=\delta_{0m}^m=\delta_{2m}^0=\delta_{2m}^m=0$ and $\delta_{1m}^0=\delta_{1m}^1$ in $\Omega$. For $(\lambda_0,\lambda_1,\lambda_2)\in \mathbb{R}^3$, we denote
			\begin{equation}\label{system:two}
			\left\{
			\begin{aligned}
			& r_m^i:=\frac{\hat{w}_m^{i+1}-2\hat{w}_m^i+\hat{w}_m^{i-1}}{h^2}+\frac{\delta_{0m}^{i+1}-2\delta_{0m}^i+\delta_{0m}^{i-1}}{h^2}\lambda_0 \\
			& -\sum_{j_1,j_2=1}^n\partial_{j_2}\big(a^{j_1,j_2}_i\partial_{j_1}(\hat{w}_m^i+\lambda_0\delta_{0m}^i)\big)-\frac{\hat{r}_{1m}^{i+1}-\hat{r}^i_{1m}}{h} \\
			&-\frac{\delta_{1m}^{i+1}-\delta^i_{1m}}{h}\lambda_1-\hat{r}_{2m}^i-\lambda_2\delta_{2m}^i-\lambda z_m^ie^{2\lambda\phi_m^i}, ~ 1\leq i\leq m-1, \\
			& r_m^0=r_m^m=0
			\end{aligned}
			\right.
			\end{equation}
			
			Then we have
			$$\big\{(\hat{w}_m^i+\lambda_0\delta_{0m}^i,\hat{r}_{1m}^i+\lambda_1\delta_{1m}^i,\hat{r}_{2m}^i)+\lambda_2\delta_{2m}^i,r_m^i\big\}_{i=0}^m\in\mathcal{A}_{ad}.$$
			
			Define a function $g$ in $\mathbb{R}^3$ by
			\begin{equation}
			g(\lambda_0,\lambda_1,\lambda_2)=J\Big(\{(\hat{w}_m^i+\lambda_0\delta_{0m}^i,\hat{r}_{1m}^i+\lambda_1\delta_{1m}^i,\hat{r}_{2m}^i)+\lambda_2\delta_{2m}^i,r_m^i\}_{i=0}^m\Big).
			\end{equation}
			Since $\{(\hat{w}_m^i,\hat{r}_{1m}^i,\hat{r}_{2m}^i),\hat{r}_m^i\}_{i=0}^m\in \mathcal{A}_{ad}$ is minimum point of $J$, $g$ has a minimum at $(0,0,0)$. Hence we have $\nabla g(0,0,0)=0$.
			
			By $\frac{\partial g(0,0,0)}{\partial\lambda_1}=\frac{\partial g(0,0,0)}{\partial\lambda_2}=0$, and the fact that $\{(\hat{w}_m^i,\hat{r}_{1m}^i,\hat{r}_{2m}^i),\hat{r}_m^i \}_{i=0}^m\in\mathcal{A}_{ad}$ satisfy \eqref{system:one}, one gets
			\begin{equation}
			-K\sum_{i=1}^{m-1}\int_\Omega\hat{r}_m^i\frac{\delta_{1m}^{i+1}-\delta_{1m}^i}{h}\dd x+\sum_{i=1}^m\int_\Omega\rho\frac{\hat{r}_{1m}^i\delta_{1m}^i}{\lambda^2}e^{-2\lambda\phi_m^i}\dd x=0,
			\end{equation}
			\begin{equation}
			-K\sum_{i=1}^{m-1}\int_\Omega \hat{r}_m^i\delta_{2m}^i\dd x+\sum_{i=1}^{m-1}\int_\Omega\rho\frac{\hat{r}_{2m}^i\delta_{2m}^i}{\lambda^4}e^{-2\lambda\phi_m^i}\dd x=0,
			\end{equation}
            combined with \eqref{system:one} and $p_m^0=p_m^m=\hat{r}_{2m}^m=0$ in $\Omega$ gives \eqref{p:one}. From $\frac{g(0,0,0)}{\partial \lambda_0}=0$, one obtains
			\begin{equation}
			\begin{aligned}
			\sum_{i=1}^{m-1}\int_\Omega\bigg\{K\hat{r}_m^i\Big[\frac{\delta_{0m}^{i+1}-2\delta_{0m}^i+\delta_{0m}^{i-1}}{h^2}- & \sum_{j_1,j_2=1}^n\partial_{j_2}\big(a^{j_1,j_2}_i 
			+ & \hat{w}_m^i\delta_{0m}^ie^{-2\lambda\phi_m^i}\bigg\}\dd x=0
			\end{aligned}
			\end{equation}
			together with $p_m^0=p_m^m=\delta_{0m}^0=\delta_{0m}^m=0$ in $\Omega$, implies that $p_m^i=K\hat{r}_m^i$ is a weak solution of \eqref{p:two}. By the regularity theory for elliptic equations, one sees that $\hat{w}^i_m,p_m^i\in H^2\cap H_0^1$ for $1\leq i \leq m-1$.
			
			\textbf{Step 3. The proof of \eqref{estimation:one} and \eqref{estimation:two}.}
			
			The proof of the above estimates are similar to those of \cite{Fu3}, so we omit the details, then we complete the proof.
		\end{proof}
	
		Now we are in a position to prove  Propostion \ref{prop:L2 Carleman theorem}.
		\begin{proof}[Proof of Propostion \ref{prop:L2 Carleman theorem}]
			The main idea is to choose a special $\eta$, so that
			$$\eta_{tt}-\sum\limits_{j,k=1}^n\big(a^{jk}\eta_{x_j}\big)_{x_k}=\lambda ze^{2\lambda\phi}+\cdots,$$
            where we get the desired term $\lambda\|\theta z\|_{L^2(Q^T)}^2$ and reduce the estimate to that for $\|\eta\|_{H_0^1(Q^T)}$. The proof is divided into several steps.
			
			\textbf{Step 1.} Firstly, recall the functions $\{(\hat{w}_m^i,\hat{r}_{1m}^i,\hat{r}_{2m}^i),\hat{r}_m^i\}_{i=0}^m$ in Proposition \ref{auxiliary OP}, put
			$$
			\left\{
			\begin{aligned}
				\tilde{w}^m(t,x) & =\frac{1}{h}\sum\limits_{i=0}^{m-1}\Big((t-ih)\hat{w}^{i+1}_m(x)-(t-(i+1)h)\hat{w}^i_m(x)\Big)\chi_{(ih,(i+1)h]}(t), \\
				\tilde{r}_1^m(t,x) & =\hat{r}^0_{1m}(x)\chi_{\{0\}}(t) \\
				&  +\frac{1}{h}\sum\limits_{i=0}^{m-1}\Big((t-ih)\hat{r}^{i+1}_{1m}(x)-(t-(i+1)h)\hat{r}^i_{1m}(x)\Big)\chi_{(ih,(i+1)h]}(t), \\
				\tilde{r}_2^m(t,x) & =\frac{1}{h}\sum\limits_{i=0}^{m-1}\Big((t-ih)\hat{r}^{i+1}_{2m}(x)-(t-(i+1)h)\hat{r}^i_{2m}(x)\Big)\chi_{(ih,(i+1)h]}(t), \\
				\tilde{r}^m(t,x) & =\frac{1}{h}\sum\limits_{i=0}^{m-1}\Big((t-ih)\hat{r}^{i+1}_m(x)-(t-(i+1)h)\hat{r}^i_m(x)\Big)\chi_{(ih,(i+1)h]}(t),
			\end{aligned}
			\right.
			$$
			
			By \eqref{estimation:one} and \eqref{estimation:two}, there exist a subsequence of $\{(\tilde{w}^m,\tilde{r}_1^m,\tilde{r}_2^m),\tilde{r}^m\}_{m=1}^\infty$ which converges weakly to some $(\tilde{w},\tilde{r}_1,\tilde{r}_2),\tilde{r}\in H^1(0,T; L^2(\Omega))$ as $m\to\infty$.
			
			Let $\tilde{p}=K\tilde{r}$ for some sufficiently large constant $K>1$. By \eqref{system:one}, \eqref{p:one}--\eqref{estimation:two} and Lemma \ref{lemma4.1 in the book}, we obtain that
			\begin{align}\label{regularity-of-w:p}
				\tilde{w}, ~ \tilde{p}\in C([0,T];H^1_0(\Omega))\cap C^1([0,T];L^2(\Omega)),
			\end{align}
			and
			\begin{equation}\label{system:w and p}
				\left\{
				\begin{aligned}
					& \tilde{w}_{tt}-\sum\limits_{j,k=1}^n\big(a^{jk}\tilde{w}_{x_j}\big)_{x_k}=\partial_t\tilde{r}_1+\tilde{r}_2+\lambda\theta^2z+\tilde{r}, & (t,x) ~ \in ~ & Q^T, \\
					& \tilde{p}_{tt}-\sum\limits_{j,k=1}^n\big(a^{jk}\tilde{p}_{x_j}\big)_{x_k}+\theta^{-2}\tilde{w}=0, &  (t,x) ~ \in ~ & Q^T, \\
					& \tilde{p}_t+\rho\theta^{-2}\frac{\tilde{r}_1}{\lambda^2}=0, &  (t,x) ~ \in ~ & Q^T, \\
					& \tilde{p}-\rho\theta^{-2}\frac{\tilde{r}_2}{\lambda^4}=0, &  (t,x) ~ \in ~ & Q^T, \\
					& \tilde{p}(t,x)=\tilde{w}(t,x)=0, &  (t,x) ~ \in ~ & \Gamma^T, \\
					& \tilde{p}(0,x)=\tilde{p}(T,x)=\tilde{w}(0,x)=\tilde{w}(T,x)=0, & x ~ \in ~ & \Omega.
				\end{aligned}
				\right.
			\end{equation}
			
			\textbf{Step 2.} Applying Theorem \ref{internal Carleman theorem} to $\tilde{p}$ in \eqref{system:w and p}, we have
			\begin{equation}\label{estimate:p1}
				\begin{aligned}
					& \lambda\int_{Q^T}\theta^2(\lambda^2\tilde{p}^2+\tilde{p}_t^2+|\nabla\tilde{p}|^2)\dd x\dd t \\
					\leq ~ & C\Big[\int_{Q^T}\theta^{-2}\tilde{w}^2\dd x\dd t+\lambda^2\int_0^T\int_\omega\theta^2(\lambda^2\tilde{p}^2+\tilde{p}_t^2)\dd x\dd t\Big] \\
					\leq ~ & C\Big[\int_{Q^T}\theta^{-2}\tilde{w}^2\dd x\dd t+\int_0^T\int_\omega\theta^{-2}\Big(\frac{\tilde{r}_1^2}{\lambda^2}+\frac{\tilde{r}_2^2}{\lambda^4}\Big)\dd x\dd t\Big].
				\end{aligned}
			\end{equation}
				Here and hence forth, $C$ is a constant independent of $K$ and $\lambda$.
			
			By \eqref{system:w and p}, we have
			\begin{equation}\label{system:pt}
				\left\{
				\begin{aligned}
					& \tilde{p}_{ttt}-\sum\limits_{j,k=1}^n\big(a^{jk}\tilde{p}_{tx_j}\big)_{x_k}+(\theta^{-2}\tilde{w})_t-\sum\limits_{j,k=1}^n\big(a^{jk}_t\tilde{p}_{x_j}\big)_{x_k}=0, &  (t,x) ~ \in ~ & Q^T, \\
					& \tilde{p}_{tt}+\frac{\rho}{\lambda}\theta^{-2}\Big(\frac{\partial_t\tilde{r}_1}{\lambda}-2\phi_t\tilde{r}_1\Big)=0, &  (t,x) ~ \in ~ & Q^T, \\
					& \tilde{p}_t-\frac{\rho}{\lambda^2}\theta^{-2}\Big(\frac{\partial_t\tilde{r}_2}{\lambda^2}-\frac{2}{\lambda}\phi_t\tilde{r}_1\Big)=0, &  (t,x) ~ \in ~ & Q^T, \\
					& \tilde{p}_t(t,x)=0, &  (t,x) ~ \in ~ & \Gamma^T,
				\end{aligned}
				\right.
			\end{equation}
			and
			\begin{equation}\label{system:tilde:p}
				\left\{
				\begin{aligned}
					& \tilde{p}_{tt}-\Delta\tilde{p}=\sum\limits_{j,k=1}^n\big((a^{jk}-\delta_{jk})\tilde{p}_{x_j}\big)_{x_k}-\theta^{-2}\tilde{w}, & (t,x)\in & ~ Q^T, \\
					& \tilde{p}(t,x)=0, & (t,x) \in & ~ \Gamma^T, \\
					& \tilde{p}(0,x)=\tilde{p}(T,x)=0, & x\in & ~ \Omega.
				\end{aligned}
				\right.
			\end{equation}
			
			Applying Theorem \ref{internal Carleman theorem} to $\tilde{p}_t$ in \eqref{system:pt}, we obtain
			\begin{align}\label{estimate:p2}
				& \lambda\int_{Q^T}\theta^2(\lambda^2 \tilde{p}_t^2+\tilde{p}_{tt}^2+|\nabla\tilde{p}_t|^2)\dd x\dd t \nonumber\\
                \leq ~ & C\bigg[\big\|\theta(\theta^{-2}\tilde{w})_t\big\|_{L^2(Q^T)}^2+\Big\|\theta\sum\limits_{j,k=1}^n\big(a^{jk}_t\tilde{p}_{x_j}\big)_{x_k}\Big\|_{L^2(Q^T)}^2+ \lambda^2\int_0^T\int_\omega\theta^2(\lambda^2\tilde{p}_t^2+\tilde{p}_{tt}^2)\dd x\dd t\bigg] \nonumber \\
				\leq ~ & C\bigg[\int_{Q^T}\theta^{-2}(\tilde{w}_t^2+\lambda^2\tilde{w}^2)\dd x\dd t+\Big\|\theta\sum\limits_{j,k=1}^n\big(a^{jk}_t\tilde{p}_{x_j}\big)_{x_k}\Big\|_{L^2 (Q^T)}^2 \\
				&+\int_0^T\int_\omega\theta^{-2}\Big( \frac{|\partial_t\tilde{r}_1|^2}{\lambda^2} +\frac{|\partial_t\tilde{r}_2|^2}{\lambda^4}+\tilde{r}_1^2+\frac{\tilde{r}_2^2}{\lambda^2}\Big) \dd x\dd t\bigg]. \nonumber
				\end{align}
			
			{\color{black} Here we note that, in view of \eqref{regularity-of-w:p} and \eqref{system:tilde:p}, we have $\tilde{p}_t\in H^1(Q^T)$, hence we can apply Theorem \ref{internal Carleman theorem} to $\tilde{p}_t$.}
			
			Recalling the smallness assumption \eqref{coecond:normA} on $a^{ij}$, we have
			\begin{equation}\label{Addition:1}
				\int_{Q^T}\theta^2\Big|\sum\limits_{j,k=1}^n\big(a^{jk}_t\tilde{p}_{x_j}\big)_{x_k}\Big|^2\dd x\dd t\leq C\varepsilon\int_{Q^T}\theta^2\big(|\nabla\tilde{p}|^2+|\nabla^2 \tilde{p}|^2\big)\dd x\dd t.
			\end{equation}

			Taking $L^2$ inner product of \eqref{system:tilde:p} with $-\Delta\tilde{p}$, we get
			$$
			\begin{aligned}
				& \int_{Q^T}|\nabla^2\tilde{p}|^2\dd x\dd t-\int_{Q^T}|\nabla\tilde{p}_t|^2\dd x\dd t \\
				\leq C\varepsilon & \int_{Q^T}\big(|\nabla\tilde{p}|^2+|\nabla^2\tilde{p}|^2\big)\dd x\dd t+\int_{Q^T}\theta^{-2}|\tilde{w}||\Delta\tilde{p}|\dd x\dd t \\
				\leq C\varepsilon & \int_{Q^T}\big(|\nabla\tilde{p}|^2+|\nabla^2\tilde{p}|^2\big)\dd x\dd t+\frac{1}{2}\int_{Q^T}\big(\theta^{-4}\tilde{w}^2+|\Delta\tilde{p}|^2\big)\dd x\dd t.
			\end{aligned}
			$$
			
			Noting that $e^{C_1\lambda}\leq\theta\leq e^{C_2\lambda}$ for some $C_1<C_2$, we obtain
			\begin{equation}\label{Addition:2}
				\begin{aligned}
					\int_{Q^T}\theta^2|\nabla^2\tilde{p}|^2\dd x\dd t\leq ~ & e^{C\lambda}\int_{Q^T}|\nabla^2\tilde{p}|^2\dd x\dd t \\
					\leq ~ & Ce^{C\lambda}\int_{Q^T}\big(\varepsilon|\nabla\tilde{p}|^2+|\nabla\tilde{p}_t|^2+\theta^{-4}\tilde{w}^2\big)\dd x\dd t \\
					\leq ~ & Ce^{C\lambda}\int_{Q^T}\Big(\varepsilon\theta^2|\nabla\tilde{p}|^2+\theta^2|\nabla\tilde{p}_t|^2+\theta^{-2}\tilde{w}^2\Big)\dd x\dd t.
				\end{aligned}
			\end{equation}
			
			By \eqref{Addition:1} and \eqref{Addition:2}, we obtain
			\begin{equation}\label{Addition:3}
				\int_{Q^T}\theta^2\Big|\sum\limits_{j,k=1}^n\big(a^{jk}_t\tilde{p}_{x_j}\big)_{x_k}\Big|^2\dd x\dd t\leq C\varepsilon e^{C\lambda}\int_{Q^T}\Big[\theta^2\big(|\nabla \tilde{p}|^2+|\nabla\tilde{p}_t|^2\big)+\theta^{-2}\tilde{w}^2\Big]\dd x\dd t.
			\end{equation}
			
			\textbf{Step 3.} Noting that by \eqref{system:w and p},
			$$-\int_{Q^T}(\partial_t\tilde{r}_1+\tilde{r}_2)\tilde{p}\dd x\dd t=\int_{Q^T}(\tilde{r}_1\tilde{p}_t-\tilde{r}_2\tilde{p})\dd x\dd t=-\int_{Q^T}\rho\theta^{-2}\Big( \frac{\tilde{r}_1^2}{\lambda^2}+\frac{\tilde{r}_2^2}{\lambda^4}\Big)\dd x\dd t.$$
			
			Thus we have
			$$
			\begin{aligned}
				0 ~ & =\Big(\tilde{w}_{tt}-\sum\limits_{j,k=1}^n\big(a^{jk}\tilde{w}_{x_j}\big)_{x_k}-\partial_t\tilde{r}_1-\tilde{r}_2-\lambda\theta^2z-\tilde{r}, ~ \tilde{p}\Big)_{L^2(Q^T)} \\
				& = ~ -\int_{Q^T}\theta^{-2}\tilde{w}^2\dd x\dd t-\int_{Q^T}\rho\theta^{-2}\Big(\frac{\tilde{r}_1^2}{\lambda^2}+\frac{\tilde{r}_2^2}{\lambda^4}\Big)\dd x\dd t-\lambda\int_{Q^T} \theta^2z\tilde{p}\dd x\dd t-K\int_{Q^T}\tilde{r}^2\dd x\dd t.
			\end{aligned}
			$$
			
			Hence we get
			$$\int_{Q^T}\theta^{-2}\tilde{w}^2\dd x\dd t+\int_{Q^T}\rho\theta^{-2}\Big(\frac{\tilde{r}_1^2}{\lambda^2}+\frac{\tilde{r}_2^2}{\lambda^4}\Big)\dd x\dd t+K\int_{Q^T}\tilde{r}^2 \dd x\dd t=-\lambda\int_{Q^T}\theta^2z\tilde{p}\dd x\dd t.$$
			
			By Cauchy-Schwartz inequality and \eqref{estimate:p1}, we obtain
			\begin{equation}\label{4.62 of the book}
				\int_{Q^T}\theta^{-2}\tilde{w}^2\dd x\dd t+\int_{Q^T}\rho\theta^{-2}\Big(\frac{\tilde{r}_1^2}{\lambda^2}+\frac{\tilde{r}_2^2}{\lambda^4}\Big)\dd x\dd t+K\int_{Q^T}\tilde{r}^2 \dd x\dd t\leq\frac{C}{\lambda}\int_{Q^T}\theta^2z^2\dd x\dd t.
			\end{equation}
			
			\textbf{Step 4.} Using \eqref{system:w and p} and \eqref{system:pt}, by the fact that $\tilde{p}_{tt}(0)=\tilde{p}_{tt}(T)=0$ in $\Omega$, we get
			\begin{align}\label{4.63 of the book}
				0=~ & \Big(\tilde{w}_{tt}-\sum\limits_{j,k=1}^n\big(a^{jk}\tilde{w}_{x_j}\big)_{x_k}-\partial_t\tilde{r}_1-\tilde{r}_2-\lambda\theta^2z-\tilde{r}, ~ \tilde{p}_{tt}\Big)_{L^2 (Q^T)} \nonumber \\
				=~& \Big(\tilde{w} ~ ,\tilde{p}_{tttt}-\sum\limits_{j,k=1}^n\big(a^{jk}\tilde{p}_{ttx_j}\big)_{x_k}\Big)_{L^2(Q^T)} \nonumber \\
				& -\int_{Q^T}(\partial_t\tilde{r}_1+\tilde{r}_2)\tilde{p}_{tt}\dd x\dd t-\lambda\int_{Q^T}\theta^2z\tilde{p}_{tt}\dd x\dd t-\int_{Q^T}\tilde{r}\tilde{p}_{tt}\dd x\dd t \\
				=~& -\int_{Q^T}\tilde{w}(\theta^{-2}\tilde{w})_{tt}\dd x\dd t+\sum\limits_{j,k=1}^n\int_{Q^T}\tilde{w}\big(2a^{jk}_t\tilde{p}_{tx_j}+a^{jk}_{tt}\tilde{p}_{x_j} \big)_{x_k}\dd x\dd t \nonumber \\
				& -\int_{Q^T}(\partial_t\tilde{r}_1+\tilde{r}_2)\tilde{p}_{tt}\dd x\dd t-\lambda\int_{Q^T}\theta^2z\tilde{p}_{tt}\dd x\dd t-\int_{Q^T}\tilde{r}\tilde{p}_{tt}\dd x\dd t. \nonumber
			\end{align}
			
			Now we should deal with the terms on the right hand side.
			
			Firstly, it's easy to see that
			\begin{equation}\label{4.64 of the book}
				\begin{aligned}
					-\int_{Q^T}\tilde{w}(\theta^{-2}\tilde{w})_{tt}\dd x\dd t & =\int_{Q^T}\Big[\theta^{-2}\tilde{w}_t^2-(\theta^{-2})_{tt}\frac{\tilde{w}^2}{2}\Big]\dd x\dd t \\
					& =\int_{Q^T}\theta^{-2}\big(\tilde{w}_t^2+\lambda\phi_{tt}\tilde{w}^2-2\lambda^2\phi_t^2\tilde{w}^2\big)\dd x\dd t.
				\end{aligned}
			\end{equation}
			
			Secondly, by \eqref{system:pt} we have
			\begin{equation}\label{4.65 of the book}
				\begin{aligned}
					& -\int_{Q^T}(\partial_t\tilde{r}_1+\tilde{r}_2)\tilde{p}_{tt}\dd x\dd t=\int_{Q^T}(\tilde{p}_t\partial_t\tilde{r}_2-\tilde{p}_{tt}\partial_t\tilde{r}_1)\dd x\dd t \\
					=~ & \int_{Q^T}\rho\theta^{-2}\bigg[\frac{\partial_t\tilde{r}_1}{\lambda}\Big(\frac{\partial_t\tilde{r}_1}{\lambda}-2\phi_t\tilde{r}_1\Big)+\frac{\partial_t\tilde{r}_2} {\lambda^2}\Big(\frac{\partial_t\tilde{r}_2}{\lambda^2}-\frac{2}{\lambda}\phi_t\tilde{r}_2\Big)\bigg]\dd x\dd t \\
					=~ & \int_{Q^T}\rho\theta^{-2}\Big(\frac{|\partial_t\tilde{r}_1|^2}{\lambda^2}+\frac{|\partial_t\tilde{r}_2|^2}{\lambda^4}-\frac{2}{\lambda}\phi_t\tilde{r}_1\partial_t \tilde{r}_1-\frac{2}{\lambda^3}\phi_t\tilde{r}_2\partial_t\tilde{r}_2\Big)\dd x\dd t.
				\end{aligned}
			\end{equation}
			
			Moreover, by $\tilde{p}=K\tilde{r}$ and integration by parts, one gets that
			\begin{equation}\label{4.66 of the book}
				-\int_{Q^T}\tilde{r}\tilde{p}_{tt}\dd x\dd t=K\int_{Q^T}\tilde{r}_t^2\dd x\dd t
			\end{equation}
				and
			\begin{align}\label{Addition:4}
				& \sum\limits_{j,k=1}^n\int_{Q^T} \tilde{w}\big(2a^{jk}_t\tilde{p}_{tx_j}+a^{jk}_{tt}\tilde{p}_{x_j}\big)_{x_k}\dd x\dd t \nonumber \\
					= & -\sum\limits_{j,k=1}^n\int_{Q^T}\tilde{w}_{x_k} \big(2a^{jk}_t\tilde{p}_{tx_j}+a^{jk}_{tt}\tilde{p}_{x_j}\big)\dd x\dd t \\
					\leq & ~ C\varepsilon\int_{Q^T}\Big[\theta^2(|\nabla\tilde{p}_t|^2+|\nabla\tilde{p}|^2)+\theta^{-2}|\nabla\tilde{w}|^2\Big]\dd x\dd t. \nonumber
				\end{align}
			
			Combining \eqref{4.63 of the book}--\eqref{Addition:4}, we end up with
			\begin{equation}\label{4.67 of the book}
				\begin{aligned}
					& \int_{Q^T}\rho\theta^{-2}\Big(\frac{|\partial_t\tilde{r}_1|^2}{\lambda^2}+\frac{|\partial_t\tilde{r}_2|^2}{\lambda^4}-\frac{2}{\lambda}\phi_t\tilde{r}_1\partial_t \tilde{r}_1-\frac{2}{\lambda^3}\phi_t\tilde{r}_2\partial_t\tilde{r}_2\Big)\dd x\dd t \\
					&+ K\int_{Q^T}\tilde{r}_t^2\dd x\dd t+\int_{Q^T}\theta^{-2}\big(\tilde{w}_t^2+\lambda\phi_{tt}\tilde{w}^2-2\lambda^2\phi_t^2\tilde{w}^2\big)\dd x\dd t \\
					\leq & ~ \lambda\int_{Q^T}\theta^2z\tilde{p}_{tt}\dd x\dd t+C\varepsilon\int_{Q^T}\Big[\theta^2(|\nabla\tilde{p}_t|^2+|\nabla\tilde{p}|^2)+\theta^{-2}|\nabla\tilde{w}|^2 \Big]\dd x\dd t.
				\end{aligned}
			\end{equation}
			
			By \eqref{4.67 of the book}+ $C\lambda^2\cdot$\eqref{4.62 of the book} with a sufficiently large $C>0$, using Cauchy-Schwartz inequality, noting \eqref{estimate:p1}, \eqref{estimate:p2} and \eqref{Addition:3}, we obtain that
			\begin{equation}\label{4.68 of the book}
				\begin{aligned}
					& \int_{Q^T}\theta^{-2}(\tilde{w}_t^2+\lambda^2\tilde{w}^2)\dd x\dd t+\int_{Q^T}\rho\theta^{-2}\Big(\frac{|\partial_t\tilde{r}_1|^2}{\lambda^2}+\frac{|\partial_t\tilde{r}_2|^2 }{\lambda^4}+\tilde{r}_1^2+\frac{\tilde{r}_2^2}{\lambda^2}\Big)\dd x\dd t \\
					\leq ~ & C\lambda\int_{Q^T}\theta^2z^2\dd x\dd t+C\varepsilon e^{C\lambda}\int_{Q^T}\Big[\theta^2\big(|\nabla\tilde{p}|^2+|\nabla\tilde{p}_t|^2\big)+\theta^{-2}\tilde{w}^2 \Big]\dd x\dd t \\
					&+ C\varepsilon\int_{Q^T}\theta^{-2}|\nabla\tilde{w}|^2\dd x\dd t.
				\end{aligned}
			\end{equation}
			
			\textbf{Step 5.} It follows from \eqref{system:w and p} that
			\begin{equation}\label{4.69 of the book}
				\begin{aligned}
					& ~ \Big(\partial_t\tilde{r}_1+\tilde{r}_2+\lambda\theta^2z+\tilde{r},\theta^{-2}\tilde{w}\Big)_{L^2(Q^T)} \\
					= & ~ \Big(\tilde{w}_{tt}-\sum\limits_{j,k=1}^n\big(a^{jk}\tilde{w}_{x_j}\big)_{x_k},\theta^{-2}\tilde{w}\Big)_{L^2(Q^T)} \\
					= & -\int_{Q^T}\tilde{w}_t(\theta^{-2}\tilde{w})_t\dd x\dd t+\sum\limits_{j,k=1}^n\int_{Q^T}a^{jk}\tilde{w}_{x_j}(\theta^{-2}\tilde{w})_{x_k}\dd x\dd t \\
					= & -\int_{Q^T}\theta^{-2}\big(\tilde{w}_t^2+\lambda\phi_{tt}\tilde{w}^2-2\lambda^2\phi_t^2\tilde{w}^2\big)\dd x\dd t+\sum\limits_{j,k=1}^n\int_{Q^T}\theta^{-2}a^{jk} \tilde{w}_{x_j}\tilde{w}_{x_k}\dd x\dd t \\
					& -2\lambda\sum\limits_{j,k=1}^n\int_{Q^T}\theta^{-2}a^{jk}\tilde{w}_{x_j}\tilde{w}\phi_{x_k}\dd x\dd t,
				\end{aligned}
			\end{equation}
				thus we get
			\begin{equation}\label{4.70 of the book}
				\begin{aligned}
					& \int_{Q^T}\theta^{-2}|\nabla\tilde{w}|^2\dd x\dd t \\
					\leq & ~ C\int_{Q^T}\Big[\theta^{-2}|\partial_t\tilde{r}_1+\tilde{r}_2+\tilde{r}||\tilde{w}|+\lambda|z\tilde{w}|+\theta^{-2}(\tilde{w}_t^2+\lambda^2\tilde{w}^2)\Big]\dd x\dd t \\
					\leq & ~ C\int_{Q^T}\Big[\theta^2z^2+\theta^{-2}\Big(\frac{|\partial_t\tilde{r}_1|^2}{\lambda^2}+\frac{\tilde{r}_2^2}{\lambda^2}+\tilde{r}^2+\tilde{w}_t^2+\lambda^2 \tilde{w}^2\Big)\Big]\dd x\dd t.
				\end{aligned}
			\end{equation}
			
			Now we combine \eqref{4.62 of the book}, \eqref{4.68 of the book} and \eqref{4.70 of the book}, and choose the constant $K$ in \eqref{4.62 of the book} so that
			$$K\geq Ce^{2\lambda\|\phi\|_{L^\infty(Q^T)}}$$
			to absorb the term $C\int_{Q^T}\theta^{-2}\tilde{r}^2\dd x\dd t$ in \eqref{4.70 of the book}. Noting that $\rho(x)\geq 1$ and $\varepsilon$ can be so small that $C\varepsilon e^{C\lambda}\ll 1 $ for given $\lambda$, we finally deduce that
			\begin{equation}\label{4.72 of the book}
				\begin{aligned}
					& \int_{Q^T}\theta^{-2}(|\nabla\tilde{w}|^2+\tilde{w}_t^2+\lambda^2\tilde{w}^2)\dd x\dd t \\
					&+ \int_{Q^T}\rho\theta^{-2}\Big(\frac{|\partial_t\tilde{r}_1|^2}{\lambda^2}+\frac{|\partial_t\tilde{r}_2|^2}{\lambda^4}+\tilde{r}_1^2+\frac{\tilde{r}_2^2}{\lambda^2}\Big) \dd x\dd t \\
					\leq~ & C\lambda\int_{Q^T}\theta^2z^2\dd x\dd t.
				\end{aligned}
			\end{equation}
			
			\textbf{Step 6.} Recall that $(\tilde{w},\tilde{r}_1,\tilde{r}_2,\tilde{r})$ depends on $K$, so we can denote it by
			$$(\tilde{w}^K,\tilde{r}_1^K,\tilde{r}_2^K,\tilde{r}^K).$$
			
			Fix $\lambda$ and let $K\to\infty$, since $\rho=\rho^K(x)\to\infty$ for $x\notin\omega$, we can see from \eqref{4.62 of the book} and \eqref{4.72 of the book} that there exists a subsequence of $(\tilde{w}^K,\tilde{r}_1^K,\tilde{r}_2^K,\tilde{r}^K)$ which converges weakly to some $(\check{w},\check{r}_1,\check{r}_2,0)$ in
			$$H_0^1(Q^T)\times(H^1(0,T;L^2(\Omega))^2\times L^2(Q^T),$$
			with $\text{supp} \,\check{r}_j\subseteq[0,T]\times\overline{\omega}, ~ j=1,2.$ By \eqref{system:w and p} we see that
			$$
			\left\{
			\begin{aligned}
				& \check{w}_{tt}-\sum\limits_{j,k=1}^n\big(a^{jk}\check{w}_{x_j}\big)_{x_k}=\partial_t\check{r}_1+\check{r}_2+\lambda\theta^2z, & (t,x)\in ~ & Q^T, \\
				& \check{w}(0,x)=\check{w}(T,x)=0, & x\in ~ & \Omega, \\
				& \check{w}(t,x)=0, & (t,x)\in ~ & \Gamma_T.
			\end{aligned}
			\right.
			$$
			
			Using \eqref{4.72 of the book} again, we find that
			\begin{equation}\label{4.74 of the book}
				\big\|\theta^{-1}\check{w}\big\|_{H_0^1(Q^T)}^2+\frac{1}{\lambda^2}\int_0^T\int_\omega\theta^{-2}\big(|\partial_t\check{r}_1|^2+\check{r}_2^2\big)\dd x\dd t\leq C\lambda \int_{Q^T}\theta^2z^2\dd x\dd t.
			\end{equation}
				Then we take the $\eta$ in \eqref{define weak solution-B} to be the above $\check{w}$, and find that
                $$\Big(\check{w},\partial_t\check{r}_1+ \check{r}_2+\lambda\theta^2z\Big)_{L^2(Q^T)}=\Big\langle z_{tt}-\sum\limits_{j,k=1}^n\big(a^{jk}z_{x_j}\big)_{x_k},\check{w} \Big\rangle_{H^{-1}(Q^T),H_0^1(Q^T)}.$$
				Hence we have
                \begin{align}\label{4.75 of the book}
			\lambda\int_{Q^T}\theta^2z^2\dd x\dd t= & ~ 
                \Big\langle z_{tt}-\sum\limits_{j,k=1}^n \big(a^{jk}z_{x_j}\big)_{x_k}+2z_t+z,\check{w}\Big\rangle_{H^{-1}(Q^T),H_0^1(Q^T)} \nonumber \\
				&+ 2(z,\check{w}_t)_{L^2(Q^T)}-(z,\check{w})_{L^2(Q^T)}-(z,\partial_t\check{r}_1+\check{r}_2)_{L^2((0,T)\times\omega)} \nonumber \\
				\leq & ~ \Big\|\theta\Big(z_{tt}-\sum\limits_{j,k=1}^n\big(a^{jk}z_{x_j}\big)_{x_k}+2z_t+z\Big)\Big\|_{H^{-1}(Q^T)}\big\| \theta^{-1}\check{w}\big\|_{H_0^1(Q^T)} \nonumber \\
				&+\|\theta z\|_{L^2(Q^T)}\big(\big\|\theta^{-1} \check{w}_t\big\|_{L^2(Q^T)}+\big\|\theta^{-1}\check{w}\big\|_{L^2(Q^T)}\big) \\
				&+ \|\theta z\|_{L^2((0,T)\times\omega)} \big\|\theta^{-1}(\partial_t\check{r}_1+\check{r}_2) \big\|_{L^2((0,T)\times\omega)} \nonumber \\
				\leq & ~ C\sqrt{J}\Big[\big\|\theta^{-1}\check{w}\big\|_{H_0^1(Q^T)}+\lambda\big\|\theta^{-1}\check{w}\big\|_{L^2(Q^T)}+\big\|\theta^{-1}\check{w}_t\big\|_{L^2 (Q^T)} \nonumber \\
				&+ \lambda^{-1}\big\|\theta^{-1}(\partial_t\check{r}_1+\check{r}_2)\big\|_{L^2((0,T)\times\omega)}\Big], \nonumber
				\end{align}
				where
			$$J:=\Big\|\theta\Big(z_{tt}-\sum\limits_{j,k=1}^n\big(a^{jk}z_{x_j}\big)_{x_k}+2z_t+z\Big)\Big\|_{H^{-1}(Q^T)}^2+\lambda^2\int_0^T\int_\omega\theta^2z^2\dd x\dd t.$$
			is exactly the right hand side of \eqref{L2 Carleman estimate}. Since that
			$$\theta^{-1}\check{w}_t=\big(\theta^{-1}\check{w}\big)_t-\big(\theta^{-1}\big)_t\check{w}=\big(\theta^{-1}\check{w}\big)_t+\lambda\phi_t\check{w},$$
            we have
			\begin{equation}\label{extra 4.76 of the book}
				\begin{aligned}
					\big\|\theta^{-1}\check{w}_t\big\|_{L^2(Q^T)}\leq & ~ C\big(\big\|\theta^{-1}\check{w}\big\|_{H^1(0,T;L^2(\Omega))}+\lambda\big\|\theta^{-1}\check{w}\big\|_{L^2(Q^T) }\big) \\
					\leq & ~ C\big(\big\|\theta^{-1}\check{w}\big\|_{H_0^1(Q^T)}+\lambda\big\|\theta^{-1}\check{w}\big\|_{L^2(Q^T)}\big).
				\end{aligned}
			\end{equation}
			
			Finally, by \eqref{4.74 of the book}--\eqref{extra 4.76 of the book}, we obtain the desired estimate \eqref{L2 Carleman estimate}. This completes the proof of Proposition \ref{prop:L2 Carleman theorem}.
		\end{proof}

		\par{\bf Acknowledgements.}
        This work was supported by the National Natural Science Foundation of China (No. 12171097, 12471421, 12001555), Science Foundation of Zhejiang Sci-Tech University (No. 25062122-Y), Key Laboratory of Mathematics for Nonlinear Sciences (Fudan University), Ministry of Education of China, Shanghai Key Laboratory for Contemporary Applied Mathematics, School of Mathematical Sciences, Fudan University and Shanghai Science and Technology Program (No. 21JC1400600, SKLCAM202403002), National Key $R\&D$ Program of China under the grant 2023YFA1010300, Funding by Science and Technology Projects in Guangzhou (No. 2023A04J1335).
		
		\bibliographystyle{plain}
		\def\cprime{$'$}

	\end{document}